\documentclass[12pt]{amsart}
\usepackage[T1]{fontenc}
\usepackage[utf8]{inputenc}

\usepackage{enumitem} 
\usepackage{amsfonts}
\usepackage{mathrsfs}
\usepackage{graphicx}
\usepackage{amssymb, comment}
\usepackage[all]{xy}
\usepackage{tikz,everypage}

\newtheorem{thm}{Theorem}[section]
\newtheorem{lemma}[thm]{Lemma}

\theoremstyle{definition}
\newtheorem{defn}[thm]{Definition}

\theoremstyle{remark}
\newtheorem{remark}[thm]{Remark}

\theoremstyle{remark}

\title[the same boundary at infinity]{One boundary,  infinitely  geometries}

\author{Jiming Ma}
\address{School of Mathematical Sciences, Fudan University, Shanghai, 200433, China}
\email{majiming@fudan.edu.cn}


\date{Oct. 11, 2026}

\thanks{The author was partially supported by   NSFC 12671086.}


\begin{document}

\maketitle

\begin{abstract}
	
	Let $N$ be a 3-dimensional integral homology sphere. We prove the existence of a sequence $\{M_k \}^{\infty}_{k=1}$ of 4-dimensional topologically negatively curved manifolds satisfying the following:
	\begin{itemize}
		\item All aspherical manifolds $M_k$ have the same universal cover $U_N$ depending only on $N$;
		\item All fundamental groups $\{ \pi_{1}(M_k) \}^{\infty}_{k=1}$ are hyperbolic and share the same Gromov boundary $X_N$  depending only on $N$;
		\item The groups $\{ \pi_{1}(M_k) \}^{\infty}_{k=1}$ represent infinitely many distinct quasi-isometry classes.
	\end{itemize}
The last conclusion, which represents the new feature, shows that negatively curved geometries in dimension 4 exhibit substantially richer structure than those in dimension 3.

\end{abstract}

\tableofcontents

\section{Introduction}

As a dramatic extension of hyperbolic  manifolds in topological setting, there is 

\begin{defn} Let $M$ be a closed $n$-manifold, we say $M$ is  \emph{topologically negatively curved} if  $M$ is aspherical and the fundamental group of $M$ is an infinite hyperbolic  group in the sense of Gromov \cite{Gromov:1983}. The quasi-isometric class of the fundamental group of a given   topologically negatively curved manifold $M$ is  called  a \emph{topologically negatively curved geometry}. 
\end{defn}

By Theorem B of \cite{BLW:2010}, for dimension $n \geq 6$, whether a group can be realized as the fundamental group of a topologically negatively curved $n$-manifold is a quasi-isometric invariant. Thus, a topologically negatively curved geometry is well-defined (at least in  dimension $n \geq 6$).
 It is well-known that in 2-dimensional case, there is exactly one quasi-isometric class of topologically negatively curved manifolds, namely that of closed hyperbolic surfaces.
 By Thurston's Geometrization Conjecture (proved by Perelman), there is exactly one quasi-isometric class of topologically negatively curved 3-manifolds, namely that of closed hyperbolic 3-manifolds. However,  in 4-dimensional case, closed manifolds with  real hyperbolic structures  and  closed manifolds with  complex  hyperbolic structures  are in two different  quasi-isometric classes 	\cite{Pansu:1989}.

 There are even more topologically negatively curved geometries in dimension at least 4 with various  finer properties,   see 	\cite{CharneyDavis:1995, DavisJ:1991, LafontR:2024, KrophollerV:2024} for more constructions.
Topological negatively curved manifolds inherit the core  ``negatively curved" geometric intuition from hyperbolic spaces, such as the exponential divergence of geodesics, while breaking free from the strict metric constraints of hyperbolic geometry.
This extension allows mathematicians to explore a broader class of spaces that retain the key topological consequences of hyperbolicity, while also accommodating more flexible structures that are critical for addressing questions about quasi-isometric classification and universal covers.

The main technical result in this paper is

\begin{thm} \label{thm:conformaldim} Let $N$ be a 
	closed orientable 3-manifold. There  is  a sequence of 4-dimensional  hyperbolic   Coxeter groups $\{W_k\}^{\infty}_{k=1}$  with   nerves a sequence of   triangulations $\{\mathcal{T}_k\}^{\infty}_{k=1}$ of $N$, such that  the conformal dimensions of the Gromov boundaries of $W_k$ diverge to the infinity. 
	
\end{thm}

In the proof of Theorem \ref{thm:conformaldim}, the triangulations $\{\mathcal{T}_k\}^{\infty}_{k=1}$ of $N$ are ``almost flag-no-square". By an ``almost flag-no-square" triangulation $\mathcal{T}$ of $N$, we mean that $\mathcal{T}$ is not a flag-no-square triangulation, but it contains only very few empty triangles and empty squares, which are moreover very sparse  in $\mathcal{T}$; see Theorem \ref{thm:almostfns} for the construction of these triangulations.

From  Theorem \ref{thm:conformaldim}, we have the following theorem, which is the main motivation of the paper. 

\begin{thm} \label{thm:infiniteqi}

	 Let $N$ be a 3-dimensional integral homology sphere. There exists a sequence of 4-dimensional topologically negatively curved manifolds $\{M_k \}^{\infty}_{k=1}$, such that:

	\begin{itemize}
		
		\item  (Same universal cover)  All of the aspherical manifolds  $\{M_k \}^{\infty}_{k=1}$ have the same universal cover $U = U_N$ that depends on $N$. $U_N$ is  not homeomorphic to $\mathbb{R}^4$ if and only if  $N$ is  not homeomorphic to 3-sphere  $\mathbb{S}^3$;
		
		\item (Same boundary)  All of the hyperbolic groups $\{ \pi_{1}(M_k) \}^{\infty}_{k=1}$ have the same boundary at infinity, namely the tree of manifolds  $\chi(N \sharp \overline{N})$, which is  not  a 3-manifold  if and only if  $N$ is  not homeomorphic to 3-sphere  $\mathbb{S}^3$;

				\item (Hierarchies of subgroups)  There are quasi-convex,  hyperbolic subgroups  $H_k < K_k < \pi_{1}(M_k)$, with     the boundary of $H_k$  being the Menger curve and   the boundary of $K_k$  being the Pontryagin sphere for each $k$;

					\item (Infinite  geometries)  The conformal dimensions of the Gromov boundaries of groups $\{\pi_{1}(M_k) \}^{\infty}_{k=1}$ diverge to the infinity. In particular, there are infinitely many  quasi-isometric classes in  $\{ \pi_{1}(M_k) \}^{\infty}_{k=1}$.
			
		\end{itemize}
	
\end{thm}

For various 3-dimensional integral homology spheres $N$, there exist 4-dimensional negatively curved manifolds with infinitely many distinct topological types of universal covers $U_N$ associated to $N$ \cite{Davis:1983}. Additionally, for various such $N$, the fundamental groups of these 4-manifolds admit infinitely many topological types of boundaries  $\chi(N \sharp \overline{N})$ at infinity \cite{PrzytyckiS:2009, Swiątkowski:2020}.
These two phenomena  imply the existence of infinitely many quasi-isometric classes among such manifolds.

The new feature of Theorem \ref{thm:infiniteqi} is the revelation that, for varying values of $N$, 4-dimensional negatively curved manifolds exhibit three-fold infinite topological diversity: infinitely many distinct topological types of universal covers, infinitely many topological types of boundaries at infinity, and consequently, infinitely many quasi-isometric classes.
It is plausible to believe that there are infinitely many quasi-isometric classes of topological negatively curved manifolds sharing the same universal cover; however, a rigorous proof has long been lacking. Even for the case where the universal cover is the four-dimensional Euclidean space and the boundary is the three-dimensional sphere, the existence of infinitely many quasi-isometric classes of such topological negatively curved manifolds has been rigorously proven for the first time. This should be compared  with \cite{SistoViaggi:2026}, where Sisto and Viaggi  develops criteria for distinguishing the homeomorphism types of Gromov-Thurston manifolds, but does not address their quasi-isometric classes; these manifolds are certain negatively curved Riemannian $n$-manifolds.
Theorem  \ref{thm:infiniteqi} implies that  negatively curved geometries in dimension at least 4 are more complicated than  dimension  3.  
Compared with the use of special subdivisions of 3-complexes to study limit sets of higher-dimensional Kleinian groups \cite{DLMR:2026}, the relatively special subdivision in Theorem \ref{thm:almostfns} appears to be independently interesting.

\textbf{Outline of the proofs.} First, there exists a flag-no-square graph~$\Gamma$ such that the conformal dimension of the hyperbolic right-angled Coxeter group~$W_{\Gamma}$ associated to~$\Gamma$ can be arbitrarily large	\cite{CDSStark:2025}. Second, for a 3-dimensional integral homology sphere $N$,  we tend to take  a triangulation $\mathcal{T}$ of $N$ which is flag-no-square and contains $\Gamma$ as a full subgraph.  Then the  right-angled hyperbolic  Coxeter   group  $W_{\Gamma}$  with nerve $\Gamma$ is a quasi-convex subgroup of the  right-angled hyperbolic  Coxeter   group  $W_{\mathcal{T}}$  with nerve $\mathcal{T}$. In turn, the conformal dimension  of the boundary of the   group $W_{\mathcal{T}}$  is bigger than  the conformal dimension of the boundary of the  group $W_{\Gamma}$.
 
 Although constructing such a triangulation~$\mathcal{T}$  of $N$ is difficult, it is still feasible via the following approach. Specifically, we construct an ``almost flag-no-square'' triangulation~$\mathcal{T}$ of~$N$: it is neither flag nor square-free, but we retain control over its empty triangles and squares. Unlike the special subdivision obtained from an initial triangulation~$\mathcal{T}'$ in~\cite{PrzytyckiS:2009}, our~$\mathcal{T}$ is a relatively special subdivision associated with the pair~$(\mathcal{T}',\Gamma)$. This triangulation ~$\mathcal{T}$ has two key features:
 \begin{itemize}
 	\item From the triangulation~$\mathcal{T}$, we obtain a hyperbolic ``almost right-angled'' Coxeter group~$W_{\mathcal{T}}$, which is not purely right-angled: we assign to some of the edges in~$\mathcal{T}$ the angle~$\pi/p$ for $p=3,4,5$. The control over empty triangles and squares in~$\mathcal{T}$ ensures that the nerve of the Coxeter group~$W_{\mathcal{T}}$ is precisely the triangulation~$\mathcal{T}$ of~$N$;
 	\item $\mathcal{T}$ contains~$\Gamma$ as a subgraph, and we assign to each edge of~$\Gamma$ the angle~$\pi/2$. Hence $W_{\mathcal{T}}$ contains the right-angled Coxeter group~$W_{\Gamma}$ as a quasi-convex subgroup.
 \end{itemize}
 Then we obtain the technical Theorem~\ref{thm:conformaldim}.
 In fact, more than 99.9 percent
  of the angles in the Coxeter system are $\pi/2$, while a small number of angles are strictly smaller. 
  Although dropping the right-angled condition is indispensable to our proofs, we believe that using the flag-no-square condition to prove Theorem \ref{thm:conformaldim} remains of independent interest.

We remark that, although the preparations in Subsections~\ref{section:graph} and \ref{section:almostfns3ball} make Theorem~\ref{thm:conformaldim} look easy to prove at first glance, a fully rigorous argument is considerably more involved. Indeed, one must handle several layers of subcases, as demonstrated in Section~\ref{section:almostfns}. The main difficulty lies in counting all the empty squares in the triangulation \(\mathcal{T}^*\) in Theorem~\ref{thm:almostfns}, a task that is quite intricate.

\textbf{Outline of the paper.} 
In Section \ref{section:Coxeter}, we give some preliminary results on Coxeter groups. In Section  \ref{section:graph}
we embed any given graph~$\Gamma$ into a triangulation~$\mathcal{T}$ of a $3$-manifold~$N$ in a controlled manner, this triangulation will be modified later in the proof Theorem  \ref{thm:almostfns}.
 We then  prepare some 3-complexes which are triangulations of 3-balls in Section  \ref{section:almostfns3ball}.
 In Section  \ref{section:almostfns}, we transform the triangulation $\mathcal{T}$ in Lemma 
 	 \ref{lemma:inducedgraph}  via the 3-complexes in  Section  \ref{section:almostfns3ball} to
 	  an ``almost flag-no-square" triangulation $\mathcal{T}^*$ of $N$. 
 	   Assigning angles on edges of  $\mathcal{T}^*$, we finally obtain Theorem \ref{thm:conformaldim} in  Section \ref{section:conformaldim}.
 In Section \ref{section:infiniteqi}, we prove Theorem \ref{thm:infiniteqi}.

   For the conciseness of this paper, we omit the definitions of hyperbolic groups, quasi-isometries, Gromov boundaries, and  conformal dimensions, etc; the reader may refer to 	\cite{Gromov:1983,	CDSStark:2025} for these concepts.

\textbf{Notations.} In this paper, 
\begin{itemize}
		\item
\(A_i\) and \(B_j\) denote the vertices of a bipartite graph \(\Gamma\);
	\item  The symbols \(A\), \(B\), \(C\), and \(D\) denote special vertices of the \(3\)-complex \(\mathcal{C}'_{543}\);
	
	\item  The symbols \(a_i\), \(v_i\), and \(w_{i,jk}\) denote special vertices of the \(3\)-complex \(\mathcal{C}_{600}\);
	
		\item  The symbols \(F_i\), for \(1 \leq i \leq 120\), denote facets of the 120-cell \(\mathcal{Y}_{120}\).
	
	\item The symbols \(G\), \(I\), \(E\), \(F\), \(H\),  \(C_i\), and \(D_j\) denote general vertices of a \(3\)-complex \(\mathcal{T}\).
	
	\end{itemize}

\textbf{AI Declaration.}	The author used an AI tool  only for English language editing and grammar correction. No mathematical arguments, proofs, or computations were generated or checked by AI. All mathematical content and results are the author's own.

\section{Coxeter groups}  \label{section:Coxeter}

We give some preliminary results of Coxeter groups we will use in this section.


In this paper, by a triangulation we always mean a simplicial triangulation.
Recall that a triangulation $\mathcal{T}$  of an $m$-dimensional complex $X$ is \emph{flag} if, for any set of vertices $\{D_1,D_2, \cdots, D_{k+1}\}$ in $\mathcal{T}$  such that $\{D_i,D_j\}$ spans an edge in $\mathcal{T}$  for all $1 \leq i<j \leq k+1$, the set $\{D_1,D_2, \cdots, D_{k+1}\}$ spans a $k$-simplex in  $\mathcal{T}$.
If a set of vertices $\{D_1,D_2, \cdots, D_{k+1}\}$ in  $\mathcal{T}$  has the property that every $k$-element subset spans a $(k-1)$-simplex in $\mathcal{T}$, but $\{D_1,D_2, \cdots, D_{k+1}\}$ itself does not span a $k$-simplex in $\mathcal{T}$, then $\{D_1,D_2, \cdots, D_{k+1}\}$ is called an \emph{empty $k$-simplex} in $\mathcal{T}$.
A 4-circle in a simplicial complex $X$ is an \emph{empty square} where neither of the diagonals is an edge  in the complex. A triangulation   $\mathcal{T}$ of an $m$-dimensional  complex $X$ is    \emph{flag-no-square} if it is flag and contains no empty squares. For a simplicial triangulation $\mathcal{T}$ of  a complex $X$, if there is an edge  connecting vertices  $D_i$ and $D_j$ in $ \mathcal{T}^{(0)}$, we sometimes denote  this edge by $[D_i,D_j]$. Recall that  for a simplicial complex $\mathcal{K}$, a subcomplex $\mathcal{L}$  of
$\mathcal{K}$   is \emph{full} if every set $\{D_0, D_1, \dots, D_n\} \subset \mathcal{L}^{(0)}$ that spans an $n$-simplex in $\mathcal{K}$ satisfies that this  $n$-simplex also lies in $\mathcal{L}$.

Given a triangulation $\mathcal{T}$ of a complex $X$, and a set of angles $$\mathcal{A}=\{\alpha_{e}\}_{e \in \mathcal{T}^{(1)}},$$ where $\alpha_{e}= \pi/m_e$ with some $m_e \in \mathbb{Z}_{\geq 2}$ for an edge $e$ of  $\mathcal{T}$. We can associate to this data a  \emph{Coxeter group} $W= W_{(\mathcal{T}, \; \mathcal{A})}$, defined by the following rules:
\begin{enumerate}
	\item For each vertex $D_i$ in $\mathcal{T}$, there corresponds  an involution; for simplicity, we reuse the vertex label $D_i$ to denote this involution;

	\item If two vertices  $D_i$ and $D_j$ in $\mathcal{T}$ are connected  by an edge $e$,  then the group $W_\mathcal{T}$ satisfies  the relation $(D_i\cdot D_j)^{m_e}={\rm id}$.
\end{enumerate}

The definition of a Coxeter group only involves the $1$-skeleton $\mathcal{T}^{(1)}$ of  $\mathcal{T}$, not the higher-dimensional simplices. However, we assume the following condition: for any subset $S=\{D_1, D_2, \dots, D_{k+1}\}$ of the vertex set $\mathcal{T}^{(0)}$, the corresponding subgroup $W_S$ of $ W_{(\mathcal{T}, \; \mathcal{A})}$ is finite  if and only if  $S$ spans a $k$-simplex in $\mathcal{T}$.
Under these conditions, the triangulation $\mathcal{T}$ is also called the \emph{nerve} of the Coxeter group $W_{(\mathcal{T},  \;\mathcal{A})}$. 
In particular, if  we assign the angle $\pi/2$ to each edge $e$ of  $\mathcal{T}^{(1)}$, the resulting group is a \emph{right-angled Coxeter group}, denoted by $W= W_{\mathcal{T}}$.

Section 6 of  \cite{CDSStark:2025} gives many examples of graphs as in the  following Theorem \ref{thm:highconformaldim}, which is one of the  keys in our proofs of Theorems \ref{thm:conformaldim} and  \ref{thm:infiniteqi}:

\begin{thm} \label{thm:highconformaldim} (Cashen-Dani-Schreve-Stark \cite{CDSStark:2025}) For any $k \in \mathbb{Z}_+$, there is a bipartite flag-no-square graph $\Gamma=\Gamma_k$ such that   the   conformal dimension of the boundary of $W_{\Gamma}$ is at least $k$, where  $W_{\Gamma}$ is the  hyperbolic right-angled Coxeter group with nerve  $\Gamma$.
\end{thm}

\section{Embeddings of graphs in triangulations of 3-manifolds}  \label{section:graph}


In this section, we embed any flag-no-square graph~$\Gamma$ into a triangulation~$\mathcal{T}$ of a closed orientable $3$-manifold~$N$ in a controlled manner. 
Lemma~\ref{lemma:inducedgraph} gives such an embedding, ensuring that $\Gamma$ is sparse in~$\mathcal{T}$. 
The three sparsity conditions may seem unnatural at first glance, but they are precisely what is needed for Theorem~\ref{thm:almostfns}. 
In the relatively special subdivision  of Theorem \ref{thm:almostfns}, the graph $\Gamma$ acts as a ``poison pill" for the no-square condition, that is, a structure that would otherwise invalidate the theorem's conclusions if left unconstrained.
Hence we impose the conditions of Lemma~\ref{lemma:inducedgraph} on~$\Gamma$.


\begin{lemma} \label{lemma:inducedgraph} Let $N$  be a 	closed orientable 3-manifold, and   $\Gamma$ be a flag-no-square,  bipartite graph $\Gamma$. The vertices of the graph~$\Gamma$ are $\{A_i\}_{i=1}^{m}$ and $\{B_j\}_{j=1}^{n}$. Then there exists a triangulation $\mathcal{T}$ of $N$, and   a genus $g \geq 1$  orientable surface $\Sigma_g$ embedded in $N$ such that:
	
		\begin{enumerate}
		
		\item \label{inducedgraph:1}  $\Sigma_g \hookrightarrow N$ as a full   subcomplex   $t$ of $\mathcal{T}$;
		
		\item \label{inducedgraph:2}  $\Gamma \hookrightarrow \Sigma_g$ as a full subgraph of  $t$;

			\item  \label{inducedgraph:3} For each 3-simplex $\Delta$ in $\mathcal{T}$, if $\Gamma \cap \Delta$  is non-empty, then it  is an edge or a vertex;

			\item \label{inducedgraph:4} (Sparsity I) For each  vertex $B_j$ of  $\Gamma$, suppose there exists vertices $A_1,  A_2, \dots, A_k$ in $\Gamma^{(0)}$ (the  $A_i$ need not be distinct), and pairwise distinct vertices $D_1,  D_2, \dots, D_k$ in $\mathcal{T}^{(0)} \setminus \Gamma^{(0)}$, with the convention $D_{k+1}=D_1$,  and 3-simplices  $\Delta_i$ whose  vertex set is  $\{B_j, A_i,D_i,D_{i+1}\}$ for every $1 \leq i \leq k$.    Then $k \geq  5$;

			\item \label{inducedgraph:5}  (Sparsity II) 	 If the tetrahedron $\Delta \in \mathcal{T}^{(3)}$ with  vertex set $\{B_j,G,H,F\}$ intersects $\Gamma$ in the vertex   $B_j$. 	Then     at most  one 2-face of $\Delta$, say the triangle with vertex set  $\{B_j, H,F\}$,  is adjacent to a  tetrahedron $\Delta' \in \mathcal{T}^{(3)}$ with vertex set $\{B_j,A_i,H,F\}$, such that the edge $[A_i,B_j]$ lies in $\Gamma$;
			
				\item \label{inducedgraph:6} (Sparsity III) There is no adjacent triangles in $\mathcal{T}^{(2)}$ with vertex sets $\{A_i,B_j, I\}$ and $\{A_k,B_j, I\}$ such that both $[A_i,B_j]$ and $[A_{k},B_j]$ lie in $\Gamma$; Similarly, there are no adjacent triangles in $\mathcal{T}^{(2)}$ with vertex sets $\{B_j,A_i, I\}$ and $\{B_k,A_i,I\}$ such that both $[B_j,A_i]$ and $[B_k,A_i]$ lie in $\Gamma$.

			\end{enumerate}

\end{lemma}

\begin{proof}	We denote the vertices of the bipartite graph~$\Gamma$ by $\{A_i\}_{i=1}^{m}$ and $\{B_j\}_{j=1}^{n}$. Thus, every edge of~$\Gamma$ joins some $A_i$ to some $B_j$; there are no edges among the $A_i$'s, nor among the $B_j$'s. Since $\Gamma$ is a full subgraph of~$\mathcal{T}$, the same edge relations hold in~$\mathcal{T}$.


	We take a  Heegaard splitting $N=H_{+} \cup_{\Sigma_g} H_{-}$ of $N$  with genus $g \geq 1$. Namely $H_{+}$ and $ H_{-}$ are handlebodies with boundary $\Sigma_g$. It is  straightforward to construct a simplicial triangulation $t$ of $\Sigma_g$ such that 	$\Gamma \hookrightarrow \Sigma_g$ is a full subgraph;  for instance, refer to	\cite{CDSStark:2025}. Then we take a simplicial triangulation $\mathcal{T}_{+}$ of $H_{+}$, which induce a  simplicial triangulation of $\partial H_{+}=\Sigma_g$, say $t_+$. 
		It is well-known that any  two triangulations of the surface $\Sigma_g$ are connected by a sequences of   bistellar moves: namely,   2-2 move, 1-3 move and 3-1 move, refer to 
		\cite{Lickorish:1999}.

		We now modify the triangulation $\mathcal{T}_+$ of $H_+$ according to  the sequence of bistellar moves to align it with the triangulation $t$ of $\Sigma_g$:
			\begin{itemize}
			
			\item  Denote $t_1=t_+$, and    assume $t_1$ changes to  a triangulation  $t_2$ of $\Sigma_g$  by a  2-2 move; that is, we assume there are two triangles in $t_1$ with vertex sets $\{a,b,d\}$ and $\{b,c,d\}$,  two triangles in $t_2$ with vertex sets $\{a,b,c\}$ and $\{a,c,d\}$. We add a tetrahedron to  the original triangulation $\mathcal{T}_{+}$  with vertex set $\{a,b,c,d\}$. Topologically, this can be viewed as a boundary connected sum of $H_{+}$ with a 3-ball, the resulting space is also $H_{+}$, we get  a new triangulation of $H_{+}$. We still denote  the new  triangulation of $H_{+}$  by $\mathcal{T}_{+}$,   which induces the triangulation $t_2$ on  $\partial H_+=\Sigma_g$;
			
					\item If $t_1=t_+$ changes to    $t_2$  by a  1-3 move;  that is, we assume there is a  triangle in $t_1$ with vertex set $\{a,b,c\}$, and  there are  three  triangles in $t_2$ with vertex sets $\{a,b,d\}$, $\{b,c,d\}$ and $\{c,a,d\}$.  We add a tetrahedron to the original triangulation $\mathcal{T}_{+}$  with vertex set $\{a,b,c,d\}$, getting a new triangulation of $H_{+}$. We still denote  the new  triangulation of $H_{+}$  by $\mathcal{T}_{+}$,  which induces the  triangulation $t_2$ on $\partial H_+=\Sigma_g$;

							\item The case that    $t_1=t_+$ changes to   $t_2$  by a  3-1 move is similar.
							
				\end{itemize}
				 After finite many steps as above,  we have  a triangulation  $\mathcal{T}_{+}$ of $H_{+}$, such that the induced  triangulation of $\partial H_+=\Sigma_g$ is exactly the triangulation  $t$. Similarly,  we have  a triangulation  $\mathcal{T}_{-}$ of $H_{-}$, such that the induced  triangulation of $\Sigma_g$ is also  exactly the triangulation $t$. The two triangulations fit together to give a triangulation $\mathcal{T}$ of $N$.

		  Now the triangulation $t$ does not need to be a full subcomplex of $\mathcal{T}$:
		  	\begin{itemize}

		  		\item If  there is an edge $[G,I]$ in $\mathcal{T}$ with $G, I$ lie in $t$, but  $[G,I]$ does not lie in $t$. Then we add a new vertex $F$ in the middle of $[G,I]$, subdivide every 3-simplex $\Delta$ in $\mathcal{T}$ with $[G,I]$ as one edge into two smaller 3-simplices; 
		  	
		  	 		\item If  a triangle $\delta$  with vertex  $\{G,I,F\}$ in $\mathcal{T}$,  such that  $[G, I]$, $[I,F]$ and $[F, G]$ lie in $t$, but  $\delta$  does not lie in $t$. Then we add a new vertex $E$ in the center  of $\delta$, subdivide every 3-simplex $\Delta$ in $\mathcal{T}$ with $\delta$ as one face into three smaller 3-simplices. 
		  		\end{itemize}	
	  			After applying the above procedure finitely many times, we obtain a triangulation~$\mathcal{T}$ that satisfies the first condition of Lemma~\ref{lemma:inducedgraph}.
	  		Note that if $\Gamma$ is a full subgraph of a triangulation $t$ of $\Sigma_g$, and if we subdivide the triangulation~$t$ of~$\Sigma_g$ as described above to obtain a new triangulation~$t'$, then $\Gamma$ remains a full subgraph of~$t'$. Consequently, the second condition of Lemma~\ref{lemma:inducedgraph} follows as well.

By conditions (\ref{inducedgraph:1}) and (\ref{inducedgraph:2}), $\Gamma$ is a full subgraph of $\mathcal{T}$. Hence, for each 3-simplex $\Delta$ in $\mathcal{T}$, if $\Gamma \cap \Delta$ is nonempty, then it is either an edge or a vertex.

For  Sparsity I, suppose there exists a   vertex $B_j$ of  $\Gamma$, and a set of  3-simplices  $\Delta_i$ whose  vertex set is  $\{B_j, A_i,D_i,D_{i+1}\}$ for every $1 \leq i \leq k$, where $D_{k+1}=D_1$. Since $\mathcal{T}$ is a simplicial triangulation,  we have  $k=3$ or $4$. 
The union of the 2-simplices with vertex sets $\{B_j,D_i,D_{i+1}\}$  is a  \emph{length-$k$ poisonous circle}.   We then  add a  central vertex $E_i$ in the interior of the triangle  $\{B_i,D_i,D_{i+1}\}$, and subdivide $\Delta_i$ into three smaller tetrahedra with vertex sets $\{A_i,B_j,D_i,E_i\}$, $\{A_i,B_j,E_i,D_{i+1}\}$, and $\{A_i,E_i,D_i,D_{i+1}\}$,  respectively. The adjacent  tetrahedron $\Delta'_i \in \mathcal{T}^{(3)}$ to  $\Delta_i$ is also subdivided into three smaller tetrahedra accordingly. This yields a new triangulation of $N$: we eliminate a poisonous circle from $\mathcal{T}$, and this procedure does not produce any new poisonous circles. Moreover, this procedure does not affect any of the conditions (\ref{inducedgraph:1}), (\ref{inducedgraph:2}), and (\ref{inducedgraph:3}).    After finitely many such operations, we may assume the triangulation  $\mathcal{T}$ satisfies   Sparsity I.


For  Sparsity II, if the tetrahedron $\Delta \in \mathcal{T}^{(3)}$ with vertex set  $\{B_j,G,H,F\}$ intersects $\Gamma$ in a point $B_j$. Moreover, there are at least  two tetrahedra   $\Delta_1, \Delta_2 \in \mathcal{T}^{(3)}$ with vertex  $\{B_j,G,H,A_1\}$ and  $\{B_j,F,H,A_2\}$ such that both edges $[A_1,B_j]$ and $[A_2,B_j]$ lie in $\Gamma$. The tetrahedron   $\Delta$ is a \emph{poisonous tetrahedron}  in the  triangulation $\mathcal{T}$.  
 We may  add a vertex $J$ in the interior of  $\Delta$, and subdivide  $\Delta$ into four smaller 
tetrahedra, getting a new  triangulation $\mathcal{T}'$ of $N$. The triangulation  $\mathcal{T}'$ still  satisfies the first four  conditions in Lemma \ref{lemma:inducedgraph}. 
Now each of the four smaller  
tetrahedra is not a poisonous tetrahedra in $\mathcal{T}'$. After finite steps,  the triangulation satisfies 
 Sparsity II.

For   Sparsity III, suppose  there  are  adjacent triangles $\delta_1$ and $ \delta_2$ in $\mathcal{T}^{(2)}$ with vertex sets $\{A_i,B_j, I\}$ and $\{A_k, B_j, I\}$ such that both $[A_i, B_j]$ and $[A_k, B_j]$  lie in $\Gamma$. The pair   $\{\delta_1,  \delta_2 \}$ is a pair of  \emph{poisonous triangles}.  Suppose $\delta_1$ is the common 2-face of two tetrahedra $\Delta_1, \Delta_2$ with vertex sets $\{A_i, B_j, I, E\}$ and $\{A_k, B_j, I, F\}$,  $\delta_2$ is the common 2-face of two tetrahedra $\Delta_3, \Delta_4$ with vertex sets $\{A_i, B_j, I, G\}$ and $\{A_k,B_j, I, H\}$. We may add a vertex $L$ in the interior of $\delta_1$, and a vertex $J$ in the interior of $\delta_2$, subdivide each of $\delta_1$ and $\delta_2$  into three smaller triangles. In turn, we may subdivide each of $\Delta_1, \Delta_2, \Delta_3 $ and $ \Delta_4$ into three smaller tetrahedra according the subdivisions of $\delta_1$ and $\delta_2$. This procedure gives a new triangulation $\mathcal{T}'$ of $N$, it still satisfies conditions in Lemma \ref{lemma:inducedgraph}, except  possibly  the condition  Sparsity III.  
But in this procedure, we kill off a pair of poisonous triangles, and we do not add any more poisonous triangles (and do not add any more poisonous tetrahedra). Thus, after finitely many steps, we obtain a triangulation as in Lemma \ref{lemma:inducedgraph}.

	 \end{proof}

	 \section{Almost flag-no-square triangulations of a 3-ball}   \label{section:almostfns3ball}

In this section, we prepare some $3$-complexes, say $\mathcal{C}_{543}$, $\mathcal{C}'_{543}$, $\mathcal{C}_{1137}$, and $\mathcal{C}_{1163}$, which are  ``almost flag-no-square"  triangulations of a 3-ball.  
In Section~\ref{section:almostfns}, we use these $3$-complexes to transform the triangulation~$\mathcal{T}$ from Lemma~\ref{lemma:inducedgraph} into an ``almost flag-no-square'' triangulation~$\mathcal{T}^*$ of~$N$, as stated in Theorem~\ref{thm:almostfns}. Although the constructions of these $3$-complexes are not difficult, their finer properties are delicate to establish; as a result, this section is the longest in the present paper.



	
	\textbf{The 3-complex $\mathcal{C}_{600}$.}	
		Consider the 600-cell \cite{PrzytyckiS:2009, SchleimerSegerman:15, Ma-Zheng:2021}, which  is the convex regular 4-polytope with Schl\"afli symbol $\{3,3,5\}$. The 120-cell is the convex regular 4-polytope with Schl\"afli symbol $\{5,3,3\}$. 
		Denote by $\mathcal{C}_{600}$ the boundary of the 600-cell, and  by $\mathcal{Y}_{120}$ the boundary of the 120-cell. Then  $\mathcal{C}_{600}$ is a triangulation of the 3-sphere, with  $\mathcal{Y}_{120}$  as the dual.  $\mathcal{Y}_{120}$ consists of 120 icosahedra, 720 pentagons, 1200 edges  and 600 vertices.  
		 $\mathcal{C}_{600}$ consists of 600 tetrahedrons, 1200 triangles, 720  edges and 120 vertices: 
		\begin{itemize}
		\item   The vertex links of $\mathcal{C}_{600}$  are icosahedra; 
			\item The edge links of  $\mathcal{C}_{600}$  are pentagons;
			\item Every vertex in $\mathcal{C}_{600}$ is a vertex of 20 tetrahedra.
				\end{itemize}

		See Figure 	\ref{figure:120} for the boundary of the 120-cell. 	It is well-known that the 120-cell can be realized as  a  compact right-angled hyperbolic polytope in $\mathbb{H}^{4}$, in turn this implies  $\mathcal{C}_{600}$ is a   flag-no-square triangulation of the 3-sphere.  $\mathcal{C}_{600}$ is important in our proof of Theorems \ref{thm:conformaldim}   and \ref{thm:infiniteqi}, so we recall some feathers of $\mathcal{C}_{600}$ from \cite{PrzytyckiS:2009}. 
			
					\begin{figure}
				\begin{center}
					\begin{tikzpicture}
					\node at (0,0) {\includegraphics[width=7.5cm,height=7.5cm]{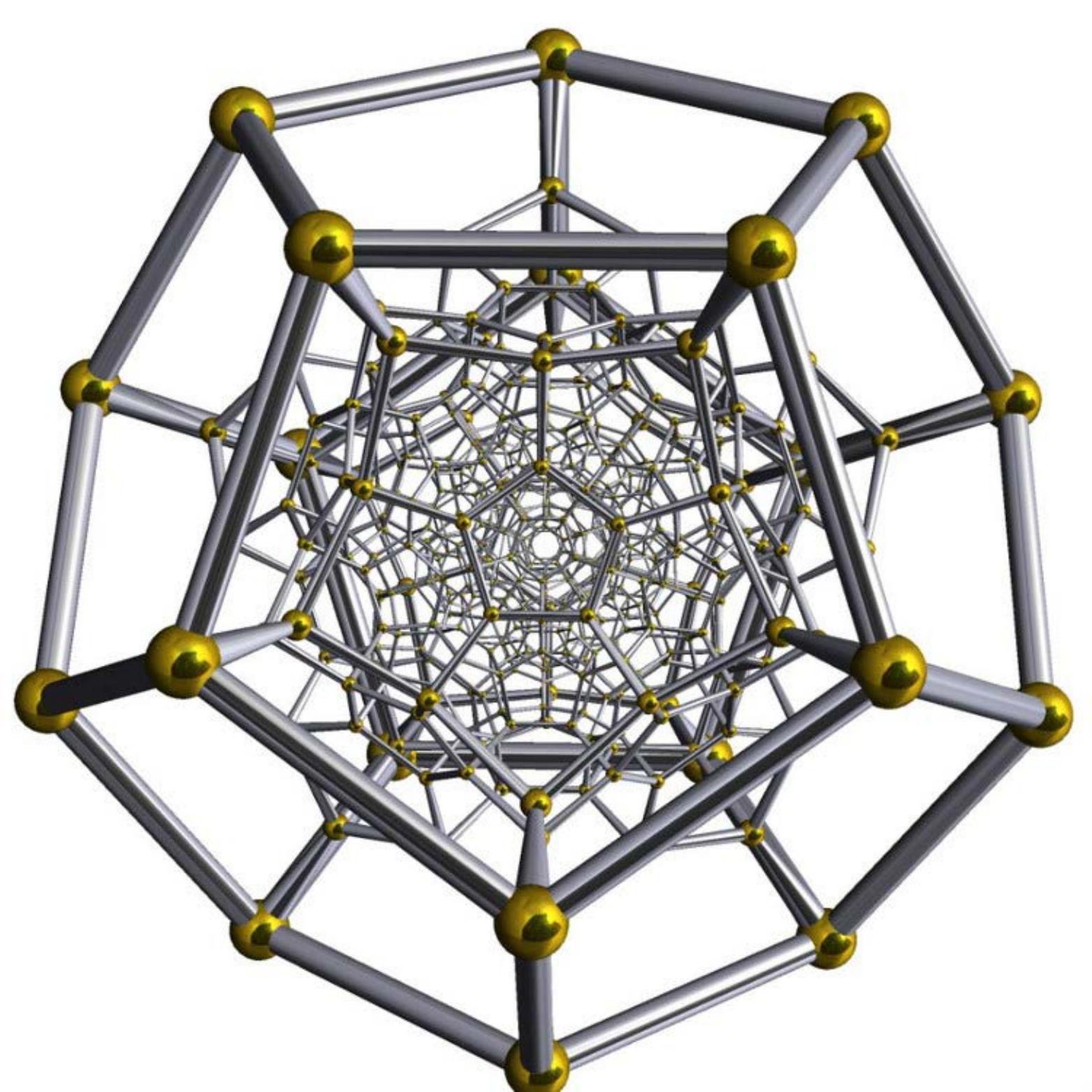}};

					\end{tikzpicture}
				\end{center}
				\caption{Facets boundary  of the 120-cell; figure taken from the web. }
				\label{figure:120}
			\end{figure}

				\subsection{The 3-complex $\mathcal{C}_{543}$}  \label{subsection:c543} We now describe a $3$-complex, denoted by $\mathcal{C}_{543}$, which is a triangulation of a $3$-ball \cite{PrzytyckiS:2009}.

		\textbf{The 3-complex $\mathcal{C}_{17}$.} 
		Consider a tetrahedron $\Delta_0$ with four vertices $\{v_1, v_2, v_3, v_4\}$ lying in the 3-sphere. We add four vertices $a_{i}$ for $i  \in \{1,2,3,4\}$ outside $\Delta_0$.   
			Glue the tetrahedron $\Delta_1$ with  vertex set  $\{a_1, v_2, v_3, v_4\}$ to  $\Delta_0$ along the triangle with vertex set $\{v_2, v_3, v_4\}$. Similarly, glue the tetrahedra  $\Delta_i$  to  $\Delta_0$ such that $a_i$ is a vertex of $\Delta_i$  for $i=2,3,4$.	
		We then add six vertices $a_{ij}=a_{ji}$ for $i \neq j \in \{1,2,3,4\}$ outside $\cup^{5}_{i=0} \Delta_{i}$. Now glue twelve tetrahedra $\Delta_{ij}$ for $i \neq j \in \{1,2,3,4\}$ to $\cup^{5}_{i=0} \Delta_{i}$. For example, $\Delta_{12}$ has vertices $\{v_3,v_4,a_1,  a_{12}\}$, and  the tetrahedron  $\Delta_{21}$ has vertices $\{v_3,v_4, a_2, a_{21}\}$. In particular, $\Delta_{ij}$ and  $\Delta_{ji}$ share the triangle with vertex  set  $\{v_k,v_l,  a_{ij}\}$ where $\{i,j,k,l\}=\{1,2,3,4\}$. The union of these seventeen tetrahedra forms a triangulation of a 3-ball; we denote this simplicial complex by  $\mathcal{C}_{17}$. The boundary of  $\mathcal{C}_{17}$ is shown in Figure \ref{figure:C17}, which is a triangulation of the 2-sphere.

	\begin{figure}
		\begin{center}
			\begin{tikzpicture}
			\node at (0,0) {\includegraphics[width=11cm,height=9cm]{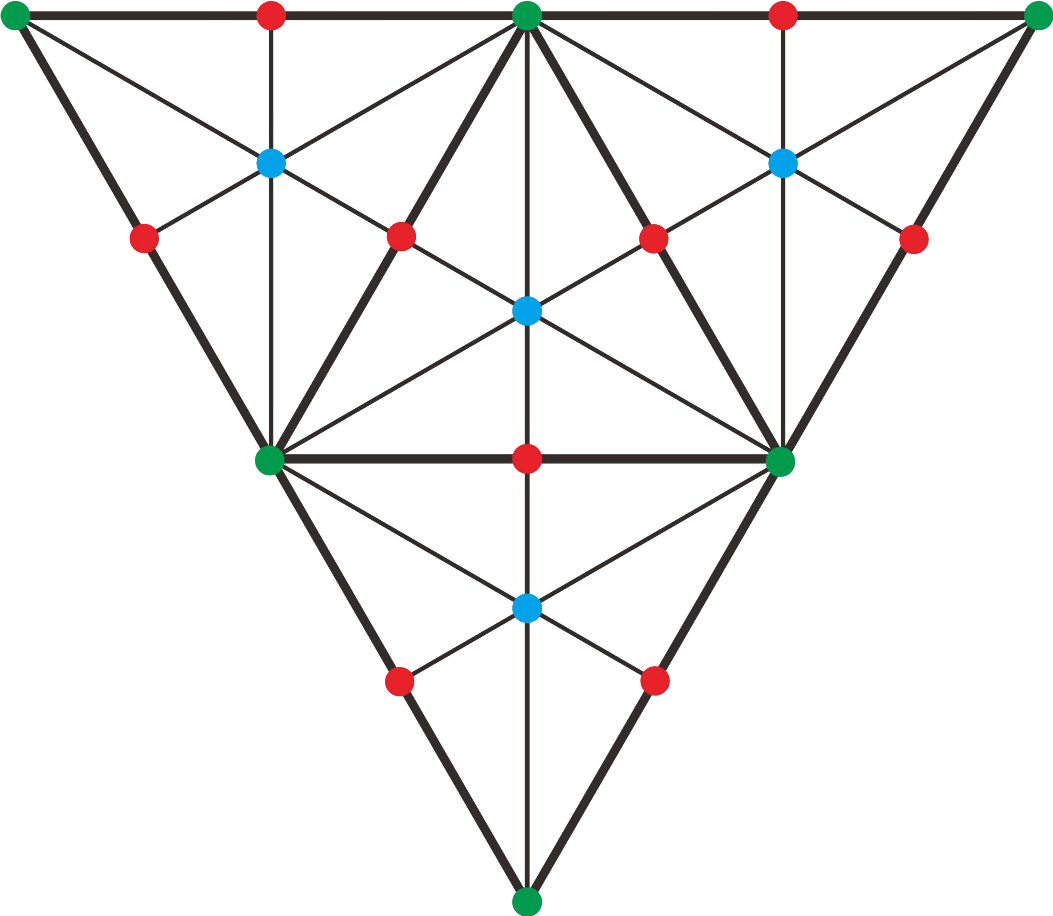}};

			\node at (0.6,1.4){ $v_{1}$};
			
			\node at (0.6,-1.5){ $v_{2}$};
			
			\node at (3.4,2.9){ $v_{3}$};
			
			\node at (-3.4,2.9){ $v_{4}$};

			\node at (0.1,4.9){\Large $a_{2}$};
			\node at (-3.2,-0.2){\Large $a_{3}$};
			\node at (3.2,-0.2){\Large $a_{4}$};
			
			\node at (5.3,4.9){\Large $a_{1}$};
			\node at (-5.3,4.9){\Large $a_{1}$};
			\node at (0.5,-4.7){\Large $a_{1}$};
			
			\node at (-3.0,4.9){\Large $a_{12}$};
			\node at (3.0,4.9){\Large $a_{12}$};

			\node at (4.8,2.0){\Large $a_{14}$};
			\node at (1.85,-2.4){\Large $a_{14}$};
			
			\node at (-4.8,2.0){\Large $a_{13}$};
			\node at (-1.85,-2.4){\Large $a_{13}$};

			\end{tikzpicture}
		\end{center}
		\caption{By gluing edges in this figure that have identical pairs of labeled vertices,  we obtain a triangulation of the 2-sphere, which is the  boundary of the 3-complex $\mathcal{C}_{17}$. We omit the labels of vertices $a_{ij}$ for $i \neq j \in \{2,3,4\}$. }
		\label{figure:C17}
	\end{figure}

	\begin{figure}
		\begin{center}
			\begin{tikzpicture}
			\node at (0,0) {\includegraphics[width=9cm,height=7.5cm]{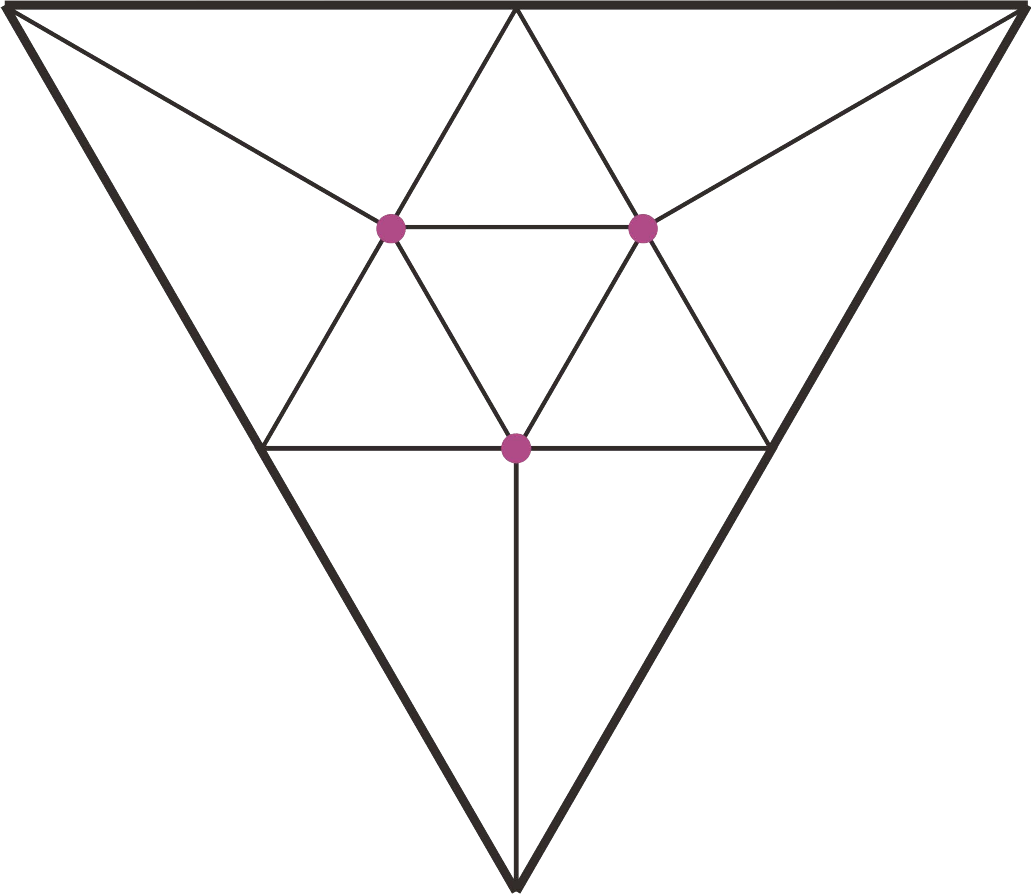}};

			\node at (-4.9,3.7){\Large $a_{1}$};
			
			\node at (4.9,3.7){\Large $a_{2}$};
			
			\node at (0.3,4.0){\Large $a_{12}$};

			\node at (0.7,-3.8){\Large $a_{3}$};

			\node at (-3.6,-0.2){\Large $a_{13}$};
			\node at (3.6,-0.2){\Large $a_{23}$};

			\node at (-1.85,1.7){\small $w_{1,23}$};

			\node at (1.85,1.7){\small $w_{2,13}$};
			
			\node at (0.85,-0.5){\small $w_{3,12}$};
			
			\end{tikzpicture}
		\end{center}
		\caption{The  triangulation $c_{10}$ of the 2-disk. The three purple vertices $w_{1,23}, w_{2,13}, w_{3,12}$ are new vertices in the boundary of the 3-complex $\mathcal{C}_{57}$. }
		\label{figure:c10}
	\end{figure}


		\textbf{The 3-complex $\mathcal{C}_{57}$.} 
			  There is a triangulation $c_{10}$ of a 2-disk in Figure 	\ref{figure:c10}, consisting of ten triangles, with labels at vertices. The triangulation $c_{10}$ is called a \emph{special subdivision}  of the triangle with vertex set $\{a_1,a_2,a_3\}$, see 	\cite{Dranishnikov:1999, PrzytyckiS:2009}.  There are three vertices in the interior of  $c_{10}$, say $$w_{1,23}=w_{1,32}, ~~ w_{2,13}=w_{2,31}, ~~ w_{3,12}=w_{3,21}.$$
	Take the cone to the vertex $v_4$ in  $\mathcal{C}_{17}$ base on  $c_{10}$, so this cone consists of     ten  tetrahedra, each tetrahedron corresponds to a triangle in $c_{10}$. The side surface of the cone is a triangulation $c_6$ of a 2-disk, see Figure \ref{figure:c6c4}. 
	Glue the cone to the 3-complex  $\mathcal{C}_{17}$ along the the side surface of the cone, which is a disk and it is one-fourth  of the boundary of  $\mathcal{C}_{17}$ as showed  in   Figure 	\ref{figure:C17}.
		Similarly, we glue the other three cones to one fourth of the boundary of the 3-complex $\mathcal{C}_{17}$ around the vertices $v_1$, $v_2$, and $v_3$, respectively. Each cone is based on a special subdivision of a triangle, and the vertex sets of the three triangles are $\{a_2,a_3,a_4\}$, $\{a_1,a_3,a_4\}$, and $\{a_1,a_2,a_4\}$, respectively.
	The union of these fifty-seven tetrahedra  now is a triangulation of a 3-ball, we denote this 3-complex  by $\mathcal{C}_{57}$. The boundary of $\mathcal{C}_{57}$ is illustrated in Figure \ref{figure:C57}.
	
	 	Now $\mathcal{C}_{17}$ and  $\mathcal{C}_{57}$ are subcomplexes of  $\mathcal{C}_{600}$. The (closed) complement of  $\mathcal{C}_{57}$ in $\mathcal{C}_{600}$ is denoted by  $\mathcal{C}_{543}$, consisting of $543$ tetrahedra. The 3-complex  $\mathcal{C}_{543}$  is a full subcomplex of $\mathcal{C}_{600}$, which is called a \emph{special subdivision}  of the tetrahedron with vertex set $\{a_1,a_2,a_3,a_4\}$, see 	\cite{ PrzytyckiS:2009}.
	 	Moreover, $a_1,a_2,a_3$ and $a_4$ serve as the \emph{genuine vertices} of $\mathcal{C}_{543}$.

	\begin{figure}
	\begin{center}
		\begin{tikzpicture}
		\node at (0,0) {\includegraphics[width=11cm,height=3.9cm]{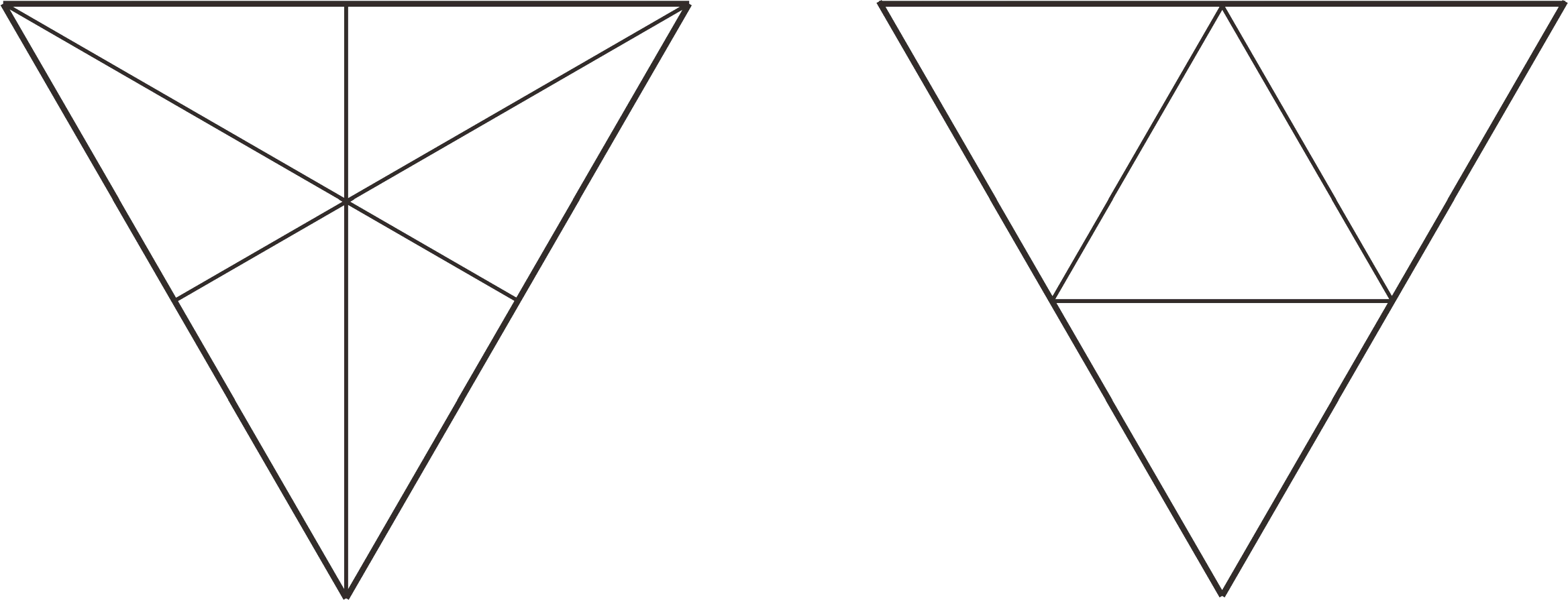}};

		\node at (5.0,2.3){\Large $a_{2}$};
		
		\node at (0.7,2.3){\Large $a_{3}$};
		
		\node at (3.9,-2.4){\Large $a_{4}$};
		
		\node at (2.8,2.4){\Large $a_{23}$};
		\node at (4.9,0.0){\Large $a_{24}$};
		\node at (1.4,0.0){\Large $a_{34}$};

		\end{tikzpicture}
	\end{center}
	\caption{The left subfigure is the triangulation  $c_6$ of a 2-disk. The right subfigure is the triangulation  $c_4$ of a 2-disk.}
	\label{figure:c6c4}
\end{figure}

	 The following is proved in  \cite{PrzytyckiS:2009}:
	 \begin{lemma} \label{lemma:fnsC543}  For subcomplexes $\mathcal{C}_{57}$ and  $\mathcal{C}_{543}$ in  $\mathcal{C}_{600}$:
	 	
	 		\begin{itemize}
	 			
	 			\item 	The 3-complex  $\mathcal{C}_{543}$  is a  flag-no-square triangulation of the 3-ball;
	 			
	 				\item $$\mathcal{C}_{57}=\text{span} \{u \in \mathcal{C}_{600}^{(0)}: \text{dist}(u,\Delta_0) \leq 1 \}.$$

	 			\end{itemize}

	 	\end{lemma}

 		We view the boundary of $\mathcal{C}_{543}$ as four parts, each of which is isomorphic to the 2-complex $c_{10}$; that is, each part is a special subdivision of a 2-face of the tetrahedron, see Figure \ref{figure:C57}. Moreover, the following lemma will be used in the proof of Lemma \ref{lemma:squareC1133}.

 		\begin{lemma} \label{lemma:interiorC543} 
 			
 			Consider the 2-complex $c_{10}$ as  one-fourth   of $\partial \mathcal{C}_{543}$, which is a special subdivision of the triangle with vertex set $\{a_2,a_3,a_4\}$. 
 			There does not exist a triple of vertices $\{E,G,F\}$  in $\mathcal{C}_{543}$ such that $G$ and $F$ lie in the 2-complex $c_{10}$, while edges  $[E,G]$ and $[E,F]$  are present in $\mathcal{C}_{543}$ but edge $[G,F]$ is not,	if:
 			 \begin{enumerate} 
 			\item \label{interiorC543:1} Either the vertex $E$ lies in the interior of $\mathcal{C}_{543}$;
 			
 			\item  \label{interiorC543:2}  Or the vertex $E$ lies in the interior of $\partial \mathcal{C}_{543} \setminus  c_{10}$.
 			 \end{enumerate}

 		\end{lemma}

 	\begin{proof} 
 		Otherwise, consider the cone $\mathcal{C}_{10}$ over $c_{10}$ with vertex $v_1$ in $\mathcal{C}_{57}$. Since $\mathcal{C}_{600}=\mathcal{C}_{543} \cup \mathcal{C}_{57}$, there exists a  4-circle $(E, G, v_1,F, E)$ in $\mathcal{C}_{600}$. Note that $\mathcal{C}_{543}$ is a full subcomplex of $\mathcal{C}_{600}$, so there is no edge $[G,F]$ in $\mathcal{C}_{600}$, and the edge $[v_1,E]$  must be presented    in $\mathcal{C}_{600}$. Thus  $E$ is a vertex of $\mathcal{C}_{57}$  by Lemma \ref{lemma:fnsC543}. Any vertex adjacent to $v_1$ is either one of $\{v_2,v_3,v_4\}$ or a vertex in $c_{10}$, where $c_{10}$ is viewed as a special subdivision of the triangle with vertex set $\{a_2,a_3,a_4\}$.
 	 However, $E$ cannot be one of  $v_2,v_3,v_4$ in $\Delta_0$, since these vertices are not contained in $\mathcal{C}_{543}$; nor can $E$  lie in $c_{10}$, which again contradicts our assumption.  This ends the proof of Lemma \ref{lemma:interiorC543}.


 			\end{proof}
 		
 	\begin{figure}
 		\begin{center}
 			\begin{tikzpicture}
 			\node at (0,0) {\includegraphics[width=11cm,height=9cm]{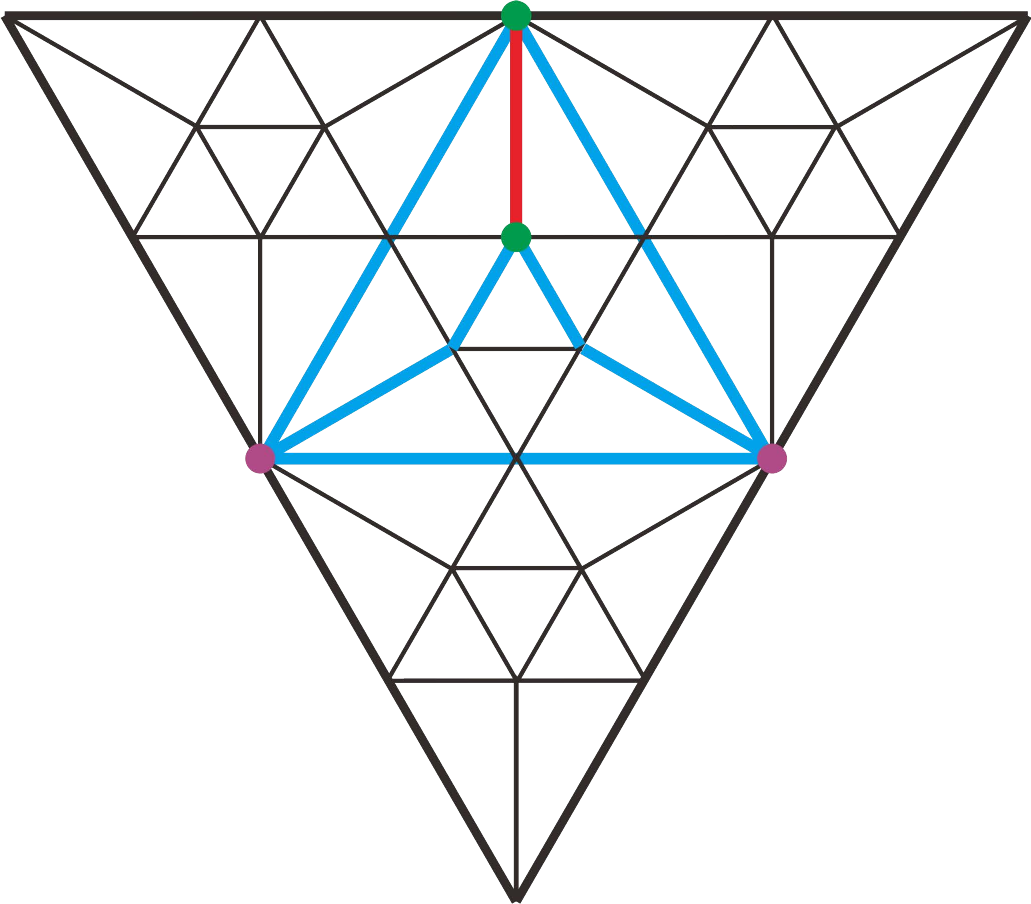}};

 			\node at (0.0,1.6){\small $B$};

 			\node at (0.1,4.9){\Large $A=a_2$};
 			\node at (-3.9,-0.2){\Large $D=a_3$};
 			\node at (3.9,-0.2){\Large $C=a_4$};
 			
 			\node at (5.3,4.9){\Large $a_{1}$};
 			\node at (-5.3,4.9){\Large $a_{1}$};
 			\node at (0.5,-4.7){\Large $a_{1}$};
 			
 			\node at (-3.0,4.9){\Large $a_{12}$};
 			\node at (3.0,4.9){\Large $a_{12}$};

 			\node at (4.8,2.0){\Large $a_{14}$};
 			\node at (1.85,-2.4){\Large $a_{14}$};
 			
 			\node at (-4.8,2.0){\Large $a_{13}$};
 			\node at (-1.85,-2.4){\Large $a_{13}$};

 			\end{tikzpicture}
 		\end{center}
 		\caption{By gluing edges in this figure that have identical pairs of labeled vertices, we obtain a triangulation of the 2-sphere. This 2-sphere is the boundary of the 3-complex $\mathcal{C}_{57}$, and also the boundary of the 3-complexes  $\mathcal{C}_{543}$ and  $\mathcal{C}'_{543}$. Compared with the labels in Figure \ref{figure:c10}, the figure contains twelve unlabeled vertices, denoted $w_{i,jk}=w_{i,kj}$ for $i \neq j \neq k \in \{1,2,3,4\}$, with the exception of the vertex $B=w_{2,34}=w_{2,43}$. }

 		\label{figure:C57}
 	\end{figure}

 		\subsection{The 3-complex $\mathcal{C}'_{543}$}  \label{subsection:cprime543}


 	This is  one of the  key objects in our proof of Theorems \ref{thm:almostfns}   and  \ref{thm:conformaldim}. 
 	Consider the 3-complex $\mathcal{C}_{543}$, but we view it as another subdivision $\mathcal{C}_{543}'$ of the tetrahedron with vertex set $\{A,B,C,D\}$. Namely, we have a tetrahedron $\Delta$ whose edges are shown in thick red and thick cyan in Figure \ref{figure:C57}. Moreover, the boundary of $\mathcal{C}_{543}'$ is a certain subdivision of $\partial \Delta$, as shown in Figure \ref{figure:C57}.
 In other words, the red and cyan arcs in Figure \ref{figure:C57} separate the boundary of $\mathcal{C}_{543}'$ into four parts:
 	\begin{itemize}
 		
 		\item 
 		The triangle with vertex set $\{A,B,D\}$ consists of three smaller triangles. We denote this triangulation of a 2-disk by $c_3$;
 		
 		\item 
 		The triangle with vertex set $\{A,B,C\}$ consists of three smaller triangles. It is also  isomorphic to the 2-complex  $c_3$;

 		\item 
 		The triangle with vertex set $\{B,C,D\}$ consists of four smaller triangles. We denote this triangulation of 2-disk by $c_4$, see Figure \ref{figure:c6c4};
 		
 		\item 
 		The triangle with vertex set $\{A,C,D\}$  consists of $30$  smaller triangles. We denote this triangulation of 2-disk by $c_{30}$, see Figure  	\ref{figure:c30}.

 	\end{itemize}

 		Moreover, $A, B, C$ and $D$ serve as the \emph{genuine vertices} of $\mathcal{C}'_{543}$.
 	As $\mathcal{C}_{543}$ is a flag-no-square triangulation of the 3-ball, the 3-complex $\mathcal{C}'_{543}$ is isomorphic to $\mathcal{C}_{543}$ as a 3-complex, we have
 	
 	\begin{lemma} \label{lemma:fnsCprime543}  The 3-complex  $\mathcal{C}'_{543}$  is a   flag-no-square triangulation of the 3-ball.
 	\end{lemma}

 \begin{figure}
 	\begin{center}
 		\begin{tikzpicture}
 		\node at (0,0) {\includegraphics[width=11cm,height=9cm]{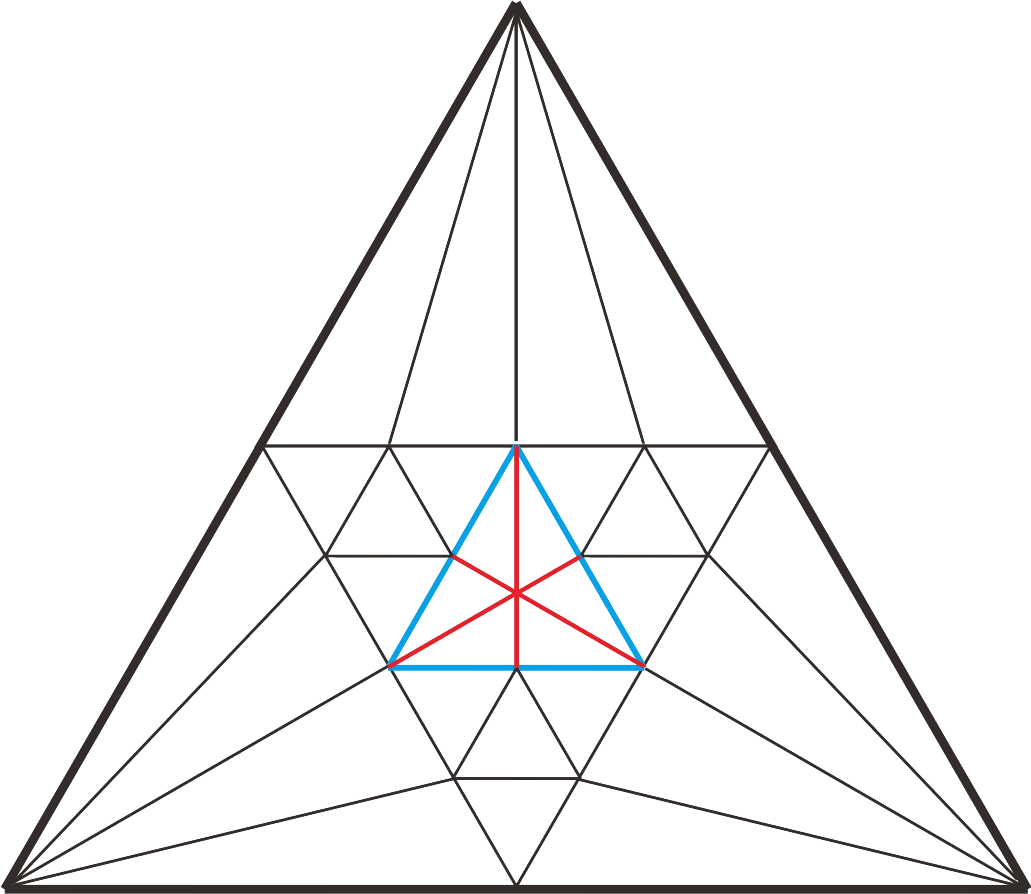}};

 		\node at (0.6,4.4){ $A$};
 		\node at (-5.7,-3.8){ $D$};
 		\node at (5.7,-3.8){ $C$};
 		
 		\node at (0.3,0.2){\tiny $a_{12}$};
 			\node at (1.7,-2.15){\tiny $a_{14}$};
 				\node at (-1.7,-2.15){\tiny $a_{13}$};

 		\end{tikzpicture}
 	\end{center}
 	\caption{The triangulation $c_{30}$ of a 2-disk. }
 	\label{figure:c30}
 \end{figure}

 \subsection{The 3-complex $\mathcal{C}_{1137}$}  \label{subsection:c1137}  
 Note that the subdivisions of the triangles with vertex sets
 \(\{A,C,D\}\) and \(\{B,C,D\}\) in \(\mathcal{C}'_{543}\)
 are not isomorphic to the triangulation \(c_{10}\);
 they are isomorphic to the \(2\)-complexes \(c_{30}\) and \(c_{4}\),
 respectively. We need two more \(3\)-complexes, say
 \(\mathcal{C}_{1137}\) and \(\mathcal{C}_{1163}\), which are certain
 subdivisions of a tetrahedron, with a \(2\)-face of the tetrahedron
 subdivided into \(c_4\) and \(c_{30}\), respectively.
 
 Before constructing \(\mathcal{C}_{1137}\) and \(\mathcal{C}_{1163}\),
 we need two preliminary \(3\)-complexes, say \(\mathcal{C}_{590}\)
 and \(\mathcal{C}_{1133}\).

 \textbf{The 3-complex $\mathcal{C}_{590}$.} Since $\mathcal{C}_{57}\cup \mathcal{C}_{543}=\mathcal{C}_{600}$ is a triangulation of the 3-sphere, we delete ten  tetrahedra from  $\mathcal{C}_{600}$. Specifically, we remove the (open) cone $\mathcal{C}_{10}$ over the vertex $v_1$ based on the triangulation $c_{10}$ of the 2-disk (a special subdivision of a triangle with vertex set $\{a_2,a_3,a_4\}$). We denote this triangulation of the 3-ball by  $\mathcal{C}_{590}$. The boundary of the 3-complex $\mathcal{C}_{590}$ is a triangulation of the 2-sphere consisting of 16 triangles, formed by the triangulations $c_6$  and $c_{10}$ of the 2-disk; see Figures~\ref{figure:c10} and~\ref{figure:c6c4}.

  \begin{lemma} \label{lemma:squareC590}  The 3-complex $\mathcal{C}_{590}$  is   a flag  triangulation of the 3-ball,   even not a full subcomplex of $\mathcal{C}_{600}$.    Moreover, $\mathcal{C}_{590}$ is  not a  square-free  triangulation of the 3-ball, the only empty squares in $\mathcal{C}_{590}$   are the circles:
  	
  		\begin{itemize}
  		
  		\item   $(v_1, a_{23}, w_{2,34}, a_{24}, v_1)$;

  			\item  $(v_1, a_{32}, w_{3,24},a_{34}, v_1)$;

  			\item   $(v_1, a_{42}, w_{4,23},a_{43}, v_1)$.
  		
  			\end{itemize}

 \end{lemma}
 
  \begin{proof}
  	First,  readers may refer to  Figure 	\ref{figure:c10} for vertex labels in the special subdivision of the triangle with vertex set $\{a_1,a_2,a_3\}$; analogous labels apply to the special subdivision of the triangle with vertex set $\{a_2,a_3,a_4\}$.  So the new vertices of the special subdivision of the triangle with
  	vertex set \(\{a_2,a_3,a_4\}\) are
  	\[
  	\{a_{23},a_{24},a_{34},w_{2,34},w_{3,24},w_{4,23}\}.
  	\]

  	Note that the 3-complexes  $\mathcal{C}_{590}$ and  $\mathcal{C}_{600}$ have identical vertex sets. Specifically,  $\mathcal{C}_{600}$ contains exactly three additional edges compared with  $\mathcal{C}_{590}$, namely $[v_1,w_{2,34}]$, $[v_1,w_{3,24}]$ and $[v_1,w_{4,23}]$. Moreover,  $\mathcal{C}_{600}$ has exactly twelve additional triangles and ten additional tetrahedra relative to  $\mathcal{C}_{590}$.
  	
Let $\{C_1, C_2,\cdots,  C_k\}$ be a collection of pairwise adjacent vertices in $\mathcal{C}_{590}$. 	Then $k \leq 4$ and  $\{C_1, C_2,\cdots,  C_k\}$ spans   a $(k-1)$-simplex $s$ in $\mathcal{C}_{600}$. We may assume $s$ is one of  the aforementioned    twelve triangles or ten tetrahedra. However,  each such simplex contains at least one edge from the set $$\{[v_1,w_{2,34}], [v_1,w_{3,24}], [v_1,w_{4,23}]\}, $$ which are edges belonging to  $\mathcal{C}_{600}$   but not to $\mathcal{C}_{590}$. Therefore, $\mathcal{C}_{590}$  is   a flag  triangulation of the 3-ball.

  	Suppose that  $(C_1, C_2, C_3,C_4, C_1)$ is an empty square in $\mathcal{C}_{590}$. 
  	Since there is no empty square in $\mathcal{C}_{600}$, we may assume that $C_1=v_1$ and  $C_3=w_{2,34}$.	
  		 Now $\{v_1, C_2, w_{2,34}\}$ spans a triangle in  $\mathcal{C}_{600}$. Since each edge in $\mathcal{C}_{600}$ has  edge link the pentagon, there are  five choices of the vertex $C_2$ (and $C_4$) in $\mathcal{C}_{600}$, say one of $$\{a_2, a_{23},w_{3,24},w_{4,23},a_{24}\}.$$ But in fact there is no edge  in $\mathcal{C}_{590}$ between $v_1$ and $w_{3,24}$ (and no edge  between $v_1$ and $w_{4,23}$). There are exact three choices of $C_2$ (and $C_4$), say one of $$\{a_2, a_{23},a_{24}\}.$$ We have  $$(C_1, C_2, C_3,C_4, C_1)= (v_1, a_{23}, w_{2,34},a_{24}, v_1)$$ up to re-ordering the vertices.

  		Similarly, there are two more empty squares, say 
  	$(v_1, a_{32}, w_{3,24},a_{34}, v_1)$ and 
  $(v_1, a_{42}, w_{4,23},a_{43}, v_1)$.

  	 \end{proof}

We also have the following lemma, which will be used in the proof of Lemma \ref{lemma:squareC1133}.
 
 \begin{lemma} \label{lemma:interiorC590}

 	Consider the 2-complex $c_{10}$, as a  part of $\partial \mathcal{C}_{590}$  and isomorphic to the  special subdivision of the triangle with vertex set $\{a_2,a_3,a_4\}$. If a vertex $E$ lies in the interior of $\mathcal{C}_{590}$, there is no pair of vertices $G$ and  $F$  in the 2-complex $c_{10}$ such that $\mathcal{C}_{590}$ contains edges $[E,G]$ and $[E,F]$ but not the edge $[G,F]$, except for the following cases:
\begin{itemize}
	\item $\{E,G,F\}=\{v_2, a_{3}, a_{4}\}$;
	
	\item $\{E,G,F\}=\{v_3, a_{2}, a_{4}\}$;
		
		\item $\{E,G,F\}=\{v_4, a_{2}, a_{3}\}$. 
		
\end{itemize}

 \end{lemma}

\begin{proof} Otherwise, consider the cone  $ \mathcal{C}_{10}$ over vertex $v_1$ based on $c_{10}$.
	Since $\mathcal{C}_{600}=\mathcal{C}_{590} \cup \mathcal{C}_{10}$,  there is a length 4 circle $(E, G, v_1,E, E)$, and it  can not be empty in $\mathcal{C}_{600}$. But from 
	$\mathcal{C}_{590}$ to $\mathcal{C}_{600}$, we only add three  more edges, each of which has $v_1$ as a vertex; thus,  there is an edge $[v_1,E]$ in $\mathcal{C}_{600}$. Then  $E$ is a vertex of the 3-complex 	$\mathcal{C}_{57}$ by Lemma \ref{lemma:fnsC543}. The vertex  $v_1$ is adjacent to exactly twelve vertices in $\mathcal{C}_{600}$,
say 	$v_2,v_3,v_4$ and nine vertices in $c_{10}$. So we may assume  $E=v_2$ in $\mathcal{C}_{57}$.  The vertex $v_2$ is adjacent to exactly twelve vertices in $\mathcal{C}_{600}$,
say 	$v_1,v_3,v_4$ and nine vertices in a copy of  $c_{10}$ (which is the special subdivision of the triangle with vertex set $\{a_1,a_3,a_4\}$). Then $\{G, F\}=\{a_{3}, a_{4}\}$  up to re-ordering the vertices.  

Similarly, there are two more   triples of vertices, say 
  $\{E,G,F\}=\{v_3, a_{2}, a_{4}\}$ and 
  $\{E,G,F\}=\{v_4, a_{2}, a_{3}\}$ satisfy the conditions in Lemma \ref{lemma:interiorC590}.

\end{proof}

 	\textbf{The 3-complex  $\mathcal{C}_{1133}$.} We take another copy of $\mathcal{C}_{543}$ and denote it by $\overline{\mathcal{C}}_{543}$, which is a special subdivision of the tetrahedron with vertex set $\{\overline{a}_1,\overline{a}_2,\overline{a}_3,\overline{a}_4\}$. Here, each of the vertices  $v_i$, $a_{i}$, $a_{ij}$ and $w_{i,jk}$ in $\mathcal{C}_{543}$ is denoted by $\overline{v}_i$, $\overline{a}_{i}$, $\overline{a}_{ij}$ and $\overline{w}_{i,jk}$ in $\overline{\mathcal{C}}_{543}$, respectively.
 	We glue $\mathcal{C}_{590}$ and $\overline{\mathcal{C}}_{543}$ along the special subdivisions $c_{10}$ of the triangles with vertex sets  $\{a_2,a_3,a_4\}$ and  $\{\overline{a}_2,\overline{a}_3,\overline{a}_4\}$,  identifying $a_i$ with $\overline{a}_i$ for $i=2,3,4$. This  results in a triangulation of a 3-ball with 1133 tetrahedra; we denote this 3-complex by $\mathcal{C}_{1133}$. We note that the vertex $$a_1 \in 
 	\mathcal{C}_{543} 	\subset \mathcal{C}_{590}	\subset \mathcal{C}_{1133}$$ lies in the interior of $ \mathcal{C}_{1133}$.
 	Figure~\ref{figure:C1133} illustrates the boundary of \(\mathcal{C}_{1133}\).


 	\begin{lemma} \label{lemma:squareC1133} 
 		
 		The \(3\)-complex \(\mathcal{C}_{1133}\) is a flag triangulation of the \(3\)-ball. However, it is not a square-free triangulation, and the only empty squares in \(\mathcal{C}_{1133}\) are the three empty squares in \(\mathcal{C}_{590}\) given in Lemma~\ref{lemma:squareC590}.
 
 	\end{lemma}

 	\begin{proof} 	
 			If $\{C_1, C_2,\cdots, C_k\}$ is a collection of vertices in $$\mathcal{C}_{1133} = \mathcal{C}_{590} \cup_{c_{10}} \overline{\mathcal{C}}_{543} $$ that are pairwise connected by edges, then:

 		  \begin{itemize}
 			
 			\item  
 			Either there exists a vertex, say $C_1$, that lies in the interior of $\mathcal{C}_{590}$ (or $\overline{\mathcal{C}}_{543}$); in this case, all vertices $C_i$ lie in the complex $\mathcal{C}_{590}$ (or $\overline{\mathcal{C}}_{543}$).
 			By Lemmas \ref{lemma:fnsC543} and \ref{lemma:squareC590}, $k$ is at most $4$, and $\{C_1, C_2,\cdots, C_k\}$ spans a $(k-1)$-simplex in $\mathcal{C}_{1133}$;

 		\item  Or all  vertices $C_i$ lie in $c_{10} \cup \partial \mathcal{C}_{1133}$. Since there is no edge between a vertex in $\partial \mathcal{C}_{590} \setminus c_{10}$ and a vertex in $\partial \overline{\mathcal{C}}_{543} \setminus c_{10}$, all  vertices $C_i$ lie in $\partial \overline{\mathcal{C}}_{543}$ or $\partial \mathcal{C}_{590}$. By Lemmas \ref{lemma:fnsC543} and \ref{lemma:squareC590}, $k$ is at most 4, and $\{C_1, C_2,\cdots, C_k\}$ spans a $(k-1)$-simplex in $\mathcal{C}_{1133}$.
 		  \end{itemize}	
 		In conclusion, the 3-complex $\mathcal{C}_{1133}$ is a flag triangulation of the 3-ball.

 	The \(3\)-complex \(\mathcal{C}_{543}\) is a full subcomplex of \(\mathcal{C}_{600}\). In particular, there are no edges between the vertices \(a_{23}\) and \(a_{24}\) in \(\mathcal{C}_{543}\), and hence no edges between the vertices \(\overline{a}_{23}\) and \(\overline{a}_{24}\) in \(\overline{\mathcal{C}}_{543}\). Consequently, when \(\overline{\mathcal{C}}_{543}\) is glued to \(\mathcal{C}_{590}\) to obtain \(\mathcal{C}_{1133}\), the three empty squares in \(\mathcal{C}_{590}\) (given in Lemma~\ref{lemma:squareC590}) remain empty squares in \(\mathcal{C}_{1133}\).

 	 Suppose \((C_1, C_2, C_3, C_4, C_1)\) is an empty square in \(\mathcal{C}_{1133}\), we may also assume that the circle is not totally embedded in \(\overline{\mathcal{C}}_{543}\) or \(\mathcal{C}_{590}\). From Figure~\ref{figure:C1133}, one can check directly that the boundary of \(\mathcal{C}_{1133}\) is a flag-no-square triangulation of a \(2\)-sphere. Thus, the circle \((C_1, C_2, C_3, C_4, C_1)\) is not totally embedded in the boundary of \(\mathcal{C}_{1133}\).  Therefore, at least one vertex, say \(C_1\), lies in one of the following:
 	 
 	 \begin{itemize}
 	 	\item the interior of \(\overline{\mathcal{C}}_{543}\);
 	 	\item the interior of \(\mathcal{C}_{590}\);
 	 	\item the interior of \(c_{10} = \overline{\mathcal{C}}_{543} \cap \mathcal{C}_{590}\).
 	 \end{itemize}

 	  		 If $C_1$ lies in the interior of    $\overline{\mathcal{C}}_{543}$,  then the vertices  $C_2$ and $C_4$ lie in $\overline{\mathcal{C}}_{543}$, and $C_3$  lies in $\mathcal{C}_{590}$. This in turn implies $C_2$ and $C_4$ lie in $c_{10}=\overline{\mathcal{C}}_{543} \cap \mathcal{C}_{590}$. The triple $\{C_1,C_2,C_4\}$  contradicts  to  Item (\ref{interiorC543:1})  of Lemma \ref{lemma:interiorC543}.

 	  		 So we assume that none of the \(C_i\) lies in the interior of \(\overline{\mathcal{C}}_{543}\), and that \(C_1\) lies in the interior of \(\mathcal{C}_{590}\). Then the vertices \(C_2\) and \(C_4\) lie in \(\mathcal{C}_{590}\), and \(C_3\) lies in \(\partial \overline{\mathcal{C}}_{543} \setminus c_{10}\). This in turn implies that \(C_2\) and \(C_4\) lie in \(c_{10} = \overline{\mathcal{C}}_{543} \cap \mathcal{C}_{590}\). By Lemma~\ref{lemma:interiorC590}, without genericity, we may assume that
 	  		 \[
 	  		 \{C_1, C_2, C_4\} = \{v_2, a_3, a_4\}.
 	  		 \]
 	  		 From Figure~\ref{figure:C57}, \(C_3\) must be \(a_{34} = \overline{a}_{34} \in c_{10}\), and all of \(\{C_1, C_2, C_3, C_4\}\) lie in \(\mathcal{C}_{590}\), a contradiction.

 	  		 Finally if $C_1$ lies in the interior of $c_{10}=\overline{\mathcal{C}}_{
 	  		 	543} \cap \mathcal{C}_{590}$,  without genericity, we may assume $C_1=w_{2,34}=\overline{w}_{2,34} $:
 	  		 
 	  		  \begin{itemize}
 	  		 	
 	  		 	\item   We first assume $C_2$  lies in $\partial \overline{\mathcal{C}}_{543}$:
 	  		 
 	  		 \begin{enumerate}
 	  		 	
 	  		 	\item \label{squareC1133item1}
 	  		 	 If $C_2 \in \partial \overline{\mathcal{C}}_{543}\setminus c_{10}$, then $C_3$  lies in $\partial \overline{\mathcal{C}}_{543}$, and $C_4$ lies in $ \partial \mathcal{C}_{590}$. This in turn implies $C_3$ lies in $c_{10}=\overline{\mathcal{C}}_{543} \cap \mathcal{C}_{590}$, and $C_3$ is one of $a_3, a_4$ or $a_{34}$. Consider the triple of vertices $\{C_2,C_1,C_3\}$, which contradictions to  Item (\ref{interiorC543:2})  of Lemma \ref{lemma:interiorC543}.
 	  		 
 	  		 	\item \label{squareC1133item2}

 	  		 	If \(C_2 \in c_{10}\), then from (\ref{squareC1133item1}), by symmetry of \(C_2\) and \(C_4\), we may assume that \(C_4\) does not lie in \(\partial \overline{\mathcal{C}}_{543} \setminus c_{10}\), so \(C_4 \in \partial \mathcal{C}_{590}\).
 	  		 	We must have \(C_3 \in \partial \overline{\mathcal{C}}_{543} \setminus c_{10}\); otherwise, all of \(\{C_1, C_2, C_3, C_4\}\) lie in \(\mathcal{C}_{590}\), a contradiction.
 	  		 	But then \(C_4\) lies in \(\partial \overline{\mathcal{C}}_{543} \cap \partial \mathcal{C}_{590} = c_{10}\),  and all of \(\{C_1, C_2, C_3, C_4\}\) lie in  $ \overline{\mathcal{C}}_{543}$,  also a contradiction.


 	  		\end{enumerate}

 	  \item

 		 By symmetry, we now  assume  both $C_2$ and $C_4$   lie in $\partial \mathcal{C}_{590}\setminus c_{10}$, and then  $C_3$   lies in $ c_{10}$.  Then  all of  $\{C_1, C_2, C_3, C_4\}$ lie in $\mathcal{C}_{590}$, a contradiction.

 	  \end{itemize}

 In conclusion, there are no other empty squares in $\mathcal{C}_{1133}$ except the three empty squares in $\mathcal{C}_{590}$ showed in Lemma \ref{lemma:squareC590}.

 	\end{proof}

 	 \begin{figure}
 	 	\begin{center}
 	 		\begin{tikzpicture}
 	 		\node at (0,0) {\includegraphics[width=11cm,height=9cm]{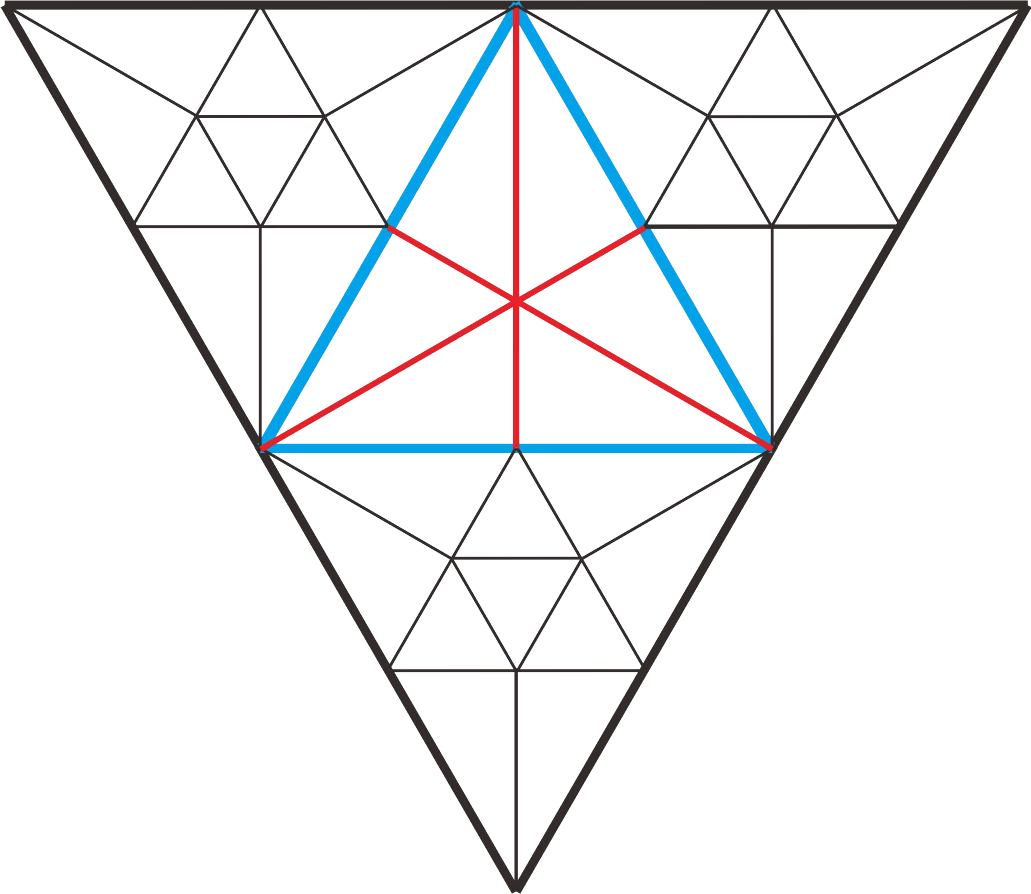}};

 	 		\node at (0.3,1.8){\small $v_1$};
 	 		
 	 		\node at (0.45,-0.25){\small $\overline{a}_{34}$};
 	 		
 	 		\node at (-1.9,2.0){\small $\overline{a}_{23}$};
 	 		
 	 		\node at (1.9,2.0){\small  $\overline{a}_{24}$};

 	 		\node at (0.1,4.9){\Large $a_{2}=\overline{a}_2$};
 	 		\node at (-3.7,-0.2){\Large $a_{3}=\overline{a}_3$};
 	 		\node at (3.7,-0.2){\Large $a_{4}=\overline{a}_4$};
 	 		
 	 		\node at (5.3,4.9){\Large $\overline{a}_{1}$};
 	 		\node at (-5.3,4.9){\Large $\overline{a}_{1}$};
 	 		\node at (0.5,-4.7){\Large $\overline{a}_{1}$};
 	 		
 	 		\node at (-3.0,4.9){\Large $\overline{a}_{12}$};
 	 		\node at (3.0,4.9){\Large $\overline{a}_{12}$};

 	 		\node at (4.8,2.0){\Large $\overline{a}_{14}$};
 	 		\node at (1.85,-2.4){\Large $\overline{a}_{14}$};
 	 		
 	 		\node at (-4.8,2.0){\Large $\overline{a}_{13}$};
 	 		\node at (-1.85,-2.4){\Large $\overline{a}_{13}$};

 	 		\end{tikzpicture}
 	 	\end{center}
 	 	\caption{By gluing edges in this figure that have identical pairs of labeled vertices,  we obtain a flag-no-square triangulation of the 2-sphere, which is the  boundary of the 3-complex $\mathcal{C}_{1133}$.}
 	 	\label{figure:C1133}
 	 \end{figure}

 	 Let \(c_4\) be the \(2\)-complex shown on the right side of Figure~\ref{figure:c6c4}, which is a subdivisition of a triangle with vertex set $\{a_2,a_3,a_4\}$. 	Consider the  cone  $\mathcal{C}_{4}$ based on the 2-complex  $c_4$  to $v_1$, the side surface of the cone  $\mathcal{C}_{4}$ is the triangulation $c_6$ of a 2-disk, shown on the left side of Figure~\ref{figure:c6c4}. 	We glue  the cone   $\mathcal{C}_{4}$ to $\mathcal{C}_{1133}$ along the 2-complex $c_6$, resulting  in  a triangulation of a 3-ball with 1137 tetrahedra, and we denote the 3-complex  by $\mathcal{C}_{1137}$. Here $c_6$ in $\mathcal{C}_{1133}$ is the 2-complex with vertex set $\{v_1, a_2, a_3, a_4, \overline{a}_{23},\overline{a}_{24},\overline{a}_{34}\}$ as shown in Figure  \ref{figure:C1133}.  The boundary of  $\mathcal{C}_{1137}$ thus  has four parts: three copies of  the 2-complex $c_{10}$,  and one copy of the 2-complex $c_4$. 	Moreover, $\overline{a}_1,a_2=\overline{a}_2,a_3=\overline{a}_3,a_4=\overline{a}_4$ serve as the \emph{genuine vertices} of  $\mathcal{C}_{1137}$. 	A sketch of the 3-complex $\mathcal{C}_{1137}$ is illustrated in Figure \ref{figure:C1137}.

 		\begin{figure}
 		\begin{center}
 			\begin{tikzpicture}
 			\node at (0,0) {\includegraphics[width=8cm,height=5cm]{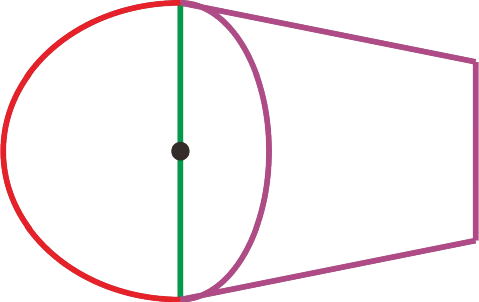}};

 			\node at (-1.4,0.1){\Large $v_{1}$};

 			\node at (2.8,0.0){\Large $\overline{\mathcal{C}}_{543}$};
 			\node at (0.00,0.4){ $\mathcal{C}_{590}$};

 			\node at (-2.25,-1.4){\Large $\mathcal{C}$};

 			\end{tikzpicture}
 		\end{center}
 		\caption{A cartoon of   the 3-complex $\mathcal{C}_{1137}$, the red arc  indicates the triangulation $c_4$, the green arc  indicates  the triangulation $c_6$, each of the purple arcs  indicates a triangulation  $c_{10}$, and  $\mathcal{C}$  indicates the cone $\mathcal{C}_{4}$.  There is a  similar cartoon for   the 3-complex $\mathcal{C}_{1163}$, replacing    $c_4$ by $c_{30}$,   and replacing  $\mathcal{C}_{4}$ by $\mathcal{C}_{30}$. }
 		\label{figure:C1137}
 	\end{figure}

 	Even through the 3-complex $\mathcal{C}_{1137}$ is not a flag-no-square triangulation of the 3-ball, but we have a control over  the empty simplices and empty squares in it.

 	 \begin{lemma} \label{lemma:fnsC1137simplex}  The \(3\)-complex \(\mathcal{C}_{1137}\) is simplicial, but is not flag. The only empty simplices in \(\mathcal{C}_{1137}\) are the triangles with the following vertex sets:
 	  \begin{itemize}
 	 	
 	 	\item  $\{a_{23},a_{24},w_{2,34}\}$;
 	 	
 	 		\item  $\{a_{24},a_{34},w_{4,23}\}$;
 	 		
 	 		\item  $\{a_{34},a_{23},w_{3,24}\}$

  		\end{itemize}
  	in $\mathcal{C}_{4} \cup \mathcal{C}_{590}\subset \mathcal{C}_{1137}=\mathcal{C}_{4} \cup \mathcal{C}_{590} \cup \overline{\mathcal{C}}_{543}$.

 \end{lemma}

  \begin{proof}

  	First,  we note that $a_i=\overline{a}_i$ and $a_{ij}=a_{ji}=\overline{a}_{ij}=\overline{a}_{ji}$ as vertices of  	$\mathcal{C}_{1137}$ for $i, j \in \{2,3,4\}$. Moreover, the  \(3\)-complexes   $\mathcal{C}_{1137}$ and  $\mathcal{C}_{1133}$ have the same vertex set, but
  	 \begin{itemize}
  		
  		\item 
  		$\mathcal{C}_{1137}$ have three more open edges than  $\mathcal{C}_{1133}$, say $$(a_{23},a_{24}), (a_{23},a_{34}), (a_{24},a_{34});$$ 
  		
  			\item 	  $\mathcal{C}_{1137}$ have four more open 2-simplexes than  $\mathcal{C}_{1133}$;
  			
  			\item   $\mathcal{C}_{1137}$ have four more open 3-simplexes than  $\mathcal{C}_{1133}$.
  		
  		 \end{itemize}
  	Since the \(3\)-complexes \(\mathcal{C}_{543}\),  \(\mathcal{C}_{590}\) and \(\mathcal{C}_{1133}\) are simplicial complexes, one can then check directly that the \(3\)-complex \(\mathcal{C}_{1137}\) is a simplicial complex.

  	There are edges between any two vertices of \(\{a_{23}, a_{24}, w_{2,34}\}\) in $\mathcal{C}_{1137}$; however, there is no \(2\)-simplex in \(\mathcal{C}_{1137}\) with this vertex set. Thus, there is an empty triangle with vertex set \(\{a_{23}, a_{24}, w_{2,34}\}\).
    Similarly, there are two more  empty triangles claimed in Lemma \ref{lemma:fnsC1137simplex}.

  	Suppose $\{C_1, C_2, \cdots, C_k\}$ is a collection of vertices in $$\mathcal{C}_{1137} = \mathcal{C}_{1133} \cup_{c_{6}} \mathcal{C}_{4}$$ that are pairwise connected by edges. Since $\mathcal{C}_{1133}$ is flag by Lemma \ref{lemma:squareC1133}, we may assume that at least one edge among $\{C_1, \cdots, C_k\}$ is one of the three new edges in $\mathcal{C}_{1137}$. Thus, we assume $C_1 = a_{23}$ and $C_2 = a_{24}$.	
  			From the construction of  $\mathcal{C}_{1133}=\mathcal{C}_{590}\cup_{c_{10}} \overline{\mathcal{C}}_{543}$, there is no edge in  $\mathcal{C}_{1133}$ between any pair in  $\{a_{23}, a_{24},a_{34}\}$:
  				\begin{itemize}
  				\item  
  				 If there is a vertex, say $C_3$,  in $\{C_3, C_4,\cdots,  C_k\}$ which lies in $\mathcal{C}_{1133} \setminus c_{6}$, then $(C_3,C_1,v_1, C_2,C_3)$ is a square in $\mathcal{C}_{1133}$:
  		
  		 \begin{enumerate}
  			
  			\item Either $(C_3,C_1,v_1, C_2,C_3)$ is an empty square in $\mathcal{C}_{1133}$. Then by Lemmas  \ref{lemma:squareC590} and \ref{lemma:squareC1133}, and we must have $C_3=w_{2,34}$ and $k=3$. We have  the  empty triangle  with vertex set 	$\{a_{23},a_{24}, w_{2,34}\}$ claimed in Lemma \ref{lemma:fnsC1137simplex};

  				\item Or there is an edge \([C_3, v_1]\) in \(\mathcal{C}_{1133}\). Then there are edges between \(C_3\) and each of \(\{v_1, a_{23}, a_{24}\}\) in \(\mathcal{C}_{1133}\). Since $v_1$ lies in the interior of $\mathcal{C}_{590}$, but not in $\overline{\mathcal{C}}_{543}$.
  				 From the construction of \(\mathcal{C}_{1133} = \mathcal{C}_{590} \cup \overline{\mathcal{C}}_{543}\), there must be edges between \(C_3\) and each of \(\{v_1, a_{23}, a_{24}\}\) in \(\mathcal{C}_{590} \subset \mathcal{C}_{1133}\). Consider the \(1\)-neighborhood of \(v_1\) in \(\mathcal{C}_{600} \supset \mathcal{C}_{590}\); we have that \(C_3\) must be \(a_2\),  and \(k = 3\). But then \(C_3 = a_2 \in c_6\), a contradiction.

  			 \end{enumerate}

\item   Otherwise,   all the vertices $\{C_1, C_2,\cdots,  C_k\}$ lie on $\mathcal{C}_{4}$. We have   $k$ is at most 4 and $\{C_1, C_2,\cdots,  C_k\}$ spans a $(k-1)$ simplex in  $\mathcal{C}_{4}$. In turn $\{C_1, C_2,\cdots,  C_k\}$ spans a $(k-1)$ simplex in  $\mathcal{C}_{1137}$. 	We note that there is a 2-simplex in $\mathcal{C}_{1137}$  with vertex set $\{a_{23},a_{24},a_{2}\}$, so it is not an empty simplex.

  	\end{itemize}
  
  Similarly, we also have the empty triangles in $\mathcal{C}_{1137}$ with vertex  sets $\{a_{24},a_{34},w_{4,23}\}$ and 
  $\{a_{34},a_{23},w_{23,24}\}$, respectively.
 In conclusion, there are no other empty simplices in $\mathcal{C}_{1137}$ except the three empty triangles  stated in the lemma.
  	
  	 \end{proof} 
   
   Lemmas~\ref{lemma:triplevertices} and~\ref{lemma:1neighborhood} will be used in the proof of Lemma~\ref{lemma:fnsC1137square} when studying the empty squares in \(\mathcal{C}_{1137}\).

 	\begin{lemma} \label{lemma:triplevertices}  
 		
 		Consider the triple of vertices  $\{E,G,F\}$ in $\mathcal{C}_{600}^{(0)}$ such that $[E,G]$ and $[G,F]$ are edges in $\mathcal{C}_{600}^{(1)}$. Up to the automorphism group of the 3-complex $\mathcal{C}_{600}$, there are exactly three equivalence classes of such triples.
 	  	\end{lemma}  
 	  		 	\begin{proof} It is well-known that  $\mathcal{C}_{600}$ is a regular polytope. We may fix $G$ be any vertex  in $\mathcal{C}_{600}^{(0)}$, say $a_{24}$, in $\mathcal{C}_{543}^{(0)} \subset \mathcal{C}_{600}^{(0)}$.
 	  		 		Note that in  the duality between  $\mathcal{C}_{600}$ and $\mathcal{Y}_{120}$, the vertex  $a_{24}$ corresponds to the center of  a regular dodecahedron  facet of $\mathcal{Y}_{120}$, and then the edges $[E,a_{24}]$ and $[a_{24}, F]$ correspond to two pentagon facets of the  dodecahedron. 	 In  Figure 	\ref{figure:a24linkin543}, we show the dual of the  dodecahedron corresponds to the vertex $a_{24}$, which is also the boundary of the closed 1-neighborhood of $a_{24}$ in $\mathcal{C}_{600}^{(1)}$, say an icosahedron.  
 	  		 		Then it is easy to see that, up to the automorphism group of \(\mathcal{C}_{600}\), \(\{E, G, F\}\) is equivalent to one of the following:
 	  		 			\begin{itemize}
 	  		 			
 	  		 			\item
 	  		 		  $\{a_2, a_{24}, w_{2,34}\}$;
 	  		 		  
 	  		 		  	\item 
 	  		 		  	  $\{a_2, a_{24}, w_{4,23}\}$;
 	  		 		  	  
 	  		 		  	  \item 	  $\{a_2, a_{24}, y_4\}$
 	  		 			\end{itemize}
 	  		 			 in Figure 	\ref{figure:a24linkin543}.

 	  		 		 	\end{proof}

 	  	  	\begin{lemma} \label{lemma:1neighborhood}
 	  	  		  For the vertices $v_2$ and $a_{24} $ in $ \mathcal{C}_{57}\subset \mathcal{C}_{600}$, 
 	  	  consider the boundaries  $\partial N_1(v_2)$ and $\partial N_1(a_{24})$  of the 1-neighborhoods of  $v_2$ and $a_{24}$ in   $\mathcal{C}_{600}$. Then $$\partial N_1(v_2)^{(0)}\cap \partial N_1(a_{24}) ^{(0)}=\{a_4,v_1,v_3\}.$$ Moreover, there are exactly three pairs of vertices $(E,G)$ in $$ (\partial N_1(v_2)^{{(0)}} \setminus \{a_4,v_1,v_3\}) \times  (\partial N_1(a_{24}) ^{(0)} \setminus \{a_4,v_1,v_3\}),$$ say $(E,G)$ be one of $(v_4,a_2)$, $(a_{34},w_{4,23})$  and $(a_{14},w_{4,12})$,  such that the distances in  $\mathcal{C}_{600}^{(1)}$ satisfy:
 	  	 		\begin{itemize}
 	  	 		
 	  	 		\item $\text{dist} (E,G)=1$;

 	  	 			\item $\text{dist} (E, a_{24}) \geq 2$;
 	  	 			
 	  	 			\item $\text{dist}(G, v_{2}) \geq 2$.
 	  	 			
 	  	 			\end{itemize}

 	  	 \end{lemma}

 	  	 \begin{proof}
 	  	 	First note that $ \partial N_1(v_2)^{{(0)}}$ is the set of twelve vertices $$\{a_1,a_3,a_4, a_{13},a_{34},a_{14}, w_{1,34},w_{3,14},w_{4,13},v_1,v_3,v_4\},$$
 	  	 	and 
 	  	 	 $ \partial N_1(a_{24})^{(0)}$  is the set of twelve vertices $$\{v_1,a_2,v_3, a_4,w_{2,14},w_{2,34},w_{4,23},w_{4,12}, y_1,y_2,y_3,y_4\}.$$ 
Where $y_i$ lies in the interior of  $\mathcal{C}_{543}$, we label them as in  Figure 	\ref{figure:a24linkin543}. 
So $\partial N_1(v_2)^{(0)}\cap \partial N_1(a_{24}) ^{(0)}$ is the set $\{a_4,v_1,v_3\}$. 
  There are two tetrahedra in $\mathcal{C}_{600}$ with vertex sets  $\{v_2,a_4,v_1,v_3 \}$ and  $\{a_{24},a_4,v_1,v_3 \}$, respectively. 
  	 	 Since $\mathcal{C}_{600}$ is the boundary of a regular polytope, there exists a $\mathbb{Z}_3$-symmetry $f$ of  $\mathcal{C}_{600}$ that fixes $v_2$ and acts non-trivially  on the elements of $\{a_4,v_1,v_3\}$. 
  	 	 In turn, this map $f$ also fixes $a_{24}$.
 From Figure 	\ref{figure:a24linkin543}, we divide the nine vertices $\partial N_1(a_{24}) ^{(0)} \setminus \{a_4,v_1,v_3\}$  into three layers according the distances of them  in $\partial N_1(a_{24}) ^{(0)}$ to the set $\{a_4,v_1,v_3\}$:
  	 		\begin{itemize}

  	 			\item  $\{a_2,w_{4,23},w_{4,12}\}$;
  	 			
  	 		\item  $\{w_{2,34},y_3, w_{2,14}\}$;
  	 		
  	 			\item  $\{y_1,y_{2}, y_{4}\}$.

  	 		\end{itemize}
  	 		Similarly, we divide the nine vertices $\partial N_1(v_{2}) ^{(0)}  \setminus \{a_4,v_1,v_3\}$  into three layers according the distances of them  in $\partial N_1(v_{2}) ^{(0)}$ to the set $\{a_4,v_1,v_3\}$:
  	 		\begin{itemize}

  	 		\item  $\{v_4,a_{34},a_{14}\}$;
  	 		
  	 		\item  $\{a_{1},a_3, w_{4,13}\}$;
  	 		
  	 		\item  $\{w_{3,14}, w_{1,34}, a_{13}\}$.

  	 	\end{itemize}

 	  	 	Let  $(E,G)$ be the pair  in $$ (\partial N_1(v_2)^{(0)} \setminus \{a_4,v_1,v_3\}) \times  (\partial N_1(a_{24})^{(0)} \setminus \{a_4,v_1,v_3\}),$$ that satisfies the distance conditions in  Lemma \ref{lemma:1neighborhood}:

 	  	 		\begin{enumerate}

 	  	 	\item 	
 	  	 	
 	  	 		\begin{itemize}

 	  	 				\item 
 	  	 				
 	  	 				If $E=v_4 \in \mathcal{C}_{57}^{(0)}$ and $G$ is any element of $$\{a_2,w_{2,14},w_{2,34},w_{4,23},w_{4,12} \}\subset \partial \mathcal{C}_{57}^{(0)}=\partial \mathcal{C}_{543}^{(0)}.$$  From an elementary observation, the distance between $E$ and $G$ in $\mathcal{C}_{57}^{(1)}$ is equal to 2, except for the pair $(E,G) =(v_4,a_2)$, whose distance in $\mathcal{C}_{57}^{(1)}$ is 1.
 	  	 				Since $\mathcal{C}_{600}$  is a triangulation of the 3-sphere, the underlying spaces of the subcomplexes  $\mathcal{C}_{57}$  and $\mathcal{C}_{543}$ are embedded 3-balls with identified boundaries. There is no edge in $\mathcal{C}_{543}$   connecting the vertex  $v_4$ to $w_{2,14}$ (or similarly to  $w_{2,34}$, $w_{4,23}$ and $w_{4,12}$, respectively).
 	  	 				Thus, the distance between $v_4$ and $G$   in  $ \mathcal{C}_{600}$ coincides with that in $\mathcal{C}_{57}^{(1)}$: it equal $2$ in all cases, except for the pair $(v_4,a_2)$, which has distance 1.
 	  	 				Now it is  straightforward to check that $(E,G)=(v_4,a_2)$  satisfies the distance conditions in  Lemma \ref{lemma:1neighborhood};

 	  	 	\item 	 If $E=v_4$ and $G$ is any element of $\{y_1,y_{2},y_{3},y_{4} \}$ that lies in the interior of $\mathcal{C}_{543}$, then the distance between them in $\mathcal{C}_{600}$ is at least 2;

 	  	 	\item 		By the $\mathbb{Z}_3$-symmetry $f$,  the pairs $(E,G)=(a_{34},w_{4,23})$  and $(E,G)=(a_{14}, w_{4,12})$ also satisfy the distance conditions in  Lemma \ref{lemma:1neighborhood};
 	  	 		
 	  	 				\end{itemize}
 	  	 					
 	  	 \item

 	  	 		\begin{itemize}

 	  	 		\item If $E$ is any of $$\{a_1,a_3,w_{4,13}\} \subset \partial \mathcal{C}_{57}^{(0)}= \partial \mathcal{C}_{543}^{(0)}$$
 	  	 		and $G$ is any of  $$\{w_{2,34},w_{2,14} \}\subset \partial \mathcal{C}_{57}^{(0)}= \partial \mathcal{C}_{543}^{(0)}.$$ 
 	  	 		By observation, the distance between $E$ and $G$ in $\mathcal{C}_{57}^{(1)}$ at least  $2$. Since  $\mathcal{C}_{600}$ is a triangulation of the 3-sphere, there is no edge connecting any vertex $v_i$ ($i=1,2,3,4$) in the interior of  $ \mathcal{C}_{57}$ to a vertex in the interior of  $ \mathcal{C}_{543}$.   By Lemma 2.10 in \cite{PrzytyckiS:2009}, $\partial \mathcal{C}_{543}$  is a full subcomplex of $\mathcal{C}_{543}$.   Since $E$  and $G$  lie in $\partial \mathcal{C}_{57}=\partial \mathcal{C}_{543}$, the distance between  $E$  and $G$ in $ \mathcal{C}_{600}$ coincides with that in  $\mathcal{C}_{57}^{(1)}$: it is  at least 2  in all cases.

 	  	 	\item 	By the $\mathbb{Z}_3$-symmetry $f$, 	if $E$ is any of $\{a_1,a_3,w_{4,13}\} $
 	  	 		and $G=y_{3}$,  the distance between  $E$  and $G$ in $ \mathcal{C}_{600}$ is  at least 2;

 	  	 			\item

 	  	 				If $E=a_1 $, $G=y_1$ and $\text{dist} (a_1, y_{i})=1$; From Figure \ref{figure:a24linkin543}, there is an empty square $(a_1,y_1,a_2,v_3,a_1)$ in $\mathcal{C}_{600}$, which is a contradiction;
 	  	 		
 	  	 			If $E=a_3 $, $G=y_1$ and  $\text{dist} (a_3, y_{1})=1$; Then there is an empty square $(a_3,y_1,a_{24},v_1,a_3)$ in $\mathcal{C}_{600}$, which is a contradiction;
 	  	 			
 	  	 			If $E=w_{4,13}$, $G=y_1$  and  $\text{dist} (w_{4,13}, y_{1})=1$; Then there is an empty square $(w_{4,13},y_4,a_{24},v_1,w_{4,13})$ in $\mathcal{C}_{600}$, which is a contradiction;

 	  	 				\item  	
 	  	 					By the $\mathbb{Z}_3$-symmetry $f$, the distance between  $E \in \{a_1,a_3,w_{4,13}\}$  and $G \in \{y_2,y_4\}$ is  at least 2;

 	  	 	\end{itemize}

 	 \item 	 \begin{itemize}

 	  	 		\item 	If $E=a_{13}$,  $G=y_1$, and  $\text{dist} (a_{13}, y_{1})=1$. From Figure \ref{figure:a24linkin543},  there is an empty square $(a_{13},y_1,a_2,v_4,a_{13})$ in $\mathcal{C}_{600}$, which is a contradiction;
 	  	 		
 	  	 		\item 	If $E=a_{13}$ and $G \in \{y_2,y_4\}$. We consider the 1-neighborhood of $a_{13}$ and $a_{24}$ in $\mathcal{C}_{600}$.
   	 			 $ \partial N_1(a_{13})^{(0)}$  is the set of twelve vertices $$\{a_1,v_2,a_3, v_4,w_{1,34},w_{1,23},w_{3,14},w_{3,12}, z_1,z_2,z_3,z_4\}.$$ 
   	 			Where $z_i$ lies in the interior of  $\mathcal{C}_{543}$, comparing to  Figure 	\ref{figure:a24linkin543} for the link of $a_{24}$.
   	 			We label  $z_i$ such that there is an edge $[z_i,z_j]$ in $ \partial N_1(a_{13})$  except when $\{i,j\}=\{1,3\}$, just as there is an edge $[y_i,y_j]$ in $ \partial N_1(a_{24})$  except when$\{i,j\}=\{1,3\}$.
   	 			We consider  $\partial N_1(a_{13}) \cap \partial N_1(a_{24})$. By all the argument above,  $y_1$ and $y_3$ ($z_1$ and $z_3$, respectively) do not lie in $\partial N_1(a_{13}) \cap \partial N_1(a_{24})$.  There is an order two  symmetry $g$  of $\mathcal{C}_{600}$ such that $g(v_1)=v_3$ and $g(v_2)=v_4$, which induces an order two isomorphism  on $\partial N_1(a_{24})$ such that $g(y_2)=y_4$. If $y_2 \in \partial N_1(a_{13}) \cap \partial N_1(a_{24})$.
   	 			By the  symmetry $g$,  $y_4 \in \partial N_1(a_{13}) \cap \partial N_1(a_{24})$, and $\{y_2,y_4\}=\{z_2,z_4\}$. 
   	 			But then there are at least six triangles in  $\mathcal{C}_{600}$  with one  edge $[y_2,y_4]$ and the remaining vertices $z_1,y_1,a_{24},z_3,y_3$ and $a_{13}$.
   	 	Which is a contradiction that the edge link in  $\mathcal{C}_{600}$  of any edge is a pentagon. 	So $\partial N_1(a_{13}) \cap \partial N_1(a_{24})=\emptyset$. We have  $\text{dist} (a_{13}, y_{i})\geq 2$ for $i=2,4$;

 	  	 			\item 	By the $\mathbb{Z}_3$-symmetry $f$, the distance between  $E \in \{w_{1,34},w_{3,14}\}$  and $G \in \{y_1,y_2,y_4\}$ is  at least 2.

 	  	 			\end{itemize}

 	  	 		\end{enumerate}

 In conclusion, there are exactly three pairs $(E,G)$ satisfying the distance conditions in Lemma \ref{lemma:1neighborhood}, as claimed.

 	  	 		 \end{proof}

 	\begin{lemma} \label{lemma:fnsC1137square}

 	The 3-complex $\mathcal{C}_{1137}=\mathcal{C}_{4} \cup \mathcal{C}_{590}\cup \overline{\mathcal{C}}_{543} $ is not square-free; there are exactly twelve empty squares in $\mathcal{C}_{1137}$:

 		 	\begin{enumerate}[label=(\arabic*).]
 				
 				\item They lie in  $\mathcal{C}_{4} \cup (\mathcal{C}_{590} \cap \overline{\mathcal{C}}_{543})$:
 				
 				\begin{enumerate}[label=(\arabic{enumi}.\arabic*)]
 					
 					\item  \label{fnsC1137square11} $(a_{23},a_{24},w_{4,23},w_{3,24},a_{23})$;
 					
 					\item   \label{fnsC1137square12} $(a_{24},a_{34},w_{3,24},w_{2,34}, a_{24})$;
 					
 					\item   \label{fnsC1137square13} $(a_{34},a_{23},w_{2,34},w_{4,23}, a_{34})$.
 				\end{enumerate}

 				\item They lie in  $\mathcal{C}_{4} \cup \overline{\mathcal{C}}_{543} \subset  \mathcal{C}_{1137}$, but are not entirely contained  in    $\mathcal{C}_{4} \cup (\mathcal{C}_{590} \cap \overline{\mathcal{C}}_{543})$:

 			\begin{enumerate}[label=(\arabic{enumi}.\arabic*)]
 					
 					\item   \label{fnsC1137square21} $(a_{23},a_{24},\overline{y}_{1},\overline{x}_{1},a_{23})$;
 					
 					\item   \label{fnsC1137square22} $(a_{24},a_{34},\overline{y}_{3}, \overline{x}_{3}, a_{24})$;
 					
 					\item   \label{fnsC1137square23} $(a_{34},a_{23},\overline{m}_{1},\overline{n}_{1}, a_{34})$.
 				\end{enumerate}
 			Here, \(\overline{y}_{1}, \overline{x}_{1}, \overline{y}_{3}, x_{3}, \overline{m}_{1}, \overline{n}_{1}\) are distinct vertices in the interior of  \(\overline{\mathcal{C}}_{543} \subset \overline{\mathcal{C}}_{600}\) that do not belong to \(\overline{\mathcal{C}}_{57}\).
 			 The vertices \(y_1\) and \(y_3\) in \(\mathcal{C}_{543}\) are labeled as in Figure~\ref{figure:a24linkin543}; hence, the vertices \(\overline{y}_1\) and \(\overline{y}_3\) in \(\overline{\mathcal{C}}_{543}\) are labeled similarly.

 				\item 	They  lie in  $\mathcal{C}_{4} \cup \mathcal{C}_{590} \subset  \mathcal{C}_{1137}$, but are not entirely contained  in    $\mathcal{C}_{4} \cup (\mathcal{C}_{590} \cap \overline{\mathcal{C}}_{543})$:

 					\begin{enumerate}[label=(\arabic{enumi}.\arabic*)]
 					
 					\item  \label{fnsC1137square31}  $(a_{23},a_{24},y_{1},x_{1},a_{23})$;
 					
 					\item  \label{fnsC1137square32}  $(a_{24},a_{34},y_{3},x_{3}, a_{24})$;
 					
 					\item  \label{fnsC1137square33}  $(a_{34},a_{23},m_{1},n_{1}, a_{34})$;

 					\item  \label{fnsC1137square34} $(a_{23},a_{24},v_{3},v_{4},a_{23})$;
 					
 					\item   \label{fnsC1137square35} $(a_{24},a_{34},v_{2},v_{3}, a_{24})$;
 					
 					\item   \label{fnsC1137square36} $(a_{34},a_{23},v_{4},v_{2}, a_{34})$.
 					
 				\end{enumerate}	
 				Here,  $y_{1},x_{1}$, $y_{3},x_{3}$, $m_{1},n_{1}$ are  distinct   vertices in the interior of  $\mathcal{C}_{543}  \subset \mathcal{C}_{600}$ that do not belong to   $\mathcal{C}_{57}$.  The vertices \(y_1\) and \(y_3\) in \(\mathcal{C}_{543}\) are labeled as in Figure~\ref{figure:a24linkin543}.

 					\end{enumerate}

 			Moreover, for each empty square $(C_1,C_2,C_3,C_4,C_1)$ in $\mathcal{C}_{1137}$, we may reorder the vertices so that $[C_1, C_2] \subset \partial \mathcal{C}_{1137}$ and $C_3, C_4 \in \operatorname{int}(\mathcal{C}_{1137})$. Furthermore, any two such empty squares intersect in at most one common edge of the form \([a_{ij},a_{ik}]\) or \([a_{ij},v_{k}]\), where \(\{i,j,k\}=\{2,3,4\}\).

 	\end{lemma} 
 
  \begin{proof} 
  	
  Firstly, the three empty squares  in    $\mathcal{C}_{1133}$ shown  in	 Lemmas \ref {lemma:squareC590} and \ref{lemma:squareC1133} are not empty in  $\mathcal{C}_{1137}$. For example, there is an edge $[a_{23}, a_{24}]$ in $\mathcal{C}_{1137}$, so the circle  $(v_1, a_{23}, w_{2,34},a_{24}, v_1)$ in not full.

  		Suppose  $(C_1, C_2, C_3,C_4, C_1)$ is an empty square in $\mathcal{C}_{1137}$. As in the proof of Lemma \ref{lemma:fnsC1137simplex}, without loss of generality, we may assume $C_1=a_{23}=\overline{a}_{23}$,  and  $C_2=a_{24}=\overline{a}_{34}$.

  		\begin{enumerate}

  		\item  We first assume $C_3 \in c_{10}$:
  			\begin{itemize}
  			
  			\item
  			 It is  not difficult to see  $C_3$ can not be any of $\{a_2,a_3,a_{34}\}$;
  		\item We also have  $C_3\neq a_4$. Otherwise, $C_4$ would be either $v_1$ or $a_{34}$.
  			In either case, there is an edge between $C_2$ and $C_4$, which contradicts the fact that $(C_1, C_2, C_3,C_4, C_1)$ is an empty square in $\mathcal{C}_{1137}$;

  			\item 	Moreover, if  $C_3= w_{4,23}$, then $C_4=w_{3,24}$. There is no edge between $a_{23}$ and  $w_{4,23}$  (or between $a_{24}$ and $w_{3,24}$) in $\mathcal{C}_{1137}$, so $(a_{23}, a_{24}, w_{4,23}, w_{3,24}, a_{23})$ is an empty square in $\mathcal{C}_{1137}$.
  			\end{itemize}

  \item 	 We may now  assume both $C_3$  and $C_4$ lie in $\mathcal{C}_{1133}\setminus c_{10}$.   Since there is an edge $[C_3, C_4]$, it cannot be the case that $C_3 \in \mathcal{C}_{590} \setminus  c_{10}$ and $C_4 \in \overline{\mathcal{C}}_{543} \setminus  c_{10}$ (and vice versa).
   We further assume both    $C_3$ and $C_4$  lie in $\overline{\mathcal{C}}_{543}$.

   Recall that for a vertex $a_{ij}$, $w_{ij,k}$, $y_i$, etc., in $\mathcal{C}_{600}$,  we denote by $\overline{a}_{ij}$, $\overline{w}_{ij,k}$ and $\overline{y}_i$ the corresponding vertex  in the copy $\overline{\mathcal{C}}_{600}$. 
   Since $C_3$ and $C_4$ lie in the 1-neighborhoods  of $\overline{a}_{24}$ and $\overline{a}_{23}$, respectively, we consider the corresponding
    1-neighborhoods  of $\overline{a}_{24}$ and $\overline{a}_{23}$.
     Note that $\{\overline{a}_{23}, \overline{w}_{2,34},\overline{a}_{24} \}$  is equivalent to $\{a_2, a_{24}, w_{4,23}\}$ in Lemma \ref{lemma:triplevertices}, which  is equivalent to $\{v_2, v_{1}, a_{24}\}$ in turn.
  	   From the proof of   Lemma 	\ref{lemma:1neighborhood}, we have $$ \partial N_1(\overline{a}_{23})^{(0)}\cap  \partial N_1(\overline{a}_{24})^{(0)}=\{\overline{v}_1,\overline{a}_2,\overline{w}_{2,34}\}.$$
  	   By  Lemma 	\ref{lemma:1neighborhood}, consider the 1-neighborhoods of $\{\overline{v}_1,\overline{a}_2,\overline{w}_{2,34} \}$ in $\partial N_1(\overline{a}_{24})$, as illustrated in a picture which is a copy of  Figure \ref{figure:a24linkin543}. The vertex  $C_3$ is one of $\{\overline{w}_{4,23},\overline{v}_3, \overline{y}_1\}$, where $\overline{y}_1 \in\overline{\mathcal{C}}_{543}$.
   Similarly, $C_4$ is one of $\{\overline{w}_{3,24},\overline{v}_4, \overline{x}_1\}$,  where $\overline{x}_1 \in \partial N_1(\overline{a}_{23}) \cap  \overline{\mathcal{C}}_{543}$:
   
   	\begin{itemize}
   	
   	\item 	Regarding the circle  $(a_{23},a_{24},\overline{v}_{3},\overline{v}_{4},a_{23})$, we observe that  neither   $\overline{v}_3$ nor $\overline{v}_4$ belongs to   $\overline{\mathcal{C}}_{543}$. Consequently,  the empty square  $(\overline{a}_{23}, \overline{a}_{24},\overline{v}_{3},\overline{v}_{4}, \overline{a}_{23})$ does not exist in  $\mathcal{C}_{1137}$.

   	\item 
   	The subcase  $C_3=\overline{w}_{4,23}$ leads to a contradiction to our assumption.
   	
   	\item 
   	  We  obtain an empty  square $(\overline{a}_{23},\overline{a}_{24},\overline{y}_{1},\overline{x}_{1},\overline{a}_{23})$, as neither  the edge $[\overline{a}_{23},\overline{y}_1]$ nor the edge $[\overline{a}_{24},\overline{x}_1]$ exists by Lemma 	\ref{lemma:1neighborhood}.

   		\end{itemize}

  		\item  	The case in which  both $C_3$ and $C_4$ lie in $\mathcal{C}_{590}$ is similar. We obtain  the empty square $(a_{23},a_{24},y_{1},x_{1},a_{23})$ as above. Moreover, 
  	 in this case both $v_3$ and $v_4$  lie in  $\mathcal{C}_{590}$, so we  also obtain an additional   empty square  $(a_{23}, a_{24},v_{3},v_{4},a_{23})$ in $\mathcal{C}_{4} \cup \mathcal{C}_{590} \subset \mathcal{C}_{1133}$.

  		 	\end{enumerate}		
  		 			
  		 	In conclusion, we obtain exactly four empty squares in \(\mathcal{C}_{1137}\), with \([a_{23},a_{24}]\) being one of their edges.
  		 	
  		 	Similarly, we have exactly eight additional empty squares in \(\mathcal{C}_{1137}\), with one of their edges being \([a_{23},a_{34}]\) or \([a_{24},a_{34}]\), as stated in the lemma.

 		From the above argument, it is readily seen that for each empty square in the lemma, we may reorder the vertices so that $[C_1, C_2] \subset \partial \mathcal{C}_{1137}$ and $C_3, C_4 \in \operatorname{int}(\mathcal{C}_{1137})$. Furthermore, one can check directly that any two such empty squares intersect in at most one common edge of the form \([a_{ij},a_{ik}]\) or \([a_{ij},v_{k}]\), where \(\{i,j,k\}=\{2,3,4\}\).

 \end{proof}

\begin{figure}
	\begin{center}
		\begin{tikzpicture}
		\node at (0,0) {\includegraphics[width=10cm,height=10cm]{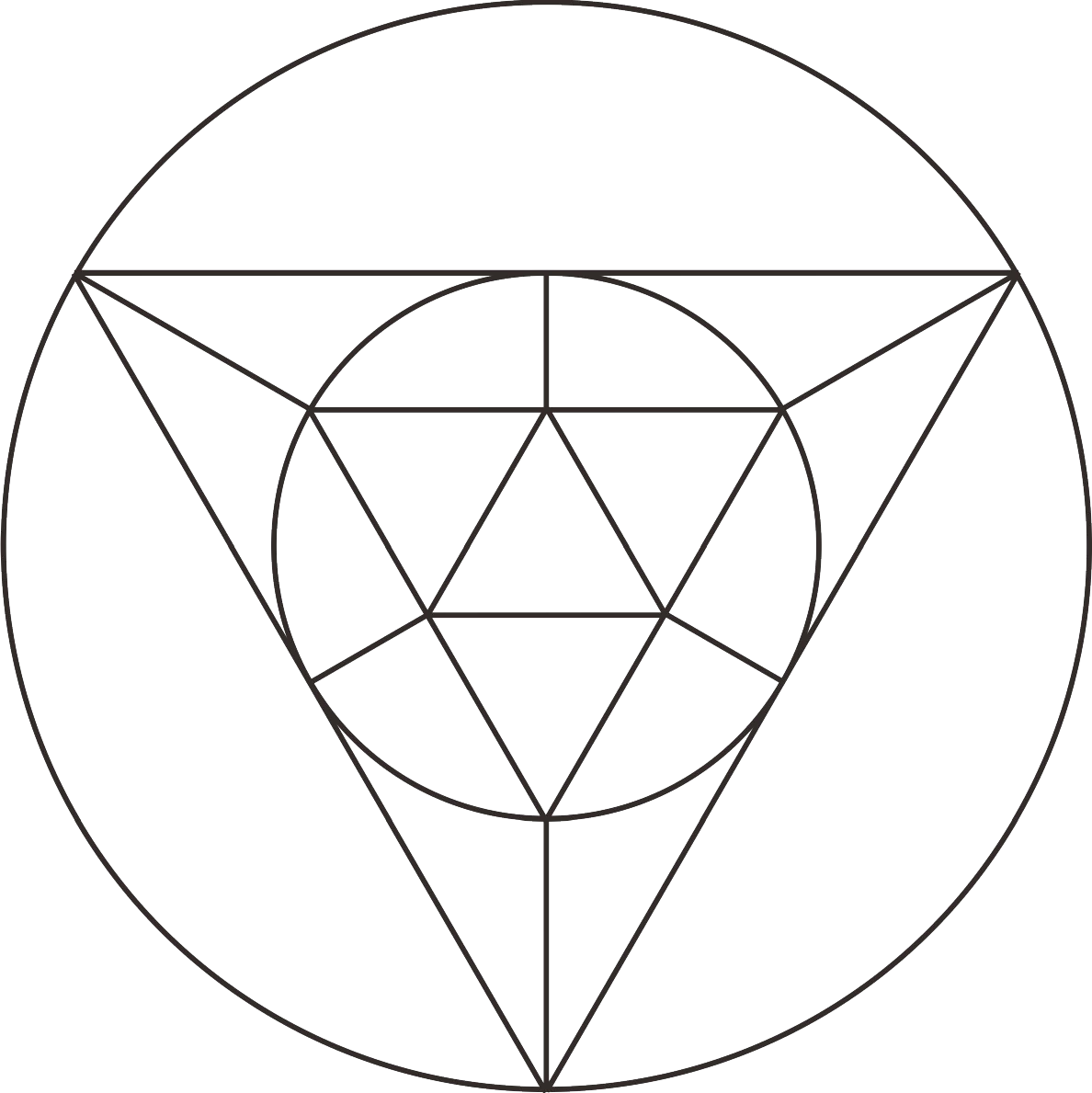}};

		\node at (0.4,-2.8){\small $a_{2}$};
		
		\node at (2.6,-1.5){\small $w_{2,34}$};
		
		\node at (2.5,1.8){\small $w_{4,23}$};

		\node at (0.4,1.5){\small $a_{4}$};
		
		\node at (-2.6,-1.5){\small $w_{2,14}$};
		
		\node at (-2.3,1.8){\small $w_{4,12}$};

		\node at (1.35,-1.1){\small $v_{1}$};
		
			\node at (-1.35,-1.1){\small $v_{3}$};
		
		\node at (0.3,-5.3){\large $y_{1}$};
		
			\node at (4.8,3.0){\large $y_{2}$};
			
				\node at (0.2,3.0){\large $y_{3}$};
				
					\node at (-4.8,3.10){\large $y_{4}$};

		\end{tikzpicture}
	\end{center}
	\caption{The link of the vertex  $a_{24}$ in the 3-complex $\mathcal{C}_{600}$  is an icosahedron.  Here, for example, the vertex labeled by $a_2$ corresponds to the edge $[a_{24},a_2]$ in  $\mathcal{C}_{600}$. Note that the vertices $v_1$ and $v_3$ lie in the interior of the 3-complex  $\mathcal{C}_{57}$.}
	\label{figure:a24linkin543}
\end{figure}

	\subsection{The 3-complex $\mathcal{C}_{1163}$}  \label{subsection:c1163}

 	In analogy to the construction  of $\mathcal{C}_{1137}$, we obtain the 3-complex $\mathcal{C}_{1163}$. Consider the cone $\mathcal{C}_{30}$ with apex at $v_1$, constructed over the 2-complex $c_{30}$ illustrated in Figure \ref{figure:c30}. Its lateral surface admits the triangulation $c_6$ of a 2-disk.  We glue the cone  $\mathcal{C}_{30}$ to $\mathcal{C}_{1133}$ along the 2-complex $c_6$. This yields a triangulated 3-ball consisting of 1163 tetrahedra, which we denote by $\mathcal{C}_{1163}$.  Accordingly, the boundary of $\mathcal{C}_{1163}$ consists of four parts: three copies of $c_{10}$ and one copy of $c_{30}$.
 		Moreover, $\overline{a}_1,a_2=\overline{a}_2,a_3=\overline{a}_3,a_4=\overline{a}_4$ serve as the \emph{genuine vertices} of  $\mathcal{C}_{1163}$. 
 	
 	Moreover, since the cone \(\mathcal{C}_{30}\) is ``bigger'' than the cone \(\mathcal{C}_{4}\), the properties of the 3-complex \(\mathcal{C}_{1163}=\mathcal{C}_{30}\cup\mathcal{C}_{1133}\) are finer than those of the 3-complex \(\mathcal{C}_{1137}=\mathcal{C}_{4}\cup\mathcal{C}_{1133}\). In comparison with Lemmas~\ref{lemma:fnsC1137simplex} and~\ref{lemma:fnsC1137square}, we conclude that
 
 	 	 \begin{lemma} \label{lemma:fnsC1163}  The 3-complex $\mathcal{C}_{1163}=\mathcal{C}_{30} \cup \mathcal{C}_{590}\cup \overline{\mathcal{C}}_{543} $ is simplicial, it is a flag triangulation of the 3-ball. 	Nevertheless, $\mathcal{C}_{1163}$ is not square-free; the three empty squares contained  in $\mathcal{C}_{590}$ and $\mathcal{C}_{1133}$ (as established in Lemmas \ref{lemma:squareC590} and \ref{lemma:squareC1133}) are the only  empty squares in  $\mathcal{C}_{1163}$. 	Moreover, for each empty square $(C_1,C_2,C_3,C_4,C_1)$ in $\mathcal{C}_{1163}$, we may reorder the vertices so that $C_1, C_3 \in \partial \mathcal{C}_{1163}$ and $C_2, C_4 \in \operatorname{int}(\mathcal{C}_{1163})$.

 	\end{lemma} 
 	\begin{proof} As in the case of  of $\mathcal{C}_{1137}$, one can check directly that  $\mathcal{C}_{1163}$ is a simplicial 3-complex.

 		 Note that  $\mathcal{C}_{1163}$ 
 		contains  thirteen  additional  vertices compared with   $\mathcal{C}_{1133}$. We denote the vertices in $\mathcal{C}_{1163} \setminus \mathcal{C}_{1133}$ by $E_{i}$  for $1 \leq i \leq 13$. Moreover, $\mathcal{C}_{1163}$  contains  $ 55=42+13$ additional edges relative to   $\mathcal{C}_{1133}$, and each new edge has at least  one-point among   the vertices $E_i$.

 			Let  $\{C_1, C_2,\cdots,  C_k\}$ be a collection of vertices in $$\mathcal{C}_{1163}=\mathcal{C}_{30} \cup_{c_{6}} \mathcal{C}_{1133}$$ such that these vertices are pairwise  adjacent by edges.
 			Since  $\mathcal{C}_{1133}$ is flag by Lemma  \ref{lemma:squareC1133}, we may assume $C_1=E_1 \in \mathcal{C}_{1163} \setminus \mathcal{C}_{1133}$. There are no edges from  $E_1$ to $\mathcal{C}_{1133}\setminus c_6$, so $\{C_1, C_2,\cdots,  C_k\}\subset \mathcal{C}_{30}$. 
 				 Since $\mathcal{C}_{30}$ is a cone  with apex at $v_1$ over  the flag 2-complex $c_{30}$,  it is a flag triangulated  3-ball. Hence  $\{C_1, C_2,\cdots,  C_k\}$ spans a $(k-1)$ simplex in $c_{30}$, which implies the 3-complex $\mathcal{C}_{1163}$  is  flag.

The three empty squares  in    $\mathcal{C}_{1133}$ obtained  in	 Lemma \ref{lemma:squareC1133} remain  empty in  $\mathcal{C}_{1163}$, since neither the edge $[v_1,w_{2,34}]$ nor the edge  $[a_{23},a_{24}]$ exists in $\mathcal{C}_{1163}$.

	Suppose  $(C_1, C_2, C_3,C_4, C_1)$ is an empty square in $\mathcal{C}_{1163}$. Without loss of generality, we may assume $C_1=E_1 \in \mathcal{C}_{1163} \setminus \mathcal{C}_{1133}$. Moreover, it is readily seen that  $\mathcal{C}_{30}$ is square-free and contains no empty squares, so $\{C_1, C_2, C_3, C_4\} \not\subset \mathcal{C}_{30}$. Thus we obtain   $C_3 \in  \mathcal{C}_{1133} \setminus c_6$.  In turn $C_2, C_4  \in c_6 \setminus v_1$. Then, up to re-ordering,  we may assume $\{C_2,C_4\}$ is $\{a_2,a_4\}$, $\{a_{23},a_{24}\}$ or  $\{a_{2},a_{34}\}$. In each case, there is no vertex $E_i  \in \mathcal{C}_{1163} \setminus \mathcal{C}_{1133}$ connected to both  $C_2$ and $C_4$. Therefore,  there are no additional empty squares other than  the three empty squares  contained in    $\mathcal{C}_{1133} \subset \mathcal{C}_{1163}$.

 	\end{proof}

In accordance with Lemma~\ref{lemma:interiorC543}, we have the following lemma, Lemma~\ref{lemma:halfsquare}, which will be used in the proof of Theorem~\ref{thm:almostfns} when studying the empty squares in the relatively special subdivision, since each triple $\{E,G,F\}$ in Lemma~\ref{lemma:halfsquare} may contribute ``half of an empty square''.

 	 	 	 \begin{lemma} \label{lemma:halfsquare}  Let \(\Delta\) be a \(3\)-simplex, and let \(\Delta^{*}\) be a subdivision of \(\Delta\) into one of the \(3\)-complexes \(\mathcal{C}'_{543}\), \(\mathcal{C}_{1137}\) and \(\mathcal{C}_{1163}\), where the four vertices of \(\Delta\) are mapped to the genuine vertices of \(\Delta^{*}\). Then  there does not exist a triple of vertices \(\{E,G,F\}\) of \(\Delta^*\) satisfying the following conditions:
 	 	 \begin{enumerate}
 	 	 	 		\item \label{halfsquare:1} \(E\) lies in the interior of \(\Delta\);
 	 	 	 		\item \label{halfsquare:2} There exists a 2-face \(\delta\) of \(\Delta\) such that \(G\) and \(F\) lie in \(\delta\);
 	 	 	 		\item  \label{halfsquare:3} \(\Delta^*\) contains the edges \([E,G]\) and \([E,F]\), but does not contain the edge \([G,F]\).
 	  	 \end{enumerate}
 	 	 	 	The only exceptions are the following cases:
 	 	 	 	\begin{enumerate}[label=(\alph*)]
 	 	 	 			
 	 	 	 			\item \label{halfsquare:a} \(\Delta^{*}\) is isomorphic to \(\mathcal{C}'_{543}\). \(G\) and \(F\) lie in the \(2\)-face \(\delta\) of \(\Delta\), and the subdivision of \(\delta\) induced by \(\mathcal{C}'_{543}\) is isomorphic to the \(2\)-complex \(c_{30}\). Moreover, there are exactly fifteen such triples.

 	 	 	 				\item \label{halfsquare:b} \(\Delta^{*}\) is isomorphic to \(\mathcal{C}_{1137}=\mathcal{C}_{4}\cup\mathcal{C}_{590}\cup\overline{\mathcal{C}}_{543}\). \(\{\,E,G,F\,\}\) is \(\{\,v_{1},a_{i},a_{j}\,\}\) or \(\{\,v_{1},a_{i},a_{jk}\,\}\) in \(\mathcal{C}_{4}\) for \(\{i,j,k\}=\{2,3,4\}\). Here remember that the cone \(\mathcal{C}_{4}\) is based on the \(2\)-complex \(c_{4}\) to \(v_{1}\);

 	 	 	 		\item  \label{halfsquare:c} \(\Delta^{*}\) is isomorphic to \(\mathcal{C}_{1163}=\mathcal{C}_{30}\cup\mathcal{C}_{590}\cup\overline{\mathcal{C}}_{543}\), and
 	 	 	 		\begin{itemize}
 	 	 	 			\item \(\{E,G,F\}\) is \(\{w_{2,34},a_{23},a_{24}\}\), \(\{w_{3,24},a_{23},a_{34}\}\), or \(\{w_{4,23},a_{24},a_{34}\}\);
 	 	 	 			\item \(\{E,G,F\}=\{v_{1},G,F\}\subset\mathcal{C}_{30}\) for some pair of vertices \(G\) and \(F\) in the \(2\)-complex \(c_{30}\). Here remember that the cone \(\mathcal{C}_{30}\) is based on the \(2\)-complex \(c_{30}\) to \(v_{1}\).
 	 	 	 		\end{itemize}

\end{enumerate}

 	\end{lemma} 
 	\begin{proof}

 		 	First we assume $\Delta^*$ is isomorphic to   $\mathcal{C}'_{543}$:
 		 		\begin{itemize}	
 		 		\item 
 		 		From Figure~\ref{figure:C57},
 		 		the special subdivision \(c_{10}\) for the triangle with vertex set \(\{A=a_{2},C=a_{4},D=a_{3}\}\) in \(\mathcal{C}_{543}\) is the union of three parts of the boundary of \(\mathcal{C}'_{543}\), say certain subdivisions of the triangles with vertex sets \(\{A,B,C\}\), \(\{A,B,D\}\), and \(\{B,C,D\}\), respectively. By Item~(\ref{interiorC543:1}) of Lemma~\ref{lemma:interiorC543}, there is no triple of vertices \(\{E,G,F\}\) in \(\mathcal{C}'_{543}\) satisfying Items~(\ref{halfsquare:1}), (\ref{halfsquare:2}), and (\ref{halfsquare:3}) of Lemma~\ref{lemma:halfsquare} when \(G,F\) lie in these three parts of the boundary of \(\mathcal{C}'_{543}\);

 		 	\item So the remaining case is that there is a triple of vertices \(\{E,G,F\}\) satisfying the three conditions of Lemma~\ref{lemma:halfsquare}, such that \(G\) and \(F\) lie in the subdivision \(c_{30}\) for the triangle with vertex set \(\{a_{2},a_{3},a_{4}\}\). We note here that such triples of vertices really do exist. This can be seen by a careful inspection of the facets \(F_{1},F_{2},F_{30},F_{31}\), and \(F_{5}\) of the \(120\)-cell in the notation of Ma--Zheng \cite{Ma-Zheng:2021}.

 		 	Now we use the notations in Ma--Zheng \cite{Ma-Zheng:2021}, where \(F_{i}\), \(1 \leq i \leq 120\), are the facets of the 120-cell \(\mathcal{Y}_{120}\). Thus each facet \(F_i\) corresponds to a vertex \(C_i\) of the 600-cell \(\mathcal{C}_{600}\).

 		 We consider the 5-circles
 		 $$F_{i_1} \textbf{---} F_{i_2} \textbf{---} F_{i_3} \textbf{---} F_{i_4} \textbf{---} F_{i_5} \textbf{---} F_{i_1}$$
 		 such that \(F_{i_k}\) and \(F_{i_{k+1}}\) are adjacent, but \(F_{i_k}\) and \(F_{i_{k+2}}\) are not adjacent for each \(k\) mod \(5\).

 		 There are exactly five  5-circles	containing  the edge $F_{2} \textbf{---}F_{1}$:
 		 	\begin{equation}
 		 	\begin{aligned}
 		 	&F_{2} \textbf{---} F_{1} \textbf{---} F_{5} \textbf{---} F_{31} \stackrel{3}{\textbf{---}} F_{30} \textbf{---} F_{2};\\
 		 	&F_{2} \textbf{---} F_{1} \textbf{---} F_{7} \textbf{---} F_{22} \stackrel{3}{\textbf{---}} F_{14} \textbf{---} F_{2};\\
 		 	&F_{2} \textbf{---} F_{1} \textbf{---} F_{9} \textbf{---} F_{24} \stackrel{3}{\textbf{---}} F_{18} \textbf{---} F_{2};\\
 		 	&F_{2} \textbf{---} F_{1} \textbf{---} F_{11} \textbf{---} F_{16} \stackrel{3}{\textbf{---}} F_{26} \textbf{---} F_{2};\\
 		 	&F_{2} \textbf{---} F_{1} \textbf{---} F_{13} \textbf{---} F_{20} \stackrel{3}{\textbf{---}} F_{28} \textbf{---} F_{2}.
 		 	\end{aligned}
 		 	\label{eq:five-circles1}
 		 	\end{equation}

 	 There are exactly five  5-circles	containing  the edge $F_{3} \textbf{---}F_{2}$:
 	 
 	 	\begin{equation}
 	\begin{aligned}
 	& F_{3} \textbf{---}F_{2} \textbf{---}F_{10} \stackrel{3}{\textbf{---}} F_{7}  \textbf{---} F_{11} \textbf{---}F_{3}; \\     
 		 	 & F_{3} \textbf{---}F_{2} \textbf{---}F_{12} \stackrel{3}{\textbf{---}} F_{9}  \textbf{---} F_{13} \textbf{---}F_{3};\\
 		 	 		 	  & F_{3} \textbf{---}F_{2} \textbf{---}F_{14} \stackrel{3}{\textbf{---}} F_{38}  \textbf{---} F_{16} \textbf{---}F_{3};\\
 		 	 &   F_{3} \textbf{---}F_{2} \textbf{---}F_{18} \stackrel{3}{\textbf{---}} F_{40}  \textbf{---} F_{20} \textbf{---}F_{3};\\
 		 	 &   F_{3} \textbf{---}F_{2} \textbf{---}F_{34} \stackrel{3}{\textbf{---}} F_{48}  \textbf{---} F_{36} \textbf{---}F_{3}.
 		 	    	\end{aligned}
 		 	    \label{eq:five-circles2}
 		 	    \end{equation}
 		 	   
 		 	     There are exactly five  5-circles	containing  the edge $F_{3} \textbf{---}F_{1}$:
 	
 		 	\begin{equation}
 		 \begin{aligned}
 		     & F_{3} \textbf{---}F_{1} \textbf{---}F_{6}\textbf{---} F_{33}  \stackrel{3}{\textbf{---}} F_{32} \textbf{---}F_{3};\\
 		       & F_{3} \textbf{---}F_{1} \textbf{---}F_{7}\textbf{---} F_{23}  \stackrel{3}{\textbf{---}} F_{16} \textbf{---}F_{3};\\
 		           & F_{3} \textbf{---}F_{1} \textbf{---}F_{9}\textbf{---} F_{25}  \stackrel{3}{\textbf{---}} F_{20} \textbf{---}F_{3};\\
 		               & F_{3} \textbf{---}F_{1} \textbf{---}F_{10}\textbf{---} F_{14}  \stackrel{3}{\textbf{---}} F_{26} \textbf{---}F_{3};\\
 		                   & F_{3} \textbf{---}F_{1} \textbf{---}F_{12}\textbf{---} F_{18}  \stackrel{3}{\textbf{---}} F_{28} \textbf{---}F_{3}.
 		                  \end{aligned}
 		                  \label{eq:five-circles3}
 		                  \end{equation}

These circles can be obtained by a careful inspection of the adjacent relation of the facets of the 120-cell in Ma--Zheng \cite{Ma-Zheng:2021}, in particular, by Figures 12, 13, and 14 of \cite{Ma-Zheng:2021}. Here the label ``3'' in
\(F_{31} \stackrel{3}{\textbf{---}} F_{30}\)
will be used in Section \ref{section:conformaldim}.

Since facet \(F_i\) corresponds to a vertex \(C_i\) of the 600-cell \(\mathcal{C}_{600}\). We may assume  \(z_2=F_1\), \(z_3=F_2\), and \(z_4=F_3\).

Let \(\{E,G,F\}\) be a triple of vertices satisfying the three conditions of Lemma~\ref{lemma:halfsquare}, such that \(G\) and \(F\) lie in the subdivision \(c_{30}\) for the triangle with vertex set \(\{a_{2},a_{3},a_{4}\}\). We may assume that \(G\) and \(F\) lie in the subdivision \(c_{10}\) for the triangle with vertex set \(\{a_{1},a_{2},a_{4}\}\) and \(\{a_{1},a_{2},a_{3}\}\), respectively. So there is an edge between \(z_3=F_2\) and \(G\); similarly, there is an edge between \(z_4=F_3\) and \(F\). Note that there is no edge between \(E\) and \(z_3\), \(E\) and \(z_4\), \(G\) and \(z_4\), or \(F\) and \(z_3\). Hence
\[
z_4=F_{3} \textbf{---} z_3=F_{2} \textbf{---} G \textbf{---} E \textbf{---} F \textbf{---} F_{3}
\]
is one of the 5-circles above. We view $F_{6} \textbf{---} F_{33} \textbf{---} F_{32}$
as a ``half-square'' in \(\mathcal{C}'_{543}\). Similarly, there are fifteen ``half-squares''.

 		 	\end{itemize}

 Secondly,  $\Delta^*$ is isomorphic to   $\mathcal{C}_{1137}$:
 	\begin{itemize}	
 	\item 
 	Suppose that the vertices \(G\) and \(F\) lie in the \(2\)-face \(\delta\) of \(\Delta\), and that the subdivision of \(\delta\) induced by \(\Delta^*\) is isomorphic to the \(2\)-complex \(c_{10}\).
 	
 	\begin{itemize}
 		\item If the vertex \(E\) lies in the interior of the \(2\)-disk
 		\(\mathcal{C}_{590}\cap\overline{\mathcal{C}}_{543}\).
 		This contradicts Item~(\ref{interiorC543:2}) of Lemma~\ref{lemma:interiorC543};
 		
 		\item Then the vertex \(E\) lies in the interior of \(\Delta\).
 		This contradicts Item~(\ref{interiorC543:1}) of Lemma~\ref{lemma:interiorC543}.
 	\end{itemize}


 	\item Suppose that the vertices \(G\) and \(F\) lie in the \(2\)-face \(\delta\) of \(\Delta\), and that the subdivision of \(\delta\) is isomorphic to the \(2\)-complex \(c_{4}\).
 	
 	\begin{itemize}
 		\item If the vertex \(E=v_1\) lies in \(c_6=\mathcal{C}_{4}\cap \mathcal{C}_{590}\), then \(\{G,F\}\)
 		are some pair of vertices in \(c_{30}\);
 		
 		\item \(E\) lies in the interior of \(\mathcal{C}_{1133}=\mathcal{C}_{590} \cup \overline{\mathcal{C}}_{543}\).
 		
 		\begin{itemize}
 			\item \(E\) lies in the interior of \(\overline{\mathcal{C}}_{543}\);
 			
 			\item \(E\) lies in the interior of \(\mathcal{C}_{590}\). Then \(\{\,G,F\,\}\) is \(\{\,a_{i},a_{j}\,\}\) or \(\{\,a_{i},a_{jk}\,\}\) in \(c_{4}\) for \(\{i,j,k\}=\{2,3,4\}\);
 			
 			\item \(E\) lies in the interior of \(\mathcal{C}_{590} \cap \overline{\mathcal{C}}_{543}\). 
 		\end{itemize}
 	
 Then by careful inspection, each subcase contradicts Item~(\ref{interiorC543:1}) of Lemma~\ref{lemma:interiorC543}.

 	\end{itemize}

 		\end{itemize}

 	Thirdly,  $\Delta^*$ is isomorphic to   $\mathcal{C}_{1163}$:
 	
 		\begin{itemize}	
 			\item Suppose that the vertices \(G\) and \(F\) lie in the \(2\)-face \(\delta\) of \(\Delta\), and that the subdivision of \(\delta\) induced by \(\Delta^*\) is isomorphic to the \(2\)-complex \(c_{10}\).
 			
 			\begin{itemize}
 				\item If the vertex \(E\) lies in the interior of
 				\(\overline{\mathcal{C}}_{543}\), this contradicts Item~(\ref{interiorC543:1}) of Lemma~\ref{lemma:interiorC543}.

 				\item If the vertex \(E\) lies in the interior of
 				\(\mathcal{C}_{590}\), this contradicts Item~(\ref{interiorC543:1}) of Lemma~\ref{lemma:interiorC543}.
 				
 				\item \(E\) lies in the interior of the \(2\)-disk \(c_{10}=\mathcal{C}_{590} \cap \overline{\mathcal{C}}_{543}\). By Lemma~\ref{lemma:interiorC543}, we must have that \(E\) is one of \(w_{2,34}\), \(w_{3,24}\) and \(w_{4,23}\). We have that \(\{E,G,F\}\) is one of \(\{w_{2,34}, a_{23}, a_{24}\}\), \(\{w_{3,24}, a_{23}, a_{34}\}\) or \(\{w_{4,23}, a_{24}, a_{34}\}\). The author remarks that this subcase is different from \(\mathcal{C}_{1137}\); in \(\mathcal{C}_{1137}\), there are no triples of vertices  $\{E,G,F\}$ in this subcase.
 			\end{itemize}


 		\item Suppose that the vertices \(G\) and \(F\) lie in the \(2\)-face \(\delta\) of \(\Delta\), and that the subdivision of \(\delta\) induced by \(\Delta^*\) is isomorphic to the \(2\)-complex \(c_{30}\).
 		
 		\begin{itemize}
 			\item If the vertex \(E=v_1\) lies in \(c_{6}=\mathcal{C}_{30}\cap \mathcal{C}_{590}\), then \(\{G,F\}\)
 			are some pair of vertices in \(c_{30}\);
 			
 			\item \(E\) lies in the interior of \(\mathcal{C}_{1133}=\mathcal{C}_{590} \cup \overline{\mathcal{C}}_{543}\).
 			Then, similar to the case that \(\Delta^*\) is isomorphic to \(\mathcal{C}_{1137}\), this contradicts Item~(\ref{interiorC543:1}) of Lemma~\ref{lemma:interiorC543}.
 		\end{itemize}

 		\end{itemize}


  	\end{proof}


 	 	\section{Relatively special subdivisions with the almost flag-no-square property} \label{section:almostfns}

Let \(N\) be a 3-manifold admitting a triangulation \(\mathcal{T}\), and let \(\Gamma\) be a subgraph of \(\mathcal{T}\) as described in Lemma~\ref{lemma:inducedgraph}.
 	 In this section, 	we will construct a subdivision \(\mathcal{T}^*\) of \(\mathcal{T}\) from the 3-complexes considered in Section~\ref{section:almostfns3ball},   while preserving the embeddedness of \(\Gamma\).
 	 	Although the triangulation \(\mathcal{T}^*\) is not flag-no-square, it satisfies the property of being ``almost flag-no-square'',  which is enough for the proof of Theorem~\ref{thm:conformaldim}.
 	 		Theorem~\ref{thm:almostfns}  is the main technical result in the paper, which is inspired by Proposition~2.13 of Przytycki-\'Swiątkowski~\cite{PrzytyckiS:2009} and Lemma~5.2 of Cashen-Dani-Schreve-Stark~\cite{CDSStark:2025}.
 	 		Compared with the special subdivision of \(\mathcal{T}\) in \cite{PrzytyckiS:2009}, we call the triangulation \(\mathcal{T}^*\) a relatively special subdivision associated with \((\mathcal{T}, \Gamma)\).

  	 \subsection{Construction of the relatively special subdivision} \label{subsection:subdivision}
  	 
  	 Let \(N\) be a closed orientable 3-manifold admitting a triangulation \(\mathcal{T}\), and let \(\Gamma\) be a bipartite, flag-no-square, full subgraph of \(\mathcal{T}\) as described in Lemma~\ref{lemma:inducedgraph}.
  	 In this subsection, 	we will construct a subdivision \(\mathcal{T}^*\) of \(\mathcal{T}\).
  	 
  	 We denote the vertices of the  bipartite graph $\Gamma$ by $\{A_{i}\}^{m}_{i=1}$ and $\{B_{j}\}^{n}_{j=1}$.
  	 Vertices in $\mathcal{T}^{(0)} \setminus \Gamma^{(0)}$ will be denoted by $G, I, E, F, H$, etc.

  	 Regarding  the triangulation $\mathcal{T}$ of $N$ in Lemma \ref{lemma:inducedgraph}, and let $\Delta \in \mathcal{T}^{(3)}$ be a 3-simplex. Recall that if $\Gamma \cap \Delta$  is non-empty, then it  is either an edge or a vertex. For each $\Delta \in \mathcal{T}^{(3)}$, we replace $\Delta$ with a $3$-complex $\Delta^*$, which is a copy of one of the $3$-complexes $\mathcal{C}_{543}$, $\mathcal{C}'_{543}$, $\mathcal{C}_{1137}$, or $\mathcal{C}_{1163}$:

  	 \begin{itemize}

  	 	\item[\textbf{Replacement 1}] \label{almostfns:replace1}
  	 	If $\Delta$ is a tetrahedron in $\mathcal{T}^{(3)}$ with vertex set $\{G,I,E,F\}$ such that $\Delta \cap \Gamma = \emptyset$, then we replace $\Delta$ by a copy of $\mathcal{C}_{543}$, where $\{G,I,E,F\}$ are mapped onto the genuine vertices of $\mathcal{C}_{543}$;
  	 	

  	 	\item[\textbf{Replacement 2}] \label{almostfns:replace2}If $\Delta$ is a tetrahedron in $\mathcal{T}^{(3)}$ such that $\Delta \cap \Gamma = \{A_i,B_j\}$ is an edge, then we replace $\Delta$ by a copy of $\mathcal{C}'_{543}$ so that the vertices of $\Delta$ are mapped onto the genuine vertices of $\mathcal{C}'_{543}$. Moreover, we require that the vertex $A_i$ of $\Gamma$ is mapped to the genuine vertex $A$ of $\mathcal{C}'_{543}$, and that the vertex $B_j$ of $\Gamma$ is mapped to the genuine vertex $B$ of $\mathcal{C}'_{543}$;

  	 	\item[\textbf{Replacement 3}] \label{almostfns:replace3} If $\Delta$ is a tetrahedron in $\mathcal{T}^{(3)}$ such that  $\Delta \cap \Gamma=A_i$ is a vertex, and  none of the 2-faces  of $\Delta$ is adjacent to a tetrahedron $\Delta' \in \mathcal{T}^{(3)}$  with vertex set $\{A_i,B_j,H,F\}$  for some $B_j$ satisfying that the edge  $[A_i, B_j]$ lies in $\Gamma$, then  we still replace $\Delta$ with a copy of $\mathcal{C}_{543}$ such that the vertices of $\Delta$ are mapped to the genuine vertices of $\mathcal{C}_{543}$;

  	 	\item[\textbf{Replacement 4}] \label{almostfns:replace4} If $\Delta$ is a tetrahedron in $\mathcal{T}^{(3)}$ such that $\Delta \cap \Gamma=B_j$ is a vertex, and  none of the 2-faces  of $\Delta$ is adjacent to a tetrahedron $\Delta' \in \mathcal{T}^{(3)}$ with vertex set $\{B_j,A_i,H,F\}$ for some $A_i$ satisfying that the edge  $[B_j, A_i]$ lies in $\Gamma$, then we still replace $\Delta$ with a copy of $\mathcal{C}_{543}$ such that the vertices of $\Delta$ are mapped to the genuine vertices of $\mathcal{C}_{543}$;

  	 	\item[\textbf{Replacement 5}] \label{almostfns:replace5} If $\Delta$ is a tetrahedron in $\mathcal{T}^{(3)}$  such that $\Delta \cap \Gamma=A_{i}$ is a vertex, and   $\Delta$ admits a 2-face, namely the triangle $\delta$ with vertex set $\{A_{i}, G, I\}$,   which  is adjacent to a tetrahedron $\Delta' \in \mathcal{T}^{(3)}$ with vertex set $\{A_{i},B_{j}, G, I\}$, satisfying  that the edge  $[A_{i},B_j]$ lies in $\Gamma$. By  Sparsity I in Lemma \ref{lemma:inducedgraph}, at most one face of $\Delta$ satisfies this property. We replace $\Delta$ with a copy of  $\mathcal{C}_{1163}$, such that the vertices of $\Delta$ are mapped to the genuine vertices of  $\mathcal{C}_{1163}$. In addition, we require that $A_i$ is mapped to one of the vertices $a_j=\overline{a}_j$ ($j=2,3,4$) of  $\mathcal{C}_{1163}$,  and  that $\delta$ is mapped to  the 2-complex  $c_{30}$;

  	 	\item[\textbf{Replacement 6}] \label{almostfns:replace6} If $\Delta$ is a tetrahedron in $\mathcal{T}^{(3)}$ such that $\Delta \cap \Gamma=B_{j}$ is a vertex, and   $\Delta$ admits a 2-face, namely the triangle $\delta$ with vertex set  $\{B_{j},H,F\}$, which  is adjacent to a  tetrahedron $\Delta' \in \mathcal{T}^{(3)}$ with vertex set $\{A_{i},B_{j},H,F\}$, satisfying  that the edge $[A_{i},B_j]$ lies in $\Gamma$.  By   Sparsity I in Lemma  \ref{lemma:inducedgraph},  at most one face of  $\Delta$ satisfies this property.   We replace $\Delta$ with  a copy of $\mathcal{C}_{1137}$, such that the vertices of $\Delta$  are mapped to the genuine vertices of  $\mathcal{C}_{1137}$. In addition, we  require that $B_j$ is mapped to one of the vertice $a_3=\overline{a}_3$ of $\mathcal{C}_{1137}$,  and  that  $\delta$ mapped to  the 2-complex   $c_{4}$.

  	 \end{itemize}

  	 We now show that all the substituted pieces $\Delta^*$,   corresponding to each  $\Delta \in \mathcal{T}^{(3)}$,  are glued together to yield a triangulation $\mathcal{T}^*$ of $N$:
  	 \begin{enumerate}
  	 	\item  \label{almostfns:match1} If 	$\Delta_1$ and $ \Delta_2 $ in $\mathcal{T}^{(3)}$ share a triangle $\delta$, and $\Delta_1 \cap \Gamma=\Delta_2 \cap \Gamma=\emptyset$, then both $\Delta^*_1$ and $\Delta^*_2$  are copies of $\mathcal{C}_{543}$, and they induce the same triangulation $c_{10}$ of $\delta$; hence,  $\Delta^*_1$ and $\Delta^*_2$ are  glued together along $\delta$;
  	 	
  	 	\item  \label{almostfns:match2} If $\Delta_1$ and $ \Delta_2 $ in $\mathcal{T}^{(3)}$ share a triangle $\delta$ with vertex set $\{G, I, F \}$, and  $\Delta_1 \cap \Gamma=\emptyset$ while $\Delta_2 \cap \Gamma=A_i \in \Gamma$, then $\Delta^*_1$ is a copy of $\mathcal{C}_{543}$  whereas  $\Delta^*_2$ is a copy of either $\mathcal{C}_{543}$ or $\mathcal{C}_{1163}$ by   [Replacement 1],  [Replacement 3],  and  [Replacement 5]. In either case, $\Delta^*_1$ and $\Delta^*_2$ induce the same triangulation $c_{10}$ of $\delta$; hence, $\Delta^*_1$ and $\Delta^*_2$ are glued together along  $\delta$;

  	 	\item  \label{almostfns:match3}  If $\Delta_1$ and $ \Delta_2$ in $   \mathcal{T}^{(3)}$ share a triangle $\delta$ with vertex set $\{G, I, F\}$ and $\delta \cap \Gamma= \emptyset$, then either $\Delta_1 \cap \Gamma=\emptyset$ and $\Delta_2 \cap \Gamma=B_i \in \Gamma$, or both $\Delta_1 \cap \Gamma$ and  $\Delta_2 \cap \Gamma$ are vertices of $\Gamma$. By analogy with Case (\ref{almostfns:match2}), both $\Delta^*_1$ and $\Delta^*_2$ are copies of  $\mathcal{C}_{543}$, $\mathcal{C}_{1133}$ or $\mathcal{C}_{1163}$. These copies induce the identical triangulation $ c_{10}$ on $\delta$; hence, $\Delta^*_1$ and $\Delta^*_2$ are glued together along  $\delta$;

  	 	\item  \label{almostfns:match4} If $\Delta_1$ and $ \Delta_2 $ in $\mathcal{T}^{(3)}$  share a triangle $\delta$ with vertex set $\{A_{i},B_j, G\}$, then both $\Delta^*_1$ and $\Delta^*_2$ are copies of $\mathcal{C}'_{543}$ by   [Replacement 2]. Moreover, they induce the same triangulation $c_3$ of $\delta$; hence,  $\Delta^*_1$ and $\Delta^*_2$ are  glued together along $\delta$;

  	 	\item  \label{almostfns:match5} If $\Delta_1$ and $ \Delta_2 $ in $\mathcal{T}^{(3)}$  share a triangle $\delta$ with vertex set $\{A_{i}, G, I \}$ such that $\Delta_1$ has  vertex set $\{A_{i}, G, I, F\}$ and  $\Delta_2$ has  vertex set $\{A_{i}, B_j, G, I\}$.  Then $\Delta^*_1$ is a copy of $\mathcal{C}_{1163}$ by   [Replacement 5], and $\Delta^*_2$ is a copy of $\mathcal{C}'_{543}$ by   [Replacement 2]. Moreover, they induce the same
  	 	the same triangulation $c_{30}$  of $\delta$; hence, $\Delta^*_1$ and $\Delta^*_2$ are  glued together along $\delta$;

  	 	\item \label{almostfns:match6} 	If $\Delta_1$ and $ \Delta_2 $ in $\mathcal{T}^{(3)}$  share a triangle $\delta$ with vertex set $\{B_{j}, G, I \}$ such that $\Delta_1$ has  vertex set $\{B_{j}, G, I, F\}$ and  $\Delta_2$ has  vertex set $\{B_{j}, A_i, G, I\}$.  Then $\Delta^*_1$ is a copy of $\mathcal{C}_{1137}$ by   [Replacement 6],  and $\Delta^*_2$ is a copy of $\mathcal{C}'_{543}$.
  	 	Moreover, they induce   the same triangulation $c_{4}$  of $\delta$; hence, $\Delta^*_1$ and $\Delta^*_2$ are glued together along $\delta$;

  	 	\item \label{almostfns:match7}  In all other cases, for adjacent tetrahedra $\Delta_1$ and $ \Delta_2 $ in $\mathcal{T}^{(3)}$ sharing a triangle $\delta$, $\Delta^*_1$ and $\Delta^*_2$  induce the same triangulation $c_{10}$  of $\delta$; hence,  $\Delta^*_1$ and $\Delta^*_2$ are glued together along $\delta$.

  	 \end{enumerate}

  	 In total, all of the substituted 3-complexes $\Delta^*$ for each  $\Delta $ in $\mathcal{T}^{(3)}$ are glued together to give a triangulation $\mathcal{T}^*$ of $N$. 
  	 Compared with the special subdivision of \(\mathcal{T}\) in \cite{PrzytyckiS:2009}, we call the triangulation \(\mathcal{T}^*\) above a \emph{relatively special subdivision} associated with \((\mathcal{T}, \Gamma)\).

  	 Let $\mathcal{T}$ be the triangulation in  Lemma  \ref{lemma:inducedgraph}; we denote by $\mathcal{X}$ the 2-skeleton of $\mathcal{T}$, and by $\mathcal{X}^*$ the subdivision of $\mathcal{X}$ induced by the relatively special subdivision  $(\mathcal{T}, \Gamma) \rightarrow (\mathcal{T}^*, \Gamma)$. We need the following in the proof of  Theorem  \ref{thm:almostfns}. 
  	 
  	 \begin{lemma} \label{lemma:2skeleton}   The 2-complex  $\mathcal{X}^*$ is a full subcomplex of  $\mathcal{T}^*$, and it is flag-no-square.

  	 \end{lemma} 
  	 \begin{proof}$\mathcal{X}^*$ is constructed  from $\mathcal{X}$ by replacing each 2-simplex $\delta$ in  $\mathcal{X}^{(2)}$ with a 2-complex $\delta^*$, where each $\delta^*$ is isomorphic to  $c_3$, $c_4$, $c_{10}$, or $c_{30}$. 
  	 	If a vertex set $\{a_0,a_1, \dots,a_k\} $ of  $ \mathcal{X}^*$  spans a $k$-simplex in $\mathcal{T}^*$, then one  easily verify  that  $k$ is either  $1$ or $2$. Furthermore, by the construction of  $\mathcal{T}^*$,   there exists a 2-simplex $\delta$ of $\mathcal{X}$ containing  all vertices  $a_i$. 		For every 3-simplex $\Delta \in \mathcal{T}^{(3)}$ and every 2-face $\delta$ of $\Delta$, $\delta^*$ is a full subcomplex of $\Delta^*$. So the vertex set $\{a_0,a_1, \dots,a_k\} $ spans a $k$-simplex in $\mathcal{X}^*$.

  	 	We first remark that we need  the   Sparse I condition  to ensure  $\mathcal{X}^*$ is  flag-no-square. For example,  suppose  an  edge $[A,B]$ of  $\Gamma$ has  its  edge link   in $\mathcal{T}$ a $4$-gon. Let $\Delta_1, \dots,  \Delta_4$ be the four tetrahedra in $\mathcal{T}$   containing $[A,B]$ as one of their edges. For each  $i=1,2,3,4$,  denote by $\delta_i$ the 2-face of  $ \Delta_i$ that contains the vertex  $B$,  but not $A$.  Each $\delta_i$  is then  subdivided into $\delta^*_i$, a copy of  $c_4$, in the procedure from  $\mathcal{X}$  to $\mathcal{X}^*$.  It is readily  seen that  there is an empty square in $\cup^4_{i=1}\delta^*_i$, which remains empty in  $\mathcal{X}^*$.

  	 	Suppose $k\geq 2$, and let $\{C_0, C_1,\cdots,  C_k\}$ be a set of vertices in $\mathcal{X}^*$  pairwise connected by edges.

  	 	Firstly, since each of  $c_3$, $c_4$, $c_{10}$ and $c_{30}$ is flag,  and $\mathcal{T}$ is a simplicial triangulation, suppose  there exists  a 2-simplex $\delta \in \mathcal{X}^{(2)}$ such that, for instance,  the vertex  $C_0$ lies in  the interior of $\delta$. Then $\delta^*$ is  isomorphic to  $c_{10}$ or $c_{30}$,  all the vertices $C_i$ lie in  $\delta^*$,  $k=2$, and $\{C_0, C_1,  C_2\}$ spans a $2$-complex in $\mathcal{X}^*$;

  	 	Secondly, we now suppose that all the vertices in  $\{C_0, C_1,\cdots,  C_k\}$ lie in the 1-skeleton of  $\mathcal{X}$, and 	 there exists  an edge  $e \in \mathcal{X}^{(1)}$ such that, for example,  the vertex  $C_0$ lies in  the interior of $e$:
  	 	\begin{enumerate}
  	 		\item If all the edges $[C_0, C_i]$ lie in $\mathcal{X}^{(1)}$, then $k=1$, and $\{C_0, C_1\}$ spans an  edge in $\mathcal{X}^*$;
  	 		\item If there is a 2-simplex $\delta \in \mathcal{X}^{(2)}$ (with  $e$ being   an edge of $\delta$) such that the open edge $(C_0,C_1)$ lies in the interior of $\delta$, then $\delta^*$ is a copy of $c_3$ or $c_4$.
  	 		Then there exists a 3-simplex $\Delta$ in $\mathcal{T}$, such that $\delta$ is a 2-face of $\Delta$, and  one edge of it, say $[A_i, B_j]$, is also an edge  in $\Gamma$, in turn the 3-complex  $\Delta^*$  is a copy of $\mathcal{C}'_{543}$.
  	 		We now consider all 3-simplices $\Delta_1=\Delta, \Delta_2, \cdots, \Delta_l$ in  $\mathcal{T}$ with $[A_i, B_j]$ as one of their edges.  Consider $\cup^l _{i=1}\Delta^*_i$, which is a union of $l$ copies of  $\mathcal{C}'_{543}$ gluing along copies of $c_3$ surrounding  the edge $[A_i, B_j]$. The boundary of $\cup^l _{i=1}\Delta^*_i$ is a union of $l$ copies of 
  	 		$c_{30}$ surrounding  the vertex  $A_i$, and $l$ copies of 
  	 		$c_{4}$ surrounding  the vertex $B_j$. There are four cases where $[C_0, C_1]$ lies in $\delta^*$:
  	 		
  	 		\begin{itemize}

  	 			\item 
  	 			$\delta^*$ is  a copy of $c_4$, and there is a 2-simplex in $\delta^*$ with vertex set $\{B_j,C_0,C_1\}$. Consider the vertex $C_2$. If $C_2$ lies in $\Delta$, then $C_2$ lies in $\delta$ (it may be $B_j$ or an interior  point of one edge of $\delta$).   In either subcases, we have $k=2$, and $\{C_0, C_1, C_2\}$ 	spans a $2$-simplex in $\mathcal{X}^*$.
  	 			Otherwise,  $\{C_0, C_1, \dots  C_k\} \cap \Delta=\{C_0,C_1\}$. All  2-simplices of  $\mathcal{T}$ containing $C_0$ are denoted by $\delta_{0,p}$ for $1 \leq p \leq  i_0$, arranged circularly.
  	 			Similarly, all  2-simplices of  $\mathcal{T}$ containing $C_1$ are denoted by $\delta_{1,q}$ for $1 \leq q \leq  i_1$, also arranged circularly. Here $\delta_{0,1}=\delta_{1,1}=\delta$, and both $\delta_{0,2}$ and $\delta_{1,2}$ are 2-faces of $\Delta$.  Moreover, $\delta^*_{0,2}$ and  $\delta^*_{1,2}$ are isomorphic to $c_4$.  Then all  $C_i$ for $2 \leq i \leq k$ lie in the intersection of  
  	 			$$(\cup^{i_0}_{p=2} \delta_{0,p}) \cap (\cup^{i_1}_{q=2} \delta_{1,q}) \setminus \{B_j\}.$$ We obtain $k=2$, and $C_2$ lies in $\delta_{0,2} \cap \delta_{1,2}$. Meanwhile, the edge valence of  $[A_i,B_j]$ is three.  This yields a contradiction to the  Sparse I condition;

  	 			\item 	$\delta^*$ is  a copy of $c_4$, but there is  no  2-simplex in $\delta^*$ with vertex set $\{B_j,C_0,C_1\}$. 	 We denote the  vertices of $\Delta$ by $\{A_i,B_j,C,D\}$, and we assume $C_0$ lie in the interior of $[B_j,C]$,  we assume $C_1$ lie in the interior of $[C,D]$.
  	 			We now also  consider the vertex $C_2$. If $C_2$ lies in $\Delta$, then $C_2$ lies in $\delta$.   We have $k=2$, and $\{C_0, C_1, C_2\}$ 	spans a $2$-simplex in $\mathcal{X}^*$.
  	 			Otherwise,  $\{C_0, C_1, \dots  C_k\} \cap \Delta=\{C_0,C_1\}$. 
  	 			All  2-simplices of  $\mathcal{T}$ containing $C_0$ are denoted by $\delta_{0,p}$ for $1 \leq p \leq  i_0$, arranged circularly.
  	 			Similarly, all  2-simplices of  $\mathcal{T}$ containing $C_1$ are denoted by $\delta_{1,q}$ for $1 \leq q \leq  i_1$, also arranged circularly. Here $\delta_{0,1}=\delta_{1,1}=\delta$, and both $\delta_{0,2}$ and $\delta_{1,2}$ are 2-faces of $\Delta$.  Moreover, $\delta^*_{0,2}$ is isomorphic to $c_3$;  $\delta^*_{1,2}$ are isomorphic to $c_{30}$.  Then all  $C_i$ for $2 \leq i \leq k$ lie in the intersection of  
  	 			$$(\cup^{i_0}_{p=2} \delta_{0,p}) \cap (\cup^{i_1}_{q=2} \delta_{1,q}) \setminus \{C\}.$$

  	 			We assume $C_2$ lies in $\delta_{0,p} \cap \delta_{1,q}$ for some $p, q \geq 2$. Since neither  $\delta^*_{0,p} $ nor $ \delta^*_{1,q}$  is isomorphic to $c_3$, $C_2$ lies in the interior of an edge of $\delta_{0,p} $ and $ \delta_{1,q}$, then  $\delta_{0,p} \cap \delta_{1,q}$ is an edge where $C_2$ lies. 
  	 			We obtain $k=2$, and $C_2$ lies in $\delta_{0,p} \cap \delta_{1,q}$. Each $\delta^*_{0,p} $ and  $ \delta^*_{1,q}$  is isomorphic to $c_4$. But this implies two vertices of $\delta_{0,p} $ (and $\delta^*_{1,q} $) are is $\Gamma$, say  $B_k$ and $B_h$, a contradiction  to our assumption in Lemma 
  	 			\ref{lemma:inducedgraph}.

  	 			\item 	$\delta^*$ is  a copy of $c_3$, and the segment $[C_0,C_1]$ (or $[C_1,C_0]$) is parallel to $[A_i,B_j]$ within $\delta$.  Consider the vertex $C_2$. If $C_2$ lies in $\Delta$, then $C_2$ lies in $\delta$ (it may be $B_j$ or another vertex of $\delta$).   In either subcases, we have $k=2$, and $\{C_0, C_1, C_2\}$ 	spans a $2$-simplex in $\mathcal{X}^*$.
  	 			Otherwise,  $\{C_0, C_1, \dots  C_k\} \cap \Delta=\{C_0,C_1\}$.  We denote the  vertices of $\Delta$ by $\{A_i,B_j,C,D\}$, and we assume $C_0$ lie in the interior of $[A_i,C]$,  we assume $C_1$ lie in the interior of $[B_j,C]$.
  	 			All  2-simplices of  $\mathcal{T}$ containing $C_0$ are denoted by $\delta_{0,p}$ for $1 \leq p \leq  i_0$, arranged circularly.
  	 			Similarly, all  2-simplices of  $\mathcal{T}$ containing $C_1$ are denoted by $\delta_{1,q}$ for $1 \leq q \leq  i_1$, also arranged circularly. Here $\delta_{0,1}=\delta_{1,1}=\delta$  and both $\delta_{0,2}$ and $\delta_{1,2}$ are 2-faces of $\Delta$.  Up to reordering $C_0$ and $C_1$, $\delta^*_{0,2}$ is isomorphic to $c_{30}$;  $\delta^*_{1,2}$ are isomorphic to $c_{4}$. 
  	 			Then all  $C_i$ for $2 \leq i \leq k$ lie in the intersection of  
  	 			$$(\cup^{i_0}_{p=2} \delta_{0,p}) \cap (\cup^{i_1}_{q=2} \delta_{1,q}) \setminus \{C\}.$$ 
  	 			We assume $C_2$ lies in $\delta_{0,p} \cap \delta_{1,q}$ for some $p, q \geq 2$.
  	 			Note that  $\delta_{0,p} $ has one vertex $A_i$, so   $\delta^*_{0,p} $  is isomorphic to  $c_3$, $c_{10}$ or $c_{30}$; Similarly,  $\delta_{1,q} $ has one vertex $B_i$, so  $ \delta^*_{1,q}$ is isomorphic to $c_3$, $c_{4}$ or $c_{10}$.  But $C_0$ is an interior point of one edge od $\delta$, we have  $\delta^*_{0,p} $  is isomorphic to  $c_3$, and then the edge $\delta_{0,p} \cap \delta_{1,q}$ has one vertex $B_l$ for some $l$. Which implies $\delta_{1,q}$ has two type $B$ vertices in $\Gamma$, say $B_j$ and $B_l$. A contradiction  to our assumption in Lemma 
  	 			\ref{lemma:inducedgraph};

  	 			\item 	$\delta^*$ is  a copy of $c_3$, and  $B_j$ is an end vertex of $e$.
  	 			Consider the vertex $C_2$. If $C_2$ lies in $\Delta$, then $C_2$ lies in $\delta$ (it may be $A_i$ or another point in the interior of an edge   of $\delta$).   In either subcases, we have $k=2$, and $\{C_0, C_1, C_2\}$ 	spans a $2$-simplex in $\mathcal{X}^*$.

  	 			Otherwise,  $\{C_0, C_1, \dots  C_k\} \cap \Delta=\{C_0,C_1\}$.  We denote the  vertices of $\Delta$ by $\{A_i,B_j,C,D\}$, and we assume $C_0$ lie in the interior of $e=[A_i,C]$.
  	 			All  2-simplices of  $\mathcal{T}$ containing $C_0$ are denoted by $\delta_{0,p}$ for $1 \leq p \leq  i_0$, arranged circularly.
  	 			Here $\delta_{0,1}=\delta$  and  $\delta_{0,2}$ is a 2-face of $\Delta$.
  	 			We assume $C_2$ lies in $\delta_{0,p}$ for some $q \geq 2$. Then $\delta^*_{0,p}$ is isomorphic to $c_3$ as there ia the vertex $A_i$ in isomorphic to $\delta_{0,p}$. So there is  also a type $B$ vertex of $\delta_{0,p}$, say, $B_{l}$. But then consider the open edge $(C_1,C_2)$ which lies in the interior of a 2-simplex $\delta'$, recall $C_1=B_j$, and now $B_j$ and $B_l$ are vertices of $\delta'$. A contradiction  to our assumption in Lemma 
  	 			\ref{lemma:inducedgraph}.

  	 		\end{itemize}

  	 	\end{enumerate}
  	 	
  	 	Thirdly,  if all vertices $\{C_0, C_1,\cdots,  C_k\}$ lie in the 0-skeleton  $\mathcal{X}^{(0)}$, then all  $C_i$ belong to the flag-no-square graph $\Gamma$,  so we  have $k \leq 1$, which yields a contradiction.
  	 	
  	 	In total,   $\mathcal{X}^*$ is a flag  2-complex.

  	 	We now prove  square-freeness for $\mathcal{X}^*$. Suppose for contradiction that   $(C_1, C_2, C_3,C_4, C_1)$ forms  an empty square    in $\mathcal{X}^*$.

  	 	Firstly, suppose there exists a 2-simplex $\delta \in \mathcal{X}^{(2)}$ such that, for instance, vertex $C_1$ lies in the interior of $\delta$. Then the substituted 2-complex $\delta^*$ is isomorphic to either $c_{10}$ or $c_{30}$. If $C_2$ also lies in the interior of $\delta$, then both $C_3$ and $C_4$ are contained in $\delta$. Since $c_{10}$ and $c_{30}$ are each square-free, $(C_1,C_2,C_3,C_4,C_1)$ cannot form an empty square.
  	 	Now consider the case where both $C_2$ and $C_4$ lie on the boundary of $\delta$:
  	 	\begin{enumerate}
  	 		\item $\delta^* \cong c_{30}$. By Figure \ref{figure:c30}, an edge $[C_2,C_4]$ exists inside the substituted 2-complex $\delta^*$, so $(C_1,C_2,C_3,C_4,C_1)$ is not an empty square.
  	 		
  	 		\item $\delta^* \cong c_{10}$. By Figure \ref{figure:c10}, one of two possibilities holds: 
  	 		either an edge $[C_2,C_4]$ exists in $\delta^*$, which again prevents $(C_1,C_2,C_3,C_4,C_1)$ from being an empty square; or there exists another 2-simplex $\delta_1\in \mathcal{X}^{(2)}$ whose substituted complex $\delta_1^*$ is isomorphic to $c_{10}$, with $[C_2,C_3]\cup[C_3,C_4]\subset\delta_1$ and $C_3$ contained in the interior of $\delta_1$. This forces $\delta\cap\delta_1$ to equal the union of two distinct edges, contradicting the hypothesis that $\mathcal{T}$ is a simplicial triangulation.

  	 	\end{enumerate}

  	 	Secondly, we now assume all four vertices $C_1,C_2,C_3,C_4$ belong to $\mathcal{X}^{(1)}$:
  	 	\begin{enumerate}
  	 		\item \label{2complex:2a}
  	 		Not all edges $[C_i,C_{i+1}]$ can lie in $\mathcal{X}^{(1)}$. Suppose to the contrary that every such edge is contained in $\mathcal{X}^{(1)}$. With the exception of edges in $\Gamma$, every edge $e\in\mathcal{X}^{(1)}$ gets subdivided into two subedges. Since $\Gamma$ is square-free, it is impossible for all $[C_i,C_{i+1}]$ to lie in $\Gamma$. We therefore fix an edge $e=[C_1,C_3]\in\mathcal{X}^{(1)}$ that has been subdivided into the two subedges $[C_3,C_4]$ and $[C_4,C_1]$. In this case, both $[C_1,C_2]$ and $[C_2,C_3]$ are edges of $\Gamma$. This contradicts the fact that $\Gamma$ is a full subgraph of $\mathcal{T}$.

  	 		\item\label{2complex:2b}  From (\ref{2complex:2a}), we now assume the open edge $[C_1,C_2]$ lies in the interior of  a $2$-simplex $\delta \in \mathcal{X}^{(2)}$, where the substituted 2-complex  $\delta^*$ is   isomorphic to  $c_4$. The 2-simplex $\delta$ is the common face of 3-simplices $\Delta_1$ and $\Delta_2$  in $\mathcal{T}^{(3)}$, where  the substituted 3-simplices $\Delta^*_1$ and  $\Delta^*_2$ are isomorphic to  $\mathcal{C}'_{543}$ and $\mathcal{C}_{1137}$, respectively. We assume the vertex set of $\Delta_1$ is $\{A_i,B_j,D,F\}$, so the vertex set of $\delta$ is $\{B_j,D,F\}$.
  	 		There are two edges $e_1$ and $e_2$ of $\delta$, such that $C_i$ is the midpoint of $e_i$ for $i=1,2$, respectively:
  	 		
  	 		\begin{itemize}
  	 			
  	 			\item If the common end points of $e_1$ and $e_2$ is  $B_j$, that is, $e_1=[B_j,D]$ and $e_2=[B_j,F]$. 	
  	 			We denote the 2-simplices in $\mathcal{X}^{(2)}$ with $e_1$ as one edge by  $\delta_{1,i}$ for  $1 \leq i \leq {m_1}$,  arranged circularly; Similarly,  the 2-simplices in $\mathcal{X}^{(2)}$ with $e_2$ as one edge  are denoted by   $\delta_{2,j}$ for  $1 \leq j \leq {m_2}$,  also  arranged circularly. We also assume $\delta_{1,1}=\delta_{2,1}=\delta$;     $\delta_{1,2}$ and  $\delta_{2,2}$ lie in $\Delta_1$, so both  the substituted 2-complexes $\delta^*_{1,2}$ and  $\delta^*_{2,2}$ are isomorphic to $c_3$; 
  	 			$\delta_{1,m_1}$ and  $\delta_{2,m_2}$ lie in $\Delta_{2}$, so both  the substituted 2-complexes $\delta^*_{1,m_1}$ and  $\delta^*_{2,m_2}$ are isomorphic  $c_{10}$. 
  	 			It is easy to see that $C_4$ cannot coincide with  $B_j$ or  the middlepoint of the edge $[D,F]$. Furthermore, if $C_4=D$, then $C_3 \neq F$. We suppose  $[C_2,C_3]$ is contained in $\delta_{2,j}$ with some $j \neq 1$. The substituted 2-complex $\delta^*_{2,j}$ is  isomorphic to either  $c_3$ or $c_4$. If the  2-complex $\delta^*_{2,j}$ is  isomorphic to $c_3$,  the vertex set of $\delta_{2,j}$ is   $\{A_k,F,B_j\}$ for some $k$. Meanwhile, the substituted 2-complex from the 2-simplex with vertex set $\{A_k,F,D\}$ contains the edge   $[C_3,D]$. This contradicts the fact  that  the substituted 2-complex  is isomorphic to either $c_{10}$ or $c_{30}$;
  	 			If  the 2-complex $\delta^*_{2,j}$ is  isomorphic to $c_4$, let the vertex set of $\delta_{2,j}$ be $\{G,F,B_j\}$, where $G \notin\Gamma$. Then the substituted 2-complex from the 2-simplex having vertex set $\{G,B_j,D\}$ or  $\{G,F,D\}$ contains the  edge $[C_3,D]$. Both cases lead  to a contradiction, as  the substituted 2-complexes are isomorphic to $c_{10}$.

  	 			So now  we assume 	$C_3\neq F$ and $C_4 \neq D$, and $[C_2,C_3]$  lies in the  substituted 2-complex $\delta^*_{2,j}$ for some $j\neq 1$, and $[C_1,C_4]$  lies in the  substituted 2-complex  $\delta^*_{1,i}$ for some $i \neq 1$. 
  	 			Since $C_1$ and $C_2$ are  the midpoints of $e_1$  and $e_2$, respectively, the open edges  $(C_2,C_3)$  and $(C_1,C_4)$ lie in the interiors of  $\delta_{2,j}$ and  $\delta_{1,i}$, respectively.  This implies that each   $\delta^*_{1,i}$ and   $\delta^*_{2,j}$ 
  	 			is isomorphic to one of $c_3$ or $c_4$:
  	 			
  	 			\begin{itemize}

  	 				\item Both  $\delta^*_{1,i}$ and   $\delta^*_{2,j}$ 
  	 				are isomorphic to $c_4$.  We assume the vertex set of  $\delta_{1,i}$ is $\{B_j,F,H\}$, and the vertex set of  $\delta_{2,j}$ is $\{B_j,D, L\}$:
  	 				
  	 				\begin{itemize}

  	 					\item $C_3$ is the middlepoint of $[B_j,H]$, and  $C_4$ is the middlepoint of $[B_j,L]$.  
  	 					Then there exists a 2-simplex $\delta_3$  in $\mathcal{X}^{(2)}$ such that the substituted 2-complexes $\delta^*_{3}$ is isomorphic to $c_4$,  and  open edge  $(C_3,C_4)$ lies in the interior of $\delta_3$. 
  	 					Consequently, the union of $\delta$, $\delta_{1,i}$,  $\delta_{2,j}$, and $\delta_{3}$ forms a length-4 poisonous circle. This contradicts our assumption of Sparsity I from Lemma \ref{lemma:inducedgraph};
  	 					
  	 					\item $C_3$ is the middlepoint of $[B_j,H]$, and  $C_4$ is the middlepoint of $[D,L]$.  Since there exists an edge $[C_3,C_4]$ in $\mathcal{X}^*$, we have  $H=L$ (it is easy to see $H \neq D$). But then $\delta_{1,i}=\delta_{2,j}$, it is still a contradiction;

  	 					\item $C_3$ is the middlepoint of $[H,F]$, and  $C_4$ is the middlepoint of $[B_j,L]$. Symmetry implies this subcase cannot occur, similarly to the prior case;

  	 					\item $C_3$ is the middlepoint of $[H,F]$,  and  $C_4$ is the middlepoint of $[D,L]$.  Since there exists an edge $[C_3,C_4]$ in $\mathcal{X}^*$, we have  $H=L$, and  open edge  $(C_3,C_4)$ lies in the interior of $\delta_3$ having vertex set $\{D,F,L\}$.  This contradicts that the substituted 2-complexes $\delta^*_{3}$ is isomorphic to $c_{10}$.
  	 					
  	 				\end{itemize}

  	 				\item  the substituted 2-complexes $\delta^*_{1,i}$ 
  	 				is isomorphic to $c_3$, we have the vertex set of  $\delta_{1,i}$  is $\{A_k, B_i,F\}$, and $C_3$ is the middlepoint of the edge $[A_k,F]$. We remark here that $A_k$ is not  necessarily to be  the vertex $A_i$. Then

  	 				\begin{itemize}
  	 					
  	 					\item

  	 					either  $C_4=A_k$. But there is no edge $[C_1,A_k]$ in $\mathcal{X}^*$, a contradiction;

  	 					\item or  there is a 2-simplex $\delta_3$ in $\mathcal{X}$ having vertex set $\{A_k,B_l,F\}$ such that $(C_3,C_4)$ lies in the interior of $\delta_3$.  Here  $B_l$ is not  necessarily to be  the vertex $B_j$.
  	 					Then $C_4=B_l$ or $C_4$ is the middlepoint of $[B_l,F]$. Either subcase  is  impossible by consider the 2-simplex 	 $\delta_{2,j}$  where $B_j$ and $D$ are vertices of it;

  	 				\end{itemize}

  	 				\item Symmetry implies that  the substituted  2-complexes  $\delta^*_{2,j}$ 
  	 				is isomorphic to $c_3$  subcase cannot occur, similarly to the prior subcase; 
  	 				
  	 			\end{itemize}

  	 			\item If the common end points of $e_1$ and $e_2$ is $D$, that is, $e_1=[B_j,D]$ and $e_2=[D,F]$. 	
  	 			We denote the 2-simplices in $\mathcal{X}^{(2)}$ with $e_1$ as one edge by  $\delta_{1,i}$ for some $1 \leq i \leq {m_1}$, arranged circularly; Similarly, the 2-simplices in $\mathcal{X}^{(2)}$ with $e_2$ as one edge by  $\delta_{2,j}$ for some  $1 \leq j \leq {m_2}$. We also assume that $\delta_{1,1}=\delta_{2,1}=\delta$,  $\delta_{1,2}$ and  $\delta_{2,2}$ lie in $\Delta_1$,   $\delta_{1,m_1}$ and  $\delta_{2,m_2}$ lie in $\Delta_{2}$. So the substituted 2-complexes $\delta^*_{1,2}$    and  $\delta^*_{2,2}$ are isomorphic to    $c_3$ and  $c_{30}$, respectively.  Both the substituted 2-complexes $\delta^*_{1,m_1}$  and  $\delta^*_{2,m_2}$
  	 			are  isomorphic to   $c_{10}$.

  	 			\begin{itemize}
  	 				\item 	
  	 				If $C_4=B_j$, then  $C_3\neq F$. There is a 2-simplex $\delta_3$ with vertex set $\{D,F,B_l\}$ for some $l$ such that 
  	 				substituted 2-complexes $\delta^*_{3}$ is  isomorphic to    $c_4$ and $(C_2,C_3)$ lies in the interior of $\delta_3$. Furthermore, $C_3$ is the middlepoint of either $[B_k,D]$ or $[B_k,F]$. Then consider the edge $[C_3,B_j]=[C_3,C_4]$, either  the open edge   $(C_3,B_j)$ lies in the interior of   2-simplex $\delta_4$ with vertex set $\{B_k,A_p, D\}$ for some $p$ such that 
  	 				substituted 2-complexes $\delta^*_{4}$ is  isomorphic to    $c_3$, but this contradicts to the fact that $C_3$ is the middlepoint of $[D,B_k]$; or  $[C_3,B_j]$ lies in $\mathcal{X}^{(1)}$ and then  $B_j=B_l$, which is also a contradiction.

  	 				\item 
  	 				If $C_3=F$, then $C_4 \neq B_j$:
  	 				\begin{itemize}
  	 					\item There is a 2-simplex $\delta_3$ with vertex set $\{B_j,D,G\}$  such that 
  	 					substituted 2-complexes $\delta^*_{3}$ is  isomorphic to    $c_4$ and $(C_1,C_4)$ lies in the interior of $\delta_3$. Either $C_4$ is the middlepoint of $[D,G]$ or $[B_j,G]$. But then $F=G$ or $F=B_j$, a contradiction.

  	 					\item 	or $\delta_3$ with vertex set $\{B_j,A_l,D\}$  for some $l$ such that 
  	 					substituted 2-complexes $\delta^*_{3}$ is  isomorphic to    $c_3$ and $(C_1,C_4)$ lies in the interior of $\delta_3$. Further, $C_4$ is the middlepoint of the edge $[A_l,D]$. Since there is an edge $[C_3,C_4]$, we have $D=F$, a contradiction. 
  	 				\end{itemize}

  	 				\item  The open edge  	$(C_2,C_3)$  lies in the interior of  $\delta_{2,p}$ for some $p \geq 2$ with vertex set $\{D,F,B_l\}$ for some $l$, so $\delta^*_{2,p}$ is isomorphic to $c_4$. 
  	 				Furthermore,   the open edge  	$(C_1,C_4)$   lies in the interior of 	$\delta_{1,q}$ for some $q \geq 2$; Then either 	$\delta_{1,q}$ has vertex set $\{A_k,B_j,D\}$ for some $k$ and    $\delta^*_{1,q}$ is isomorphic to $c_3$, or  $\delta_{1,q}$ has vertex set $\{G,B_j,D\}$ for some $l$ and  $\delta^*_{1,q}$ is isomorphic to $c_4$:

  	 				\begin{itemize}
  	 					
  	 					\item  $C_3$ is the middlepoint of $[D,B_l]$,  $\delta_{1,q}$ has vertex set $\{A_k,B_j,D\}$, and $C_4$ is the middlepoint of $[A_k,D]$. Then the open edge  $(C_3,C_4)$ lies in the interior of $\delta_3$ with vertex set $\{A_k,B_l,D\}$, and  $\delta^*_3$ is isomorphic to $c_3$.  Now the 2-simplices  $\delta_{1,q}$ and $\delta_3$  are adjacent, a contradiction to our assumption Sparse III of $\mathcal{T}$ in Lemma \ref{lemma:inducedgraph};

  	 					\item $C_3$ is the middlepoint of $[D,B_l]$,  $\delta_{1,q}$ has vertex set $\{G,B_j,D\}$, and $C_4$ is the middlepoint of $[G,D]$. Then the open edge  $(C_1,C_4)$ lies in the interior of $\delta_{1,q}$ with vertex set $\{D,G,B_j\}$, and  $\delta^*_{1,q}$ is isomorphic to $c_4$. But this implies there is an vertex $A_k$ and there is  a 3-simplex $\Delta'$ in $\mathcal{T}^{(3)}$
  	 					with vertex set $\{A_k,B_j,D,G\}$. Now thw 2-simplices with vertex sets $\{A_k,B_j,D\}$ and $\{A_i,B_j,D\}$ are adjacent, a contradiction to our assumption Sparse III of $\mathcal{T}$ in Lemma \ref{lemma:inducedgraph};

  	 					\item  $C_3$ is the middlepoint of $[D,B_l]$,   $\delta_{1,q}$ has vertex set $\{G,B_j,D\}$ and $C_4$ is the middlepoint of $[G,B_j]$.  Then the open edge  $(C_3,C_4)$ lies in the interior of $\delta_{3}$  such that   $\delta^*_{3}$ is isomorphic to $c_4$. But this implies there is an vertex $AB_l=B_j$, a contradiction;

  	 					\item  $C_3$ is the middlepoint of $[F,B_l]$,  $\delta_{1,q}$ has vertex set $\{A_k,B_j,D\}$,  and $C_4$ is the middlepoint of $[A_k,D]$. Then the open edge  $(C_3,C_4)$ lies in the interior of $\delta_{3}$  such that   $\delta^*_{3}$ is isomorphic to $c_3$. It is ready to see  a contradiction;

  	 					\item $C_3$ is the middlepoint of $[F,B_l]$,  $\delta_{1,q}$ has vertex set $\{G,B_j,D\}$, and $C_4$ is the middlepoint of $[G,D]$.  Then the open edge  $(C_3,C_4)$ lies in the interior of $\delta_{3}$  such that   $\delta^*_{3}$ is isomorphic to $c_4$. In particular, this implies $F= D$ or $F=G$,   a contradiction;

  	 					\item  $C_3$ is the middlepoint of $[F,B_l]$,   $\delta_{1,q}$ has vertex set $\{G,B_j,D\}$, and $C_4$ is the middlepoint of $[G,B_j]$. Then the open edge  $(C_3,C_4)$ lies in the interior of $\delta_{3}$  such that   $\delta^*_{3}$ is isomorphic to $c_4$. In particular, this implies $B_l=B_j$,   a contradiction.

  	 				\end{itemize}

  	 			\end{itemize}
  	 			
  	 			\item	Symmetry implies the  case that   the common end points of $e_1$ and $e_2$ is $F$ cannot occur, similarly to the prior subcase.

  	 		\end{itemize}

  	 		\item   We now assume the open edge  $(C_1,C_2)$ lies in the interior  of a  2-simplex $\delta \in \mathcal{X}^{(2)}$, where the substituted 2-complex $\delta^*$ is   isomorphic to   $c_3$.  We assume the vertex set of $\Delta_1$ is $\{A_i,B_j,D,F\}$, so the vertex set of $\delta$ is $\{A_i,B_j,F\}$:
  	 		
  	 		\begin{itemize}

  	 			\item If $C_i$ is the middle point of the edge $e_i$ of $\delta$ for $i=1,2$, respectively, where $e_1=[A_i,F]$ and  $e_2=[B_j,F]$. It is easy to see $C_3 \neq B_j$ and $C_3 \neq F$. So the open edge $(C_2,C_3)$ lies in the interior of a 2-simplex $\delta_3$.  If substituted 2-complex $\delta^*_1$ is   isomorphic to   $c_4$, wo go the the previous case; If substituted 2-complex $\delta^*_3$ is   isomorphic to   $c_3$, so $\delta_1$ has vertex set $\{A_l,B_j,F\}$,  then $\delta$ and $\delta_1$ are adjacent 2-simplexes,  a contradiction to our assumption Sparse III of $\mathcal{T}$ in Lemma \ref{lemma:inducedgraph};

  	 			\item  If $C_1=B_i$ is a vertex of $\Gamma$ and $C_2$ is the middle point of the edge $e_2=[A_i,F]$ of $\delta$. 
  	 			It is easy to see $C_3 \neq A_i$, and $C_3$ is not the middle point of $[B_j,F]$.

  	 			If $C_3=F$, we consider $C_4$:

  	 			\begin{itemize}

  	 				\item $C_4=A_k$ for some $k$, then there is no edge $[F,A_k]$ in  $\mathcal{X}^{(2)}$, a contradiction;
  	 				
  	 				\item  There an edge $[B_j,G]$  in  $\mathcal{X}^{(1)}$ such that $C_4$ is the middle point of $[B_j,G]$. But then consider the edge $[C_3,C_4]$, we have $F=G$, a contradiction.

  	 			\end{itemize}

  	 			We now have $C_3 \neq F$, and the open edge $(C_2,C_3)$ lies in the interior of a 2-simplex $\delta_3$. The 2-complex $\delta^*_3$ is   isomorphic to   $c_3$, so  $\delta_3$ has vertex set $\{A_i,F,B_l\}$ for some $l$. Then $\delta$ and $\delta_3$ are adjacent 2-simplexes,  a contradiction to our assumption Sparse III of $\mathcal{T}$ in Lemma \ref{lemma:inducedgraph};

  	 		\end{itemize}

  	 	\end{enumerate}

  	 	In total,   $\mathcal{X}^*$ is a square-free complex.

  	 \end{proof}

  	 \subsection{Almost flag-no-square property of the triangulation} \label{subsection:almostfns} After the preparations in Subsection~\ref{subsection:subdivision}, we have

	  \begin{thm} \label{thm:almostfns}  Let $N$   be a 
	  	closed orientable 3-manifold with a triangulation $\mathcal{T}$,  and let   $\Gamma$ be a  full, bipartite, flag-no-square subgraph of $\mathcal{T}$ as in Lemma \ref{lemma:inducedgraph}. Let $(\mathcal{T}^*,  \Gamma)$ be the relatively special subdivision associated with \((\mathcal{T}, \Gamma)\). Then 
	  
	  	\begin{enumerate}[label=(\arabic*)]

	  		\item 	$\Gamma$ is a full subgraph of $\mathcal{T}^*$;
	  		
	  			\item   $\Sigma_g \hookrightarrow N$ as a full   subcomplex   $t^*$ of $\mathcal{T}^*$;
	  		
	  			\item  \label{emptytriangle} The triangulation \(\mathcal{T}^*\) is not flag, where the only empty simplices are the triangles contained in the copies of \(\mathcal{C}_{1137}\), as shown in Lemmas~\ref{lemma:fnsC1137simplex};

	  				\item  The triangulation \(\mathcal{T}^*\) is not square-free. If \((C_1,C_2,C_3,C_4,C_1)\) is an empty square in \(\mathcal{T}^*\), then
	  					\begin{enumerate}[label=(\arabic{enumi}.\arabic*)]
	  			
	  					\item \label{emptysquare1} it is contained in one of the copies of \(\mathcal{C}_{1137}\) and \(\mathcal{C}_{1163}\), as shown in Lemmas~\ref{lemma:fnsC1137square} and~\ref{lemma:fnsC1163};
	  					
	  					\item \label{emptysquare2} it is contained in two adjacent copies, one of \(\mathcal{C}'_{543}\) and the other of \(\mathcal{C}_{1163}\), along \(c_{30}\), but not totally in a copy of \(\mathcal{C}_{1163}\). By Lemma~\ref{lemma:halfsquare}, up to cyclic reordering, we may assume
	  					\begin{itemize}
	  						\item \(\{C_2, C_1, C_3\}\) is \(\{w_{2,34}, a_{23}, a_{24}\}\), \(\{w_{3,24},a_{23},a_{34}\}\), or \(\{w_{4,23}, a_{24}, a_{34}\}\);

	  							  						\item \(\{C_2, C_1, C_3\}=\{v_{1},G,F\}\subset\mathcal{C}_{30}\) for some pair of vertices \(G\) and \(F\) in the \(2\)-complex \(c_{30}\).
	  					\end{itemize}
  					
  				 Moreover, \(\{C_4, C_1, C_3\}=\{E,G,F\}\), as stated in Item \ref{halfsquare:a} of Lemma \ref{lemma:halfsquare}.

  				 		\item \label{emptysquare3} it is contained in two copies of \(\mathcal{C}_{1137}\) adjacent along an edge. That is, there are two 3-simplices \(\Delta_1\) and \(\Delta_2\) in \(\mathcal{T}\) such that \(\Delta_1 \cap \Delta_2 = e\) is an edge, and the substituted complexes \(\Delta^*_1\) and \(\Delta^*_2\) are both isomorphic to \(\mathcal{C}_{1137}\). Moreover,
  				 		\[
  				 		(C_1,C_2,C_3,C_4,C_1)=(a_k,v_{1,1},a_l,v_{1,2},a_k),
  				 		\]
  				 		where \(a_k, a_l\) are the endpoints of \(e\), and \(v_{1,1}\) and \(v_{1,2}\) are the copies of the vertex \(v_1\) in \(\Delta^*_1\) and \(\Delta^*_2\), respectively, under the isomorphisms with \(\mathcal{C}_{1137}\).

	  				\end{enumerate}

	  		\end{enumerate}

	  \end{thm}

\begin{proof}

In the process from the triangulation  $\mathcal{T}$  to the triangulation $\mathcal{T}^*$, we do not add any edge  connecting two vertices of $\Gamma$; thus,   $\Gamma$ is  a full subgraph of $\mathcal{T}^*$, and also  a full subgraph of $t^*$,  the corresponding subdivision  of the triangulation $t$ of the surface $\Sigma_g$ induced by the triangulation $\mathcal{T}^*$. It is also ready to see that $t^*$  is a full   subcomplex   of $\mathcal{T}^*$.

	We now  consider the empty triangles and squares in the 3-complex  $\mathcal{T}^*$.  Since, by gluing copies of \(\mathcal{C}_{543}\), \(\mathcal{C}'_{543}\), \(\mathcal{C}_{1137}\), and \(\mathcal{C}_{1163}\) together to get \(\mathcal{T}^*\), we do not fill the empty triangles, the empty triangles in each copy of \(\mathcal{C}_{1137}\), as those referred to in Lemmas~\ref{lemma:fnsC1137simplex}, remain empty in \(\mathcal{T}^*\).
	Moreover, we do not connect two non-adjacent vertices in every empty square in each copy of \(\mathcal{C}_{1137}\) and \(\mathcal{C}_{1163}\), as those referred to in Lemmas~\ref{lemma:fnsC1137square} and~\ref{lemma:fnsC1163}; then these empty squares remain empty in \(\mathcal{T}^*\).

  Suppose $\{C_1, C_2,\cdots,  C_k\}$ is a collection of vertices in $\mathcal{T}^*$ that are pairwise connected by edges:
		\begin{itemize}
			
			\item 	If there exists  a 3-simplex  $\Delta \in \mathcal{T}^{(3)}$ such that, for example,  the vertex  $C_1$ lies in  the interior of $\Delta$, then  all the vertices $C_i$ lie in  $\Delta^*$. 
			  Now,  each 
			$\mathcal{C}_{543}$,  $\mathcal{C}'_{543}$,  and  $\mathcal{C}_{1163}$ is flag, and  $\mathcal{C}_{1137}$ is almost flag, except for the empty triangles  in Lemma \ref{lemma:fnsC1137simplex}. Thus,  $k$ is at most 4,  and $\{C_1, C_2,\cdots,  C_k\}$ span a $(k-1)$ simplex in  $\mathcal{T}^*$, except the empty triangles in copies of $\mathcal{C}_{1137}$;

			\item 	Recall that  $\mathcal{X}$ is the 2-skeleton of $\mathcal{T}$, and $\mathcal{X}^*$ is  the subdivision of $\mathcal{X}$ induced by the relatively special subdivision  $\mathcal{T}^*$.
	We now suppose all vertices $C_i$ lie in  $\mathcal{X}^*$. By Lemma 	\ref{lemma:2skeleton},    $\mathcal{X}^*$ is a full subcomplex of  $\mathcal{T}^*$, and it is flag-no-square. So  $k \leq 3$,  $\{C_1, C_2,\cdots,  C_k\}$ spans a $(k-1)$ simplex in $\mathcal{X}^*$, in turn, a $(k-1)$ simplex in $\mathcal{T}^*$.

			\end{itemize}
	In total,  the only empty simplices in $\mathcal{T}^*$ are  those contained in the copies of $\mathcal{C}_{1137}$  as shown  in Lemma  \ref{lemma:fnsC1137simplex}.

	Suppose  $(C_1, C_2, C_3,C_4, C_1)$ is an empty square   in $\mathcal{T}^*$:
		\begin{itemize}
			\item If there is a 3-simplex $\Delta$   in $\mathcal{T}^{(3)}$ such that $C_1, C_2$
			lie in the interior of  $\Delta$, then all vertices $C_i$ lie in  $\Delta^*$, and $(C_1, C_2, C_3,C_4, C_1)$ is an empty square   contained in a copy  of $\mathcal{C}_{1137}$,  as shown  in Lemma    \ref{lemma:fnsC1137square};
		
		\item If there is a 3-simplex $\Delta$   in $\mathcal{T}^{(3)}$ such that $C_1, C_3$
		lie in the interior of  $\Delta$, then all vertices $C_i$ lie in  $\Delta^*$, and $(C_1, C_2, C_3,C_4, C_1)$ is an empty square   contained in a copy  of    $\mathcal{C}_{1163}$,  as shown  in Lemma     \ref{lemma:fnsC1163};
		
			\item  If there exists a 3-simplex \(\Delta\in\mathcal{T}^{(3)}\) such that \(C_1\) lies in the interior of \(\Delta\), while \(C_2\) and \(C_4\) lie on \(\partial\Delta\).
			Further, suppose that the points \(C_2\) and \(C_4\) do not lie on a common 2-face of \(\Delta\).
			Thus we suppose that \(C_2\) lies on a 2-face \(\delta_2\) of \(\Delta\) and \(C_4\) lies on a different 2-face \(\delta_4\) of \(\Delta\), with \(\delta_2 \neq \delta_4\), and neither \(C_2\) nor \(C_4\) belongs to the intersection \(\delta_2 \cap \delta_4\):

				\begin{itemize}
				
				\item Suppose there exists a 3-simplex $\Delta_1 \neq \Delta \in\mathcal{T}^{(3)}$ such that $C_3$ lies in the interior of $\Delta_1$.
				Consider the two edge pairs $\{[C_1,C_4],[C_4,C_3]\}$ and $\{[C_1,C_2],[C_2,C_3]\}$.
				Then $\Delta$ and $\Delta_1$ intersect in two distinct simplices.
				This contradicts the fact that $\mathcal{T}$ is a simplicial triangulation;

						\item  Suppose the vertex $C_3$ also lies in the interior of $\Delta$.
					Thus all vertices $C_i$ are contained in $\Delta$.
					The cycle $(C_1, C_2, C_3, C_4, C_1)$ forms an empty 2-square, and this square is contained in a copy of  $\mathcal{C}_{1163}$, as established in the Lemma \ref{lemma:fnsC1163};
					
					\item Otherwise, \(C_3\) lies in the 2-skeleton \(\mathcal{X}\) of \(\mathcal{T}\). There are still three subcases, depending on whether \(C_2\) lies in the interior of \(\delta_2\) and whether \(C_4\) lies in the interior of \(\delta_4\). But in either subcase, it can be seen that \(C_3\) lies in the boundary of the 3-simplex \(\Delta\).
					Thus all vertices \(C_i\) are contained in \(\Delta\).
					As established in Lemmas \ref{lemma:fnsC543}, \ref{lemma:fnsC1137square}, and \ref{lemma:fnsC1163}, there is no such empty \(2\)-square \((C_1, C_2, C_3, C_4, C_1)\) in \(\Delta^*\).

					\end{itemize}

		\item If there is a 3-simplex \(\Delta\) in \(\mathcal{T}^{(3)}\) such that \(C_1\) lies in the interior of \(\Delta\), and \(C_2\) and \(C_4\) lie in a common 2-face \(\delta\) of \(\Delta\), then there is no edge between \(C_2\) and \(C_4\). By Lemmas \ref{lemma:interiorC543} and \ref{lemma:halfsquare}, the substituted 3-complex $\Delta^*$ is isomorphic to one of \(\mathcal{C}'_{543}\), \(\mathcal{C}_{1137}\), and \(\mathcal{C}_{1163}\):
		
			\begin{itemize}
			
		\item  If the 3-complex \(\Delta^*\) is isomorphic to \(\mathcal{C}'_{543}\), then by Lemma  \ref{lemma:halfsquare}, \(C_2\) and \(C_4\) lie in the subcomplex \(c_{30}\) of \(\mathcal{C}'_{543}\), where \(c_{30}\) is the subdivision of the 2-face \(\delta\) of \(\Delta\).

			\begin{itemize}
			
			\item Suppose there exists a 3-simplex \(\Delta_1 \in \mathcal{T}^{(3)}\), \(\Delta_1 \neq \Delta\), such that \(C_3\) lies in the interior of \(\Delta_1\). Then by Lemma \ref{lemma:halfsquare}, the 3-complex \(\Delta_1^*\) is isomorphic to \(\mathcal{C}_{1163}\). Thus we obtain an empty square contained in two adjacent copies, one of \(\mathcal{C}'_{543}\) and the other of \(\mathcal{C}_{1163}\), glued along \(c_{30}\), but not entirely contained in a copy of \(\mathcal{C}_{1163}\), as stated in the theorem. We note here that there are two subcases in Item \ref{halfsquare:c} of Lemma \ref{lemma:halfsquare}, each of which yields some empty squares;

				\item Suppose \(C_3\) lies in the 2-skeleton \(\mathcal{X}\) of \(\mathcal{T}\). Then \(C_3\) lies in the boundary of \(\Delta\). By Lemma \ref{lemma:fnsCprime543}, no such empty square exists.

				\end{itemize}

		\item  So we now assume no vertex $C_i$ lies in the interior of 3-simplex $\Delta$, such that  the substituted 3-complex $\Delta^*$ is isomorphic to  \(\mathcal{C}'_{543}\).  Suppose the vertex  \(C_1\) lies in the interior of \(\Delta\), such that the substituted 3-complex $\Delta^*$ is isomorphic to  \(\mathcal{C}_{1163}\), then by Lemma  \ref{lemma:halfsquare}:
		
			\begin{itemize}

			\item Without loss of generality, we assume that \(\{C_1,C_2,C_4\}=\{w_{2,34},a_{23},a_{24}\}\). Now consider the vertex $C_3$. If $C_3$ lies in  the interior of 3-simplex $\Delta_1$, the  substituted 3-complex $\Delta^*_1$ is isomorphic to one of  \(\mathcal{C}_{1137}\), and \(\mathcal{C}_{1163}\).
			But in either case, there is no empty squares in the union of  $\Delta^*$  and $\Delta^*_1$. 
			If $C_3$   lies in the 2-skeleton \(\mathcal{X}\) of \(\mathcal{T}\).  Then \(C_3\) lies in the boundary of \(\Delta\). By Lemma  \ref{lemma:fnsC1163}, no such empty square exists.

		\item Or \(\{C_1,C_2,C_4\}=\{v_{1},G,F\}\subset\mathcal{C}_{30}\) for some pair of vertices \(G\) and \(F\) in the \(2\)-complex \(c_{30}\). As in the previous subcase, no empty square of this type exists.

			\end{itemize}

			\item  So we now assume that no vertex \(C_i\) lies in the interior of the 3-simplex \(\Delta\), and that the substituted 3-complex \(\Delta^*\) is isomorphic to either \(\mathcal{C}'_{543}\) or \(\mathcal{C}_{1163}\). Suppose that the vertex \(C_1\) lies in the interior of \(\Delta\), and that the substituted 3-complex \(\Delta^*\) is isomorphic to \(\mathcal{C}_{1137}\). By Item \ref{halfsquare:b} of Lemma \ref{lemma:halfsquare}, without loss of generality, we may assume that \(\{C_1,C_2,C_4\}\) is either \(\{v_1,a_2,a_3\}\) or \(\{v_1,a_2,a_{34}\}\). Now consider the vertex \(C_3\). If \(C_3\) lies in the interior of the 3-simplex \(\Delta_1\), then the substituted 3-complex \(\Delta^*_1\) is isomorphic to \(\mathcal{C}_{1137}\). But then \(\{C_1, C_2, C_4\}\) and \(\{C_3, C_2, C_4\}\) are two ``half squares'' in Lemma \ref{lemma:halfsquare}. If their union is an empty square, then \(\Delta \cap \Delta_1\) is an edge. There are such empty squares, as stated in Item \ref{emptysquare3} of the theorem. If \(C_3\) lies in the 2-skeleton \(\mathcal{X}\) of \(\mathcal{T}\), then \(C_3\) lies in the boundary of \(\Delta\). By Lemma \ref{lemma:fnsC1137square}, no such empty square exists.

				\end{itemize}

			\item Finally, suppose all the vertices $C_i$ lie in $\mathcal{X}$. Then by Lemma~\ref{lemma:2skeleton}, we obtain a contradiction, so this case cannot occur.

			\end{itemize}

		In total,  the only empty squares in $\mathcal{T}^*$ are  those  contained in the copies of $\mathcal{C}_{1137}$ and    $\mathcal{C}_{1163}$,  as shown  in Lemmas    \ref{lemma:fnsC1137square}   and  \ref{lemma:fnsC1163}.

\end{proof}

	 \begin{remark} \label{remark:reprove}
	 	Our proof of Theorem \ref{thm:almostfns} also yields an alternative proof of Lemma 5.2 of \cite{CDSStark:2025}, although applying the proof of Theorem \ref{thm:almostfns} to obtain Lemma 5.2  of \cite{CDSStark:2025} is overkill. Nevertheless, our proof differs slightly from that in \cite{CDSStark:2025}.
	 	   We need the   Sparsity III condition in the triangulation $t$ of the surface $\Sigma_g$, so we may replace each 2-simplex $\delta \in t$ by a copy of $c_3$ or $c_{10}$ to obtain a flag-no-square triangulation of $\Sigma_g$. Without the Sparsity III condition, the subdivision may not be a flag-no-triangulation, see Figure 	\ref{figure:c3}.
	
\end{remark}

\begin{figure}
	\begin{center}
		\begin{tikzpicture}
		\node at (0,0) {\includegraphics[width=11cm,height=4.9cm]{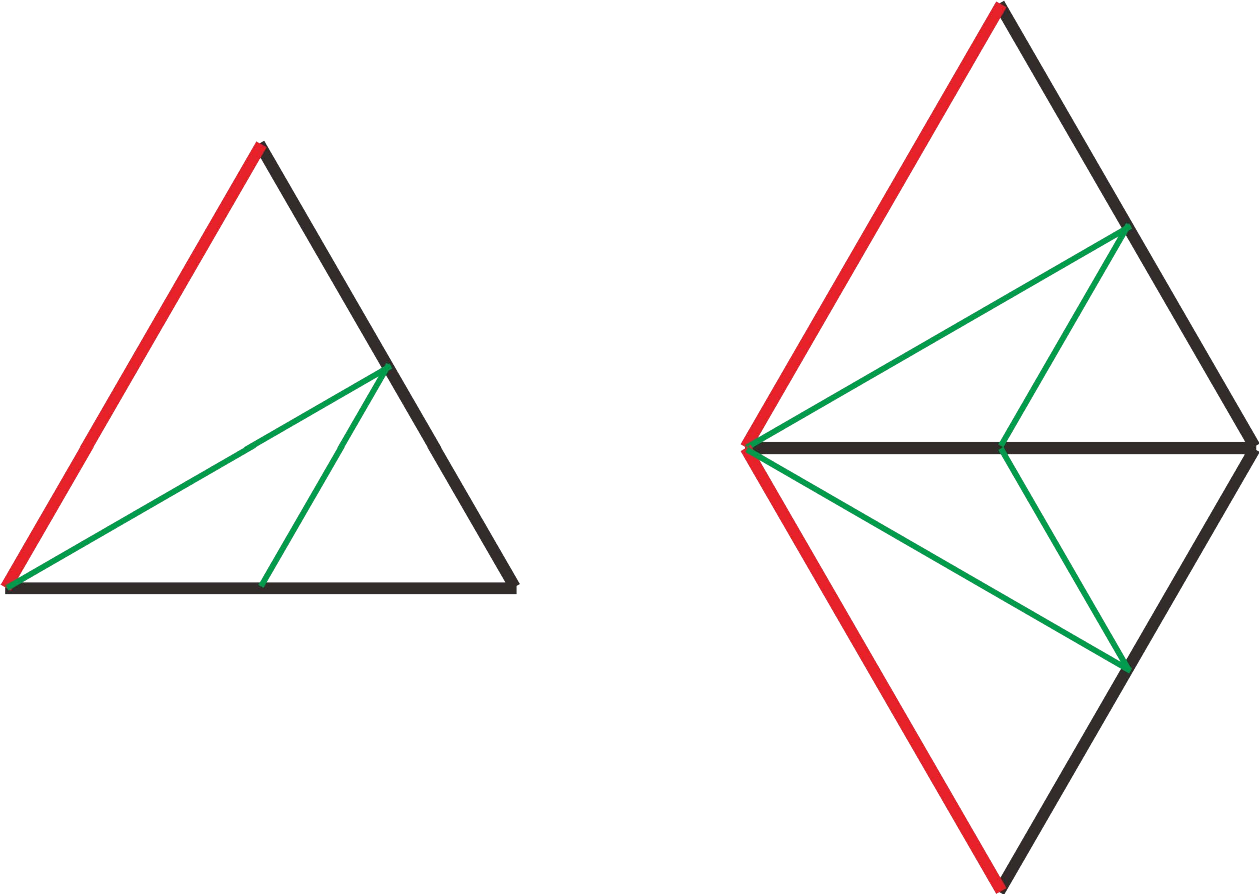}};
		\node at (-2.9,2.2){\small $A_i$};
		\node at (-5.6,-1.3){\small $B_k$};
		\node at (-1.1,-1.3){\small $D$};

		\node at (2.6,2.3){\small $A_i$};
		
		\node at (0.5,0.4){\small $B_k$};
		
		\node at (2.6,-2.3){\small $A_j$};

		\node at (5.7,-0.11){\small $D$};
		\node at (4.7,1.31){\small $C_1$};
		\node at (4.7,-1.31){\small $C_2$};
		
		\end{tikzpicture}
	\end{center}
	\caption{The left subfigure is the triangulation  $c_3$ of a 2-disk, by adding two green arcs to get a  subdivision of a 2-simplex.  If there is a poison edge in $\mathcal{T}^{(1)}$ as described in Lemma \ref{lemma:inducedgraph}, then the 2-skeleton of the special subdivision $\mathcal{T}^*$ is the not flag-no-square, as the right subfigure shows: triangles with vertex sets $\{B_k, A_i, D\}$ and $\{B_k, A_j, D\}$ are adjacent in $\mathcal{T}$.   The edges $[B_k,A_i]$ and $[B_k,A_j]$ lie in the graph $\Gamma$, then $[B_k,D]$ is a poison edge in $\mathcal{T}$. We have   $[B_k, C_1, D, C_2, B_k]$ is an empty square in the special subdivision (by adding green arcs).}
	\label{figure:c3}
\end{figure}

	\section{Assigning angles to edges of $3$-complexes}   \label{section:conformaldim}

In this section, by assigning angles on the edges of \(\mathcal{T}^*\) in Theorem \ref{thm:almostfns}, we prove Theorem \ref{thm:conformaldim}. We note that the empty squares in Theorem \ref{thm:almostfns} impose no real restriction, since we only require that at least one of the angles be smaller than \(\pi/2\). The empty triangles, however, are very restrictive: the sum of the three
angles on an empty triangle must be smaller than \(\pi\), forcing the
angles to be relatively small. On the other hand, a \(2\)-simplex may have
two edges lying in two different empty triangles, and this \(2\)-simplex
is required to yield a spherical Coxeter group, forcing the angles to be
sufficiently large. These two requirements are somewhat contradictory,
and hence the possible angle assignments are very limited on some edges.
This occurs on some \(2\)-simplices in \(\mathcal{C}_{1137}\), see
Figure~\ref{figure:angle}.


	\textbf{Proof of  Theorem \ref{thm:conformaldim}.}  
 For a closed orientable 3-manifold $N$, and a   bipartite,  flag-no-square graph $\Gamma$, we have embeddings  $\Gamma \hookrightarrow \Sigma_g \hookrightarrow N $ with a triangulation $\mathcal{T}$ of $N$ as   in  Lemma \ref{lemma:inducedgraph}. Let $\mathcal{T}^*$ be the relatively special subdivision associated with $(\mathcal{T}, \Gamma)$ as in  Theorem \ref{thm:almostfns}. We now  assign angles $$\mathcal{A}=\{A_{e}\}_{e \in (\mathcal{T}^*)^{(1)}}$$ on the edges of $\mathcal{T}^*$ to get a hyperbolic  Coxeter group. We assign the angle $\pi/2$ for each edge $e$ of  $\mathcal{T}^*$ except:



 \textbf{Angle assignment for \(\mathcal{C}_{1163}\):}
 For a \(3\)-simplex \(\Delta \in \mathcal{T}^{(3)}\), if its substituted \(3\)-complex \(\Delta^*\) in \(\mathcal{T}^*\) is isomorphic to
 \[
 \mathcal{C}_{1163} = \mathcal{C}_{30} \cup_{c_6} \mathcal{C}_{590} \cup_{c_{10}} \overline{\mathcal{C}}_{543},
 \]
 we assign the angle \(\pi/p\) for any \(p \in \mathbb{Z}_{\geq 3}\) to the edges \([v_1, a_{23}]\), \([v_1, a_{24}]\), and \([v_1, a_{34}]\) in \(c_6=C_{30}\cap \mathcal{C}_{590} \subset \mathcal{C}_{1163}\).   This condition is intended to ensure  that the empty squares in copies of
 \(\mathcal{C}_{1163}\) in Item~\ref{emptysquare1} in
 Theorem~\ref{thm:almostfns} are hyperbolic quadrilateral groups.


\textbf{Angle assignment for \(\mathcal{C}'_{543}\):} For a \(3\)-simplex \(\Delta \in \mathcal{T}^{(3)}\), if its substituted \(3\)-complex \(\Delta^*\) in \(\mathcal{T}^*\) is isomorphic to \(\mathcal{C}'_{543}\):

	
	\begin{enumerate}
	\item \label{assign:543prime1}
	
	 We assign the angle \(\pi/3\) to certain edges of the form
$$F_{i_3} \textbf{---} F_{i_4} \quad \text{or} \quad F_{i_4} \textbf{---} F_{i_5},$$
as specified in \eqref{eq:five-circles1}, \eqref{eq:five-circles2}, and
\eqref{eq:five-circles3} in the proof of Item~\ref{halfsquare:a} of
Lemma~\ref{lemma:halfsquare}.
This condition is intended to ensure that the empty squares in Item \ref{emptysquare2} in Theorem \ref{thm:almostfns} are hyperbolic quadrilateral groups;

\item \label{assign:543prime2} Note that there is an adjacent \(3\)-simplex
\(\Delta_1 \in \mathcal{T}^{(3)}\) of \(\Delta\) such that
\(\Delta \cap \Delta_1 = \delta\) is a \(2\)-simplex. The substituted
\(3\)-complex \(\Delta^*_1\) in \(\mathcal{T}^*\) is isomorphic to
\(\mathcal{C}_{1137}\), and the induced subdivision of \(\delta\) is
isomorphic to \(c_4\). We assign \(\pi/4\) to the interior edge in
\(c_4\) opposite \(B\), and \(\pi/3\) and \(\pi/2\) to the remaining
interior edges in \(c_4\). Note that there is an ambiguity here as to which edge is assigned
\(\pi/3\) and which is assigned \(\pi/2\). However, since the angles assigned here are used only to match those in
Item~ (\ref{assign:11371}) in ``Angle assignment for \(\mathcal{C}_{1137}\)''
below, they will be uniquely determined later.

	\end{enumerate}


\textbf{Angle assignment for \(\mathcal{C}_{1137}\):} For a 3-simplex  $\Delta  \in \mathcal{T}^{(3)}$, if its substituted 3-complex \(\Delta^*\) in \(\mathcal{T}^*\) is isomorphic to  $$\mathcal{C}_{1137}=\mathcal{C}_{4}\cup_{c_6}\mathcal{C}_{590} \cup_{c_{10}} \overline{\mathcal{C}}_{543}:$$
	
		
			\begin{enumerate}
		\item \label{assign:11371} For the 2-complex  $c_4 \subset \mathcal{C}_{4} \subset \mathcal{C}_{1137}$, we assign the angle \(\pi/3\) to the edge \([a_{23},a_{24}]\), the angle \(\pi/4\) to the edge \([a_{24},a_{34}]\), and the angle \(\pi/2\) to the edge \([a_{34}, a_{23}]\), see Figure~\ref{figure:angle}.
		This condition is intended to ensure that the empty triangles in
		Item~\ref{emptytriangle} in Theorem~\ref{thm:almostfns} are hyperbolic
		triangle groups, but it does not by itself guarantee this.
 We note that the triangle with vertex set \(\{a_{23},a_{24},a_{34}\}\) is a spherical Coxeter group, which is a necessary condition. 
  Moreover, this 3-simplex \(\Delta \in \mathcal{T}^{(3)}\) is adjacent to a
 \(3\)-simplex \(\Delta_1 \in \mathcal{T}^{(3)}\) such that the
 substituted \(3\)-complex \(\Delta^*_1\) in \(\mathcal{T}^*\) is
 isomorphic to \(\mathcal{C}'_{543}\). Using the angle assignment here,
 we assign angles to the edges of \(\Delta^* \cap \Delta^*_1\) as specified in
 Item~(\ref{assign:543prime2}) in ``Angle assignment for \(\mathcal{C}'_{543}\)'', so that the assigned angles on
 \(\Delta^* \cap \Delta^*_1\) are coherent;

		  	\item	\label{assign:11375} For the $2$-complex \(c_{6}=\mathcal{C}_{4} \cap \mathcal{C}_{590}\), we assign an angle of \(\pi/3\) to the edges \([v_{1},a_{2}]\), \([v_{1},a_{3}]\), and 
	\([v_{1},a_{4}]\).
		This choice of angles is intended to ensure that certain empty squares in Item~\ref{emptysquare3} of Theorem~\ref{thm:almostfns} correspond to hyperbolic quadrilateral groups;

			\item \label{assign:11372} For the 2-complex  \(c_{10}=\mathcal{C}_{590} \cap \overline{\mathcal{C}}_{543}\), we assign to the edges \([a_{23},w_{2,34}]\), \([w_{2,34},a_{24}]\), \([a_{24},w_{4,23}]\), \([w_{4,23},a_{34}]\), \([a_{34},w_{3,24}]\), and \([w_{3,24},a_{23}]\) the angles \(\pi/3\), \(\pi/4\), \(\pi/3\), \(\pi/3\), \(\pi/5\), and \(\pi/4\), respectively; see Figure~\ref{figure:angle}.
		   This condition is intended to ensure  that the empty triangles in Item \ref{emptytriangle} in Theorem \ref{thm:almostfns}  are hyperbolic triangle groups;

		  	\item   \label{assign:11376} We assign the angle \(\pi/p\)  for any $p \in \mathbb{Z}_{\geq 3}$  to the the edge  $[\overline{m}_{1}, \overline{n}_{1}]$  in $\overline{\mathcal{C}}_{543} \subset \mathcal{C}_{1137}$  in Item~\ref{fnsC1137square23} in Lemma \ref{lemma:fnsC1137square}. 	  This condition is intended to ensure  that  some empty squares  in Item \ref{emptysquare1} in Theorem \ref{thm:almostfns} are hyperbolic quadrilateral groups;

		\item \label{assign:11373}  We assign the angle \(\pi/p\)  for any $p \in \mathbb{Z}_{\geq 3}$  to the edge  $[m_{1}, n_{1}]$  in $\mathcal{C}_{590} \subset \mathcal{C}_{1137}$ in Item~\ref{fnsC1137square33} in Lemma \ref{lemma:fnsC1137square}.    This condition is intended to ensure  that some empty squares in Item \ref{emptysquare1} in Theorem \ref{thm:almostfns} are hyperbolic quadrilateral groups;

		\item  \label{assign:11374} We assign the angle \(\pi/p\)  for any $p \in \mathbb{Z}_{\geq 3}$  to  the edge  $[v_2, v_4]$  in $\mathcal{C}_{4} \cup \mathcal{C}_{590} \subset \mathcal{C}_{1137}$  in Item~\ref{fnsC1137square36} in Lemma \ref{lemma:fnsC1137square}.
		 This condition is intended to ensure   that some  empty squares in Item \ref{emptysquare1} in Theorem \ref{thm:almostfns} are hyperbolic quadrilateral groups.

	\end{enumerate}


We note that all angles other than \(\pi/2\) occur on the interior
edges of each \(3\)-complex \(\Delta^*\). The only exceptions are those
occurring in the interior of the \(2\)-simplex
\(\delta = \Delta_1 \cap \Delta_2\). For these, the substituted
\(3\)-complexes \(\Delta^*_1\) and \(\Delta^*_2\) are isomorphic to
\(\mathcal{C}'_{543}\) and \(\mathcal{C}_{1137}\), respectively. By
Item~(\ref{assign:543prime2}) in  ``Angle assignment for \(\mathcal{C}'_{543}\)'' and Item~(\ref{assign:11371}) in  ``Angle assignment for \(\mathcal{C}_{1137}\)'', we assign angles to the edges of
\(\mathcal{T}^*\) coherently.  In particular, since all the edges in the graph \(\Gamma\) lie in copies of \(\mathcal{C}'_{543}\), they are all assigned an angle of \(\pi/2\).

\begin{figure}
	\begin{center}
		\begin{tikzpicture}
		\node at (0,0) {\includegraphics[width=9.5cm,height=8cm]{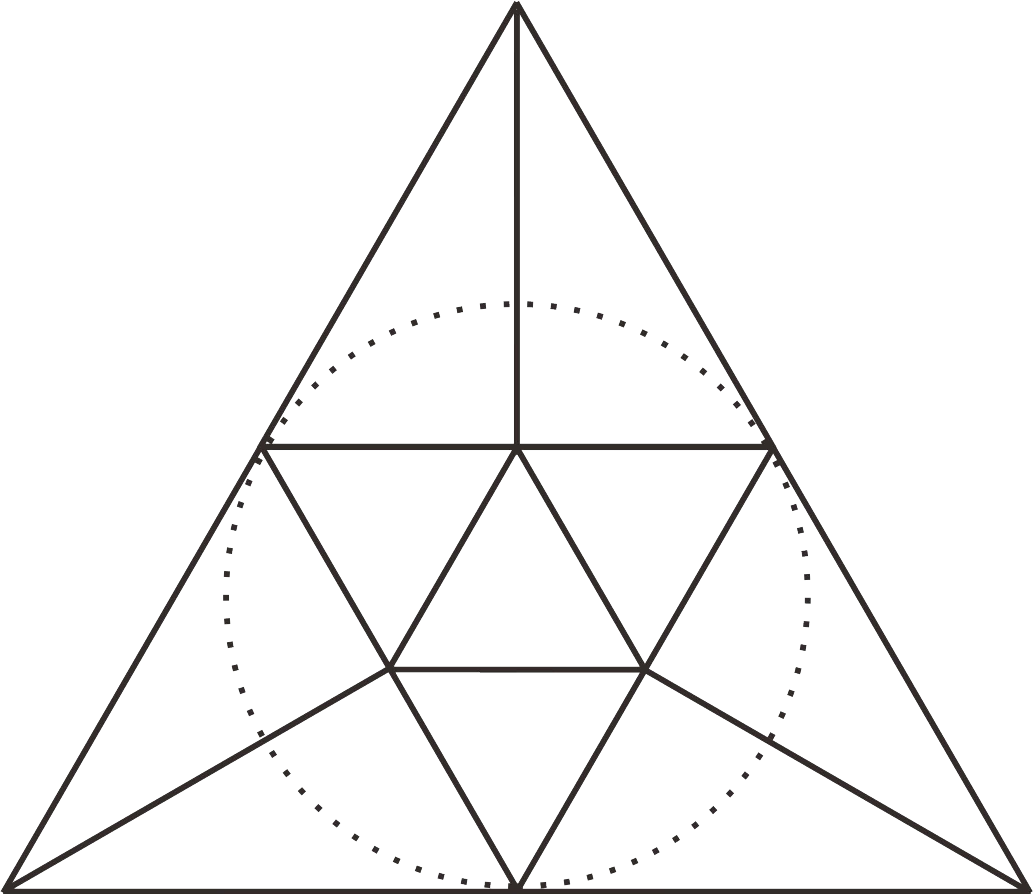}};
		
		\node at (0.5,3.8){\large $a_{2}$};
		
		\node at (-5.1,-3.8){\large $a_{4}$};
		
		\node at (5.1,-3.8){\large $a_{3}$};
		
		\node at (3.1,-0.1){\large $a_{23}$};
		
		\node at (-3.1,-0.1){\large $a_{24}$};
		
		\node at (0.1,-4.3){\large $a_{34}$};

		\node at (0.05,-0.75){\tiny $w_{2,34}$};
		
		\node at (-0.55,-1.7){\tiny $w_{4,23}$};
		
		\node at (0.55,-1.7){\tiny $w_{3,24}$};

		\node at (1.99,-3.4){\large $2$};

		\node at (-0.7,1.5){\large $3$};
		
		\node at (-2.9,-1.8){\large $4$};

		\node at (1.2,0.3){\large $3$};
		
		\node at (-1.2,0.3){\large $4$};
		
		\node at (-1.8,-1.4){\large $3$};	
		
		\node at (-0.9,-2.9){\large $3$};	
		
		\node at (0.98,-2.9){\large $5$};
		
		\node at (2.1,-0.9){\large $4$};
		
		\end{tikzpicture}
	\end{center}
	\caption{The angles near $c_{10}=\mathcal{C}_{590} \cap  \overline{\mathcal{C}}_{543} \subset \mathcal{C}_{1137}$. Here for example, the label $3$ for the edge $[a_{23},a_{24}]$ means the angle is $\pi/3$. The angle on each non-labeled edge is $\pi/2$. }
	\label{figure:angle}
\end{figure}

Then we have,



\textbf{Finite groups for \(2\)- and \(3\)-simplices:}

\begin{enumerate}

	\item For a triangle \(\delta\) with vertex set \(\{C_i,C_j,C_k\}\) in \(\mathcal{T}^*\), suppose \(\delta\) is a 2-face of \(\Delta \in \mathcal{T}^{*}\). 
	From the four isomorphic types of the subdivision \(\Delta^*\), it is easy to see that there is at most one edge of \(\delta\) with angle smaller than \(\pi/2\), so there is a spherical Coxeter group \(W_{\delta}\) associated with \(\delta\). 
	The only exceptions are certain triangles in \(c_{10}=\mathcal{C}_{590} \cap \overline{\mathcal{C}}_{543}\); see Figure~\ref{figure:angle}. 
	However, the angles we assigned to these triangles force them to be spherical Coxeter triangles.  For example, consider the triangles with vertex sets
	\(\{v_1,a_{23}, w_{3,24}\}\), \(\{v_1,a_{34}, w_{3,24}\}\), and \(\{a_{34},w_{4,23}, w_{3,24}\}\).
	As we noted at the beginning of this section, the assigned angles here are the trickiest ones;

		\item Let \(\Delta_1\) be a \(3\)-simplex in \(\mathcal{T}\), and let \(\Delta\) be a \(3\)-simplex with vertex set \(\{C_i,C_j,C_k,C_l\}\) in the subdivision \(\Delta^*_1 \subset \mathcal{T}^*\). 
		Since the subdivision \(\Delta_1^*\) has only four isomorphic types, at most one edge of \(\Delta\) has angle smaller than \(\pi/2\); hence there is a spherical Coxeter group \(W_{\Delta}\) associated with \(\Delta\).
		The only exceptions are:

		\begin{itemize}

			\item
			
			The  \(3\)-complex \(\Delta^*_1\)  is isomorphic to  \(\mathcal{C}'_{543}\).
			Assume that \(E\) is a vertex of \(\Delta\) at which two edges of
			\(\Delta\) whose angles are different from \(\pi/2\) meet. The vertex
			\(E\) either comes from Item~(\ref{assign:543prime1}) or from
			Item~(\ref{assign:543prime2}) in ``Angle assignment for
			\(\mathcal{C}'_{543}\)''.
			
			In the first subcase, for all of the fifteen half-squares
			\[
			F_{i_3} \text{ --- } F_{i_4} \text{ --- } F_{i_5}
			\]
			in (\ref{eq:five-circles1}), (\ref{eq:five-circles2}) and (\ref{eq:five-circles3}) appearing in the proof of Item~\ref{halfsquare:a} in Lemma~\ref{lemma:halfsquare},
			there are exactly six vertices, namely
			\[
			\{F_{14},F_{16},F_{18},F_{20},F_{26},F_{28}\},
			\]
			such that two edges at each of these vertices are assigned the angle
			\(\pi/3\), while every other vertex has at most one incident edge
			assigned the angle \(\pi/3\). Hence \(E\) is one of the vertices above.
			A \(3\)-simplex \(\Delta\) with two edges having angle \(\pi/3\) and
			all other edges having angle \(\pi/2\) yields a spherical Coxeter
			group. Thus \(W_{\Delta}\) is the spherical Coxeter group associated
			with \(\Delta\).
			In the second subcase, if \(E\) is adjacent to an edge in the copy of $c_4$ whose angle is
				\(\pi/4\), then \(E\) cannot be any of the vertices
				\(\{F_{14},F_{16},F_{18},F_{20},F_{26},F_{28}\}\) above.
					Any \(3\)-simplex \(\Delta\) in \(\mathcal{C}'_{543}\) having \(E\) as a
				vertex has at most two edges whose angles are different from
				\(\pi/2\), and its smallest angle is \(\pi/4\). Hence, with the
				assigned angles, \(\Delta\) is a spherical Coxeter polytope;
				We now assume that \(E\) is adjacent to an edge in the copy of $c_4$ whose angle is
		\(\pi/3\). 	Any \(3\)-simplex \(\Delta\) in \(\mathcal{C}'_{543}\) having \(E\) as a
		vertex has at most three edges whose angles are different from
		\(\pi/2\), and its smallest angle is \(\pi/3\). Hence, with the
		assigned angles, \(\Delta\) is a spherical Coxeter polytope;

	\item  The \(3\)-complex \(\Delta^*_1\) is isomorphic to \(\mathcal{C}_{1137}\). We need to consider the \(3\)-simplex \(\Delta\) appearing in the angle assignments (\ref{assign:11371}), (\ref{assign:11375}), and (\ref{assign:11372}) of  ``Angle assignment for \(\mathcal{C}_{1137}\)''.
All four \(3\)-simplexes in \(C_4 \subset \mathcal{C}_{1137}\), with the assigned angles, are spherical Coxeter polytopes. Moreover, f
or each \(3\)-simplex containing a \(2\)-simplex in \(c_{10}=\mathcal{C}_{590} \cap \mathcal{C}'_{543} \subset \mathcal{C}_{1137}\), see Figure~\ref{figure:angle}. By the well-known classification of spherical Coxeter polytopes, see, for example, Figure~1 of~\cite{Ma-Zheng:2025}, each such labeled \(3\)-simplex is a spherical Coxeter polytope.

	
\end{itemize}

	\end{enumerate}


 \textbf{Infinite groups for empty triangles and squares:}
	\begin{enumerate}
\item  For the empty triangles in Item~\ref{emptytriangle} in Theorem~\ref{thm:almostfns} in copies of \(\mathcal{C}_{1137}\), with vertex sets \(\{a_{23},a_{24},w_{2,34}\}\), \(\{a_{23},a_{34},w_{3,24}\}\), and \(\{a_{24},a_{34},w_{4,23}\}\), by Items~(\ref{assign:11371}) and~(\ref{assign:11372}) in ``Angle assignment for \(\mathcal{C}_{1137}\)'' we have an infinite hyperbolic group \(G\) for each of them; moreover, \(G\) is virtually the fundamental group of a closed surface of genus \(2\);

	\item For the empty square in Item~\ref{emptysquare1} in Theorem~\ref{thm:almostfns}, if the empty square lies in a copy of \(\mathcal{C}_{1137}\), it is in one of the twelve forms in Lemma~\ref{lemma:fnsC1137square}; by Items~(\ref{assign:11371}), (\ref{assign:11373}), (\ref{assign:11374}), and~(\ref{assign:11376})  in ``Angle assignment for \(\mathcal{C}_{1137}\)'', at least one of the edges is assigned the angle \(\pi/q\) for some \(q \geq 3\); we obtain an infinite hyperbolic group \(G\) for each of them; moreover, each such \(G\) is virtually the fundamental group of a closed surface of genus \(2\).
	If the empty square lies in a copy of \(\mathcal{C}_{1163}\), it is in one of the three forms in Lemma~\ref{lemma:squareC590}; by  ``Angle assignment for \(\mathcal{C}_{1163}\)'', two of the edges are assigned the angle \(\pi/p\) for some \(p \geq 3\); we obtain an infinite hyperbolic group \(G\) for each of them;

		\item  For the empty squares \((C_1,C_2,C_3,C_4,C_1)\) in Item~\ref{emptysquare2} in Theorem~\ref{thm:almostfns}, which lies in  in two adjacent copies of \(\mathcal{C}'_{543}\) and \(\mathcal{C}_{1163}\). Recall that,   up to reordering, we may assume \((C_4,C_1,C_2)\) lies in \(\mathcal{C}'_{543}\), and it is in one of the fifteen forms in Item~\ref{halfsquare:a} in Lemma~\ref{lemma:halfsquare}.
		By Item~(\ref{assign:543prime1}) in ``Angle assignment for \(\mathcal{C}'_{543}\)'', exactly one edge of the empty square is assigned the angle \(\pi/3\); we obtain an infinite hyperbolic group \(G\) for each of them;

			\item  For the empty squares \((C_1,C_2,C_3,C_4,C_1)\) in Item~\ref{emptysquare3} in Theorem~\ref{thm:almostfns}, as in Lemma~\ref{lemma:halfsquare}, we may assume there are two \(3\)-simplices \(\Delta_1\) and \(\Delta_2\) in \(\mathcal{T}\) such that \(\Delta_1 \cap \Delta_2 = e\) is an edge, and the substituted complexes \(\Delta^*_1\) and \(\Delta^*_2\) are both isomorphic to \(\mathcal{C}_{1137}\), and
			$$
			(C_1,C_2,C_3,C_4,C_1)=(a_k,v_{1,1},a_l,v_{1,2},a_k),$$
			where \(a_k, a_l\) are the endpoints of \(e\), and \(v_{1,1}\) and \(v_{1,2}\) are the copies of the vertex \(v_1\) in \(\Delta^*_1\) and \(\Delta^*_2\), respectively, under the isomorphisms with \(\mathcal{C}_{1137}\).
			By Item~(\ref{assign:11375}) in ``Angle assignment for \(\mathcal{C}_{1137}\)'', at least  two edges of the empty square are assigned the angle \(\pi/3\); we obtain an infinite hyperbolic group \(G\) for each of them.


\end{enumerate}


Thus, for a \(3\)-complex \(\mathcal{T}^*\) in Theorem~\ref{thm:almostfns}, the assigned angles \(\mathcal{A}\) above on the edges of \(\mathcal{T}^*\) determine a Coxeter system \((\mathcal{T}^*,\mathcal{A})\) with nerve exactly $\mathcal{T}^*$.

Fix a closed orientable \(3\)-manifold \(N\). We have a sequence of graphs \(\Gamma_k\) in Lemma~\ref{thm:highconformaldim}. Taking \(\Gamma_k\) as the nerve, we obtain a sequence of hyperbolic, right-angled Coxeter groups \(H_k\) such that the conformal dimension of the boundary of \(H_k\) diverges to infinity.
For the sequence \(\{\Gamma_k\}^{\infty}_{k=1}\) and \(N\), by Lemma~\ref{lemma:inducedgraph} and Theorem~\ref{thm:almostfns}, there is a Coxeter system \((\mathcal{T}^*_k,\mathcal{A}_k)\) such that the associated Coxeter group \(W_k\) is hyperbolic with nerve \(\mathcal{T}^*_k\) by~\cite{Moussong:1988}. Moreover, the right-angled Coxeter group \(H_n\) with nerve \(\Gamma_k\) is a quasi-convex subgroup whose boundary is a Menger curve. Therefore, the conformal dimension of the boundary of \(W_k\) is bounded below by that of the boundary of \(H_k\). Theorem~\ref{thm:conformaldim} follows. \qed


\section{Proof of Theorem  \ref{thm:infiniteqi}} \label{section:infiniteqi}

In this section, we prove Theorem~\ref{thm:infiniteqi}. To this end, we first recall some material from \cite{Davis:1983}.

\begin{defn} A \emph{\(4\)-manifold with corners} is a Hausdorff space equipped with an atlas of local charts onto open subsets of
	\[
	\mathbb{R}_{+}^{4}=\{x=(x_1, x_2, x_3, x_4) \in \mathbb{R}^{4} \mid x_{i} \geq 0\},
	\]
	such that the overlaps are homeomorphisms preserving codimension. Here, for \(x=(x_1, x_2, x_3, x_4) \in \mathbb{R}_{+}^{4}\), the codimension \(c(x)\) is the number of coordinates \(x_i\) that are zero \cite{Davis:1983}.

\end{defn}   

For a \(4\)-manifold \(M\) with corners, an open pre-face of \(M\) of codimension \(i\) is a connected component of \(c^{-1}(i)\), and the closure of a pre-face is called a face of \(M\). In our case, we assume that each face of codimension \(2\) is adjacent to exactly two codimension \(1\) faces.



\textbf{Proof of Theorem \ref{thm:infiniteqi}.}
Let \(N\) be an integral homology \(3\)-sphere. By a theorem of Freedman~\cite{Freedman:1982}, \(N\) bounds a contractible compact \(4\)-manifold \(M\).
Let \(\mathcal{T}^*\) be the almost flag-no-square triangulation of \(N\) in Theorem~\ref{thm:almostfns}. The dual of \(\mathcal{T}^*\) in \(N\) gives a cell decomposition \(\mathcal{D}\) of \(N\), each \(3\)-cell of $\mathcal{D}$ is a \(3\)-dimensional polytope. Thus, together with the cell structure on \(\partial M=N\), the decomposition \(\mathcal{D}\) endows \((M,\partial M)\) with a 4-dimensional corner manifold structure.

Since \(\mathcal{T}^*\) is a simplicial triangulation, we may view \((M, \mathcal{D})\) as a ``simple polytope''. Together with the angle structure \(\mathcal{A}\) in the proof of Theorem~\ref{thm:conformaldim}, we may then view \((M, \mathcal{D}, \mathcal{A})\) as a ``Coxeter polytope''. For example, if \(N\) is homeomorphic to the 3-sphere, then \(M\) is homeomorphic to the 4-ball, and \((M, \mathcal{D}, \mathcal{A})\) is a genuine Coxeter polytope.

For the infinite hyperbolic Coxeter group associated to the Coxeter system \((\mathcal{T}^*, \mathcal{A})\), by Theorem 15.3 of \cite{Davis:1983}, infinitely many copies of \(M\) are glued together to an open contractible \(4\)-manifold \(\widetilde{M}\) according to the Coxeter group \(W_{(\mathcal{T}^*, \mathcal{A})}\). The \(4\)-manifold \(\widetilde{M}\) is not homeomorphic to \(\mathbb{R}^4\) if and only if \(N\) is not homeomorphic to the 3-sphere \(\mathbb{S}^3\). We note that the proof in \cite{Davis:1983} does not use the right-angled condition, even though that paper is concerned with the universal cover of the \(4\)-orbifold associated with a right-angled ``Coxeter polytope''. In fact, the same proof applies to any ``Coxeter polytope'', and hence to the groups considered here.

The Coxeter group \(W_{(\mathcal{T}^*, \mathcal{A})}\) acts on \(\widetilde{M}\) properly discontinuously and cocompactly. Let \(G\) be a finite-index torsion-free subgroup of \(W_{(\mathcal{T}^*, \mathcal{A})}\). Then \(G\) acts on \(\widetilde{M}\) freely and properly discontinuously, and we obtain the quotient space \(M^{*}\).
 The space \(M^{*}\) is aspherical with Gromov hyperbolic fundamental group, so it  is a topological negatively curved \(4\)-manifold. Thus in this case, the Coxeter group \(W_{(\mathcal{T}^*, \mathcal{A})}\) is virtually a \(PD_{4}\) hyperbolic group. See also \cite[page 431]{Fischer:2003}. Moreover, by \cite{PrzytyckiS:2009, Swiatkowski:2020}, the boundary at infinity of \(W_{(\mathcal{T}^*, \mathcal{A})}\) is the tree of manifolds \(\chi(N \sharp \overline{N})\), which depends on \(N\); and the tree of manifolds is not a \(3\)-manifold if and only if \(N\) is not homeomorphic to the 3-sphere \(\mathbb{S}^3\).

Let \(N\) be an integral homology \(3\)-sphere. Given a sequence of graphs \(\Gamma_k\) as in Lemma \ref{thm:highconformaldim}, we have the triangulations \(\mathcal{T}_k\) as in Lemma \ref{lemma:inducedgraph}, and then the triangulations \(\mathcal{T}_k^*\) of \(N\) as in Theorem~\ref{thm:almostfns}. Consequently, we obtain hyperbolic Coxeter groups \(W_{(\mathcal{T}^*_k, \mathcal{A}_k)}\) and topological negatively curved \(4\)-manifolds \(M^{*}_k\) as above.

With nerve \(\Gamma_k\), we have a sequence of hyperbolic right-angled Coxeter groups \(H_k\) such that the conformal dimension of the boundary of \(H_k\) diverges to infinity. From the proof of Theorem \ref{thm:almostfns}, we also have a simplicial triangulation \(t^*_k\) of \(\Sigma_g\) that is a full subcomplex of \(\mathcal{T}_k^*\). Let \(K_k\) be the hyperbolic Coxeter group whose nerve is the triangulation \(t^*_k\) of \(\Sigma_g\), with angles given by the restriction of \(\mathcal{A}_k\). For each \(k\), since \(\pi_{1}(M^*_k)\) is a finite-index subgroup of \(W_{(\mathcal{T}^*_k, \mathcal{A}_k)}\), the groups \(H_k \cap \pi_{1}(M^*_k)\) and \(K_k \cap \pi_{1}(M^*_k)\) are commensurable with the groups \(H_k\) and \(K_k\), respectively. Thus the groups \(H_k \cap \pi_{1}(M^*_k)\) and \(K_k \cap \pi_{1}(M^*_k)\) are quasiconvex hyperbolic subgroups of \(\pi_{1}(M^*_k)\), where the boundary of \(K_k \cap \pi_{1}(M^*_k)\) is the Pontryagin sphere and the boundary of \(H_k \cap \pi_{1}(M^*_k)\) is the Menger curve. Theorem \ref{thm:infiniteqi} follows from Theorem \ref{thm:conformaldim}.

  \qed

\bibliographystyle{amsplain}

\end{document}